\documentclass{amsart}

    \usepackage[utf8]{inputenc}
    \usepackage{amsfonts}
    \usepackage{amsmath}
    \usepackage{amsthm}
    \usepackage{stmaryrd} 
    \usepackage{bm} 
    \usepackage{enumitem} 

    \usepackage[style=alphabetic,maxbibnames=99]{biblatex}
        \DeclareFieldFormat{title}{\mkbibquote{#1}}
    
    \usepackage{hyperref}
    \usepackage[noabbrev,nameinlink,capitalise]{cleveref}
        \crefname{subsection}{Subsection}{Subsections}
        \crefname{subsection}{Subsection}{Subsections}
    
    \usepackage{subcaption}
    \usepackage{color}
    \usepackage[dvipsnames]{xcolor}
    \usepackage{graphicx}\graphicspath{{./images/}}
    \usepackage{tikz}\tikzset{every picture/.style=semithick}
    \usetikzlibrary{
        shapes,arrows,
        decorations.markings,
        decorations.pathreplacing}
    \tikzset{->-/.style={decoration={
        markings,
        mark=at position #1 with {\arrow{angle 60}}},postaction={decorate}}}

    \theoremstyle{plain}
        \newtheorem{theorem}{Theorem}[section]
        \newtheorem{corollary}[theorem]{Corollary}
        \newtheorem{lemma}[theorem]{Lemma}
        \newtheorem{proposition}[theorem]{Proposition}
    \theoremstyle{definition}
        \newtheorem{definition}[theorem]{Definition}
        
    \theoremstyle{remark}
        \newtheorem{remark}[theorem]{Remark}

    \title{Conjugacy in Rearrangement Groups of Fractals}
    \author{Matteo Tarocchi}
    \date{}
    
    \thanks
    {The author would like to thank James Belk, Miguel Del R\'io, Francesco Matucci and Davide Perego for useful discussions and advice and Matthew Zaremsky for conversations on Houghton groups.
    The author is also grateful to James Belk and Bradley Forrest for kindly providing the images of the Airplane, Basilica, Vicsek and Bubble Bath limit spaces depicted in \cref{fig_A_limit_space} and \cref{fig_limit_spaces} from their work \cite{belk2016rearrangement}.
    The image of the Airplane limit space was also colored by the author to produce \cref{fig_cells_A} and \cref{fig_action_alpha}.
    The author is a member of the Gruppo Nazionale per le Strutture Algebriche, Geometriche e le loro Applicazioni (GNSAGA) of the Istituto Nazionale di Alta Matematica (INdAM)}
    
    \address{Dipartimento di Matematica e Applicazioni, Universit\`a degli Studi di Milano-Bicocca, Ed. U5, Via R.Cozzi 55, 20125 Milano, Italy, EU}
    \email{\href{mailto:matteo.tarocchi.math@gmail.com}{matteo.tarocchi.math@gmail.com}}

\begin{document}

\begin{abstract}
    We describe a method for solving the conjugacy problem in a vast class of rearrangement groups of fractals, a family of Thompson-like groups introduced in 2019 by Belk and Forrest.
    We generalize the methods of Belk and Matucci for the solution of the conjugacy problem in Thompson groups $F$, $T$ and $V$ via strand diagrams.
    In particular, we solve the conjugacy problem for the Basilica, the Airplane, the Vicsek and the Bubble Bath rearrangement groups and for the groups $QV$ (also known as $QAut(\mathcal{T}_{2,c})$), $\tilde{Q}V$, $QT$, $\tilde{Q}T$ and $QF$, and we provide a new solution to the conjugacy problem for the Houghton groups and for the Higman-Thompson groups, where conjugacy was already known to be solvable.
    Our methods involve two distinct rewriting systems, one of which is an instance of a graph rewriting system, whose confluence in general is of interest in computer science.
\end{abstract}

\maketitle

\section*{Introduction}

The groups $F$, $T$ and $V$ were first introduced in 1965 in handwritten notes by Richard Thompson.
Soon $T$ and $V$ caught the attention of group theorists for being the first examples of infinite simple finitely presented groups (in truth, it was later shown that they are $F_\infty$ \cite{BROWN198745});
while the smaller sibling $F$ is not simple, it shows many other interesting properties of its own and is probably most known for the half-century old question regarding its amenability, which is still unresolved.
In general, there are countless topics where these groups show up: from analysis and dynamical systems to logic and cryptography.
A standard introduction to Thompson groups can be found in \cite{cfp}.

Thompson groups $F$, $T$ and $V$ can be described as groups of certain piecewise-linear homeomorphisms of the unit interval, the unit circle and the Cantor space, respectively.
In \cite{belk2016rearrangement} Belk and Forrest introduced Thompson-like groups that act by piecewise-linear homeomorphisms on limit spaces of certain sequences of graphs, which they called \textit{rearrangement groups of fractals}.
The first of these groups to have been studied were the Basilica rearrangement group $T_B$ in \cite{Belk_2015} and the Airplane rearrangement group $T_A$ in \cite{Airplane} (both names descend from the fractals on which the groups act: the Basilica and Airplane Julia sets).

The \textbf{conjugacy problem} is the decision problem of determining whether two given elements of a group $G$ are conjugate in $G$.
The \textbf{conjugacy search problem} is the problem of producing a conjugating element between two given elements of a group $G$ that are conjugate in $G$.
Both problems have been shown to be solvable in the three original Thompson groups:
Guba and Sapir solved it for the whole class of diagram groups, which includes $F$ \cite{guba1997diagram}; it was solved for $V$ by Higman \cite{higman1974finitely} and more generally for Higman-Thompson groups in \cite{HigmanConj}.
In \cite{Belk2007ConjugacyAD} Belk and Matucci produced a unified solution for the conjugacy problem in Thompson groups $F$, $T$ and $V$ with a technique that involves the use of special graphs called \textit{strand diagrams} as a way to represent the elements of the groups and, after being ``closed'', their conjugacy classes.
Generalizations of strand diagrams were also used to study conjugacy in almost automorphism groups of trees by Goffer and Lederle \cite{Goffer2019ConjugacyAD} and to solve the conjugacy problem in symmetric Thompson groups (a generalization of $F$, $T$ and $V$) by Aroca \cite{Aroca2018TheCP}.
Diagrams of this kind have also been used by Bux and Santos Rego to solve the conjugacy problem in the braided Thompson's group $V_{br}$ in a forthcoming work \cite{brVconj}, which has been announced in \cite{braidedreport} and is currently being written up.

Mainly inspired by \cite{Belk2007ConjugacyAD}, the aim of this paper is to introduce a version of strand diagrams that represent the elements of a rearrangement group, and to use them to solve the conjugacy problem under a certain condition on the replacement system that generates the fractal on which the group acts.
More precisely, we prove the following, which is a collection of \cref{THM conjugate,THM known groups}:

\addvspace{\topsep}
\noindent
\textbf{Main Theorem.}
\emph{Given an expanding replacement system whose replacement rules are reduction-confluent, the conjugacy problem and the conjugacy search problem are solvable in the associated rearrangement group. In particular, this method can be used to solve the conjugacy and the conjugacy search problems in the following groups: 
\begin{itemize}
    \item the Higman-Thompson groups $V_{n,r}$,
    \item the Basilica rearrangement group $T_B$, 
    \item the Airplane rearrangement group $T_A$ (through an adapted method, see below),
    \item the other rearrangement groups mentioned in \cite{belk2016rearrangement}, i.e., the Vicsek, the Bubble Bath and the Rabbit rearrangement groups,
    \item the Houghton groups $H_n$ from \cite{Houghton1978TheFC},
    \item the groups $QV$, $\tilde{Q}V$, $QT$, $\tilde{Q}T$ and $QF$ from \cite{QV1} ($QV$ was introduced in \cite{QV} under the name $QAut(\mathcal{T}_{2,c})$ and was also studied in \cite{QV1}),
    \item certain topological full groups of (one-sided) subshifts of finite type.
\end{itemize}
In the end, this solves the conjugacy problem for all rearrangement groups that have been studied so far.}
\par\addvspace{\topsep}

We recall that, for the Houghton groups, the conjugacy problem was already solved with different methods in \cite{HoughtonConjugacy}, and even the twisted conjugacy problem was solved in \cite{COX2017390}.
However, the observation that such groups can be realized as rearrangement groups is new and leads to an independent solution to the conjugacy problem.
Moreover, the solution to the conjugacy problem for the groups $QV$, $\tilde{Q}V$, $QT$, $\tilde{Q}T$ and $QF$ and for the rearrangement groups of the Airplane, the Basilica, the Vicsek, the Bubble Bath and the Rabbit are new, as far as the author knows.
We also remark here that, despite $QV$ being an extension of $V$, Theorem 3.1 from \cite{10.2307/40590898} does not apply to solve the conjugacy problem in $QV$ because the third condition is not met.
Finally, it is interesting to note that $T_B$ is a finitely generated group with solvable conjugacy problem that is not finitely presented (which was proved in \cite{TBnotfinpres}).

We employ strand diagrams, but unlike the solution of the conjugacy problem for $F,T,V$ in \cite{Belk2007ConjugacyAD} we need to introduce two different yet related rewriting systems, one of which is an instance of the so-called \textit{graph rewriting systems}.
We then need to find conditions to find uniquely reduced diagrams, for which we introduce the \textit{reduction-confluence} condition to achieve confluence in our instance of graph rewriting system.
We remark that establishing whether a graph rewriting system is confluent is in general an unsolvable problem and the search for sufficient conditions is actively studied in computer science (see \cite{Plump2005}).

All the groups in the Main Theorem above satisfy the reduction-confluence hypothesis with the exception of $T_A$, but this issue can be circumvented by adding one ``virtual reduction'' to the Airplane replacement system, as described in \cref{SUB non conf}.
More generally, as discussed at the end of this work, it seems that strand diagrams can be used to solve the conjugacy problem with this method whenever the reduction system can be made confluent by adding finitely many reduction rules that preserve termination of the rewriting system and the equivalence generated by the rewriting.
Determining when this can be achieved is beyond the scope of this paper, since it is related to problems in the more general setting of graph rewriting systems, as briefly discussed in \cref{SUB DPO}.

Before delving into rearrangement groups, note that many generalizations of the conjugacy problem have been studied for Thompson groups:
for example, the simultaneous conjugacy problem for $F$ is solved in \cite{simultaneous} and the twisted conjugacy problem and property $R_\infty$ have been studied in $F$ and $T$ in \cite{Bleak2007TwistedCC, Burillo2013TheCP, Gonalves2018TwistedCI}.
Also, \cite{Belk2021ConjugatorLI} makes use of the technology of strand diagrams to study the conjugator lengths in Thompson groups.
All of these questions can be asked about rearrangement groups as well, and maybe they can be tackled with the aid of strand diagrams.
Additionally, rearrangement groups can be seen as a special case of the recently defined \emph{graph diagram groups} \cite{graphdiagramgroups} and it would be interesting to develop a concept of strand diagram in that setting and investigate conjugacy.

Finally, we observe that when strand diagrams were introduced in \cite{Belk2007ConjugacyAD}, they were also used to study dynamics in Thompson groups.
In a similar fashion, a first version of the forest pair diagrams for rearrangements that is fully developed here was previously used to understand some of the dynamics of rearrangements in \cite{IG} where a generalization of Brin's \textit{revealing pairs} from \cite{Brin2004HigherDT} was introduced.
The strand diagrams for rearrangement groups introduced in this work are likely usable to shed further light on the dynamics of elements as done in \cite{Belk2007ConjugacyAD} for $F$, $T$ and $V$, but we leave such questions for future study.

\subsection*{Organization of the paper}

\cref{SEC rearrangement groups} provides all of the background that is needed on replacement systems and rearrangement groups, and also introduces forest pair diagrams;
\cref{SEC strand diagrams} is about strand diagrams, which is an alternative way to represent rearrangements;
\cref{SEC closed strand diagrams} explains what closed strand diagrams are and how to reduce them;
finally, \cref{SEC conjugacy problem} shows how equivalence classes of closed strand diagrams represent conjugacy classes and features a pseudo-algorithm that solves the conjugacy problem under some conditions on the replacement system.

\section{Rearrangement Groups}
\label{SEC rearrangement groups}

In this Section we briefly recall the basics of \textit{replacement systems}, their \textit{limit spaces} and their \textit{rearrangement groups} (Subsections \ref{SUB replacement systems}, \ref{SUB limit spaces}, \ref{SUB rearrangements}, respectively);
for an exhaustive introduction to the topic, we refer to the paper \cite{belk2016rearrangement} that introduced these groups, which goes into much more details than we will get to do.
We then describe many examples of replacement systems and their rearrangement groups (\cref{SUB examples}).
We also introduce \textit{forest pair diagrams} (\cref{SUB FPDs}), which is a new way of describing rearrangements and a natural generalization of \textit{tree pair diagrams} commonly used for Thompson groups.

\subsection{Replacement Systems}
\label{SUB replacement systems}

\begin{definition}\label{DEF replacement}
If $\mathrm{C}$ is a finite set of colors and $R = \{ X_i \mid i \in \mathrm{C} \}$ is a set of graphs colored by $\mathrm{C}$ and each equipped with a single initial and a single terminal vertices, we say that the pair $(R, \mathrm{C})$ is a set of \textbf{replacement rules} and each of the $X_i$'s is a \textbf{replacement graph}.

A \textbf{replacement system} is a set of replacement rules $(R, \mathrm{C})$ along with a graph $X_0$ colored by $\mathrm{C}$, called the \textbf{base graph}.
\end{definition}

For example, \cref{fig_replacement_A} depicts the so called Airplane.
We denote it by $\mathcal{A}$ and we will use it as a guiding example going forward.
For many more examples of replacement systems, limit spaces and rearrangement groups, we refer the reader to \cref{SUB examples}

We can expand the base graph $\Gamma$ by replacing one of its edges $e$ with the replacement graph $X_c$ indexed by the color $c$ of $e$, as in \cref{fig_exp_A}.
The graph resulting from this process of replacing one edge with the appropriate replacement graph is called a \textbf{simple graph expansion}.
Simple graph expansions can be iterated any finite amount of times, which generate the so-called \textbf{graph expansions} of the replacement system, such as the one in \cref{fig_exp_A_generic}.

\begin{figure}\centering
\begin{subfigure}{.375\textwidth}
\centering
\begin{tikzpicture}[scale=.85]
    \draw[->-=.5,blue] (-0.5,0) -- node[above]{L} (-2,0) node[black,circle,fill,inner sep=1.25]{}; \draw[->-=.5,blue] (0.5,0) -- node[above]{R} (2,0) node[black,circle,fill,inner sep=1.25]{};
    \draw[->-=.5,red] (0.5,0) node[circle,fill,inner sep=1.25]{} to[out=90,in=90,looseness=1.7] node[above]{T} (-0.5,0);
    \draw[->-=.5,red] (-0.5,0) node[black,circle,fill,inner sep=1.25]{} to[out=270,in=270,looseness=1.7] node[below]{B} (0.5,0) node[black,circle,fill,inner sep=1.25]{};
\end{tikzpicture}
\caption{The base graph.}
\label{fig_A_base}
\end{subfigure}
\begin{subfigure}{.6\textwidth}
    \centering
    \begin{subfigure}{.45\textwidth}
    \centering
    \begin{tikzpicture}[scale=1]
        \draw[->-=.5,red] (0,1) node[black,circle,fill,inner sep=1.25]{} node[black,above]{$\iota$} -- node[below]{a} (1.05,1);
        \draw[->-=.5,red] (1.05,1) -- node[below]{c} (2.1,1) node[black,circle,fill,inner sep=1.25]{} node[black,above]{$\tau$};
        \draw[->-=.5,blue] (1.05,1) node[black,circle,fill,inner sep=1.25]{} -- node[right]{b} (1.05,2) node[black,circle,fill,inner sep=1.25]{};
    \end{tikzpicture}
    \end{subfigure}
    \begin{subfigure}{.45\textwidth}
    \centering
    \begin{tikzpicture}[scale=1]
        \draw[->-=.5,blue] (0.7,1.35) -- node[below]{e} (0,1.35) node[black,circle,fill,inner sep=1.25]{} node[black,above]{$\iota$};
        \draw[->-=.5,blue] (1.4,1.35) -- node[below]{h} (2.1,1.35) node[black,circle,fill,inner sep=1.25]{} node[black,above]{$\tau$};
        \draw[->-=.5,red] (1.4,1.35) to[out=90,in=90,looseness=1.7] node[above]{f} (0.7,1.35);
        \draw[->-=.5,red] (0.7,1.35) node[black,circle,fill,inner sep=1.25]{} to[out=270,in=270,looseness=1.7] node[below]{g} (1.4,1.35) node[black,circle,fill,inner sep=1.25]{};
    \end{tikzpicture}
    \end{subfigure}
    \caption{The \textcolor{red}{red} and \textcolor{blue}{blue} replacement graphs, respectively.}
    \label{fig_A_replacement_rule}
\end{subfigure}
\caption{The Airplane replacement system $\mathcal{A}$.}
\label{fig_replacement_A}
\end{figure}

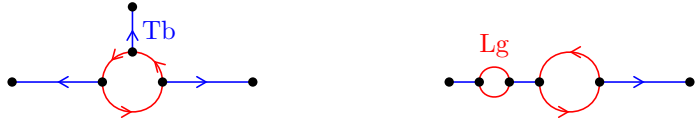
\begin{figure}\centering
\begin{subfigure}{.45\textwidth}\centering
\begin{tikzpicture}[scale=.8]
    \draw[->-=.5,blue] (-0.5,0) -- (-2,0) node[black,circle,fill,inner sep=1.25]{}; \draw[->-=.5,blue] (0.5,0) -- (2,0) node[black,circle,fill,inner sep=1.25]{};
    \draw[->-=.5,red] (0.5,0) to[out=90,in=0] (0,0.5);
    \draw[->-=.5,blue] (0,0.5) -- node[right]{Tb} (0,1.25) node[black,circle,fill,inner sep=1.25]{};
    \draw[->-=.5,red] (0,0.5) node[black,circle,fill,inner sep=1.25]{} to[out=180,in=90] (-0.5,0);
    \draw[->-=.5,red] (-0.5,0) node[black,circle,fill,inner sep=1.25]{} to[out=270,in=270,looseness=1.7] (0.5,0) node[black,circle,fill,inner sep=1.25]{};
\end{tikzpicture}
\end{subfigure}
\begin{subfigure}{.45\textwidth}
\centering
\begin{tikzpicture}[scale=.8]
    \draw[blue] (-0.5,0) -- (-1,0); \draw[blue] (-2,0) node[black,circle,fill,inner sep=1.25]{} -- (-1.5,0);
    \draw[red] (-1,0) to[out=90,in=90,looseness=1.7] node[above]{Lg} (-1.5,0);
    \draw[red] (-1.5,0) node[black,circle,fill,inner sep=1.25]{} to[out=270,in=270,looseness=1.7] (-1,0) node[black,circle,fill,inner sep=1.25]{};
    \draw[->-=.5,blue] (0.5,0) -- (2,0) node[black,circle,fill,inner sep=1.25]{};
    \draw[->-=.5,red] (0.5,0) to[out=90,in=90,looseness=1.7] (-0.5,0);
    \draw[->-=.5,red] (-0.5,0) node[black,circle,fill,inner sep=1.25]{} to[out=270,in=270,looseness=1.7] (0.5,0) node[black,circle,fill,inner sep=1.25]{};
    \node at (0,1.25) {};
\end{tikzpicture}
\end{subfigure}
\caption{Two simple graph expansions of the base graph of the Airplane replacement system $\mathcal{A}$.}
\label{fig_exp_A}
\end{figure}

\begin{figure}\centering
\begin{tikzpicture}[scale=1.5]
\draw[blue] (-0.5,0) -- (-2,0);
\draw[blue] (0.5,0) -- (2,0);
\draw[blue] (0,-0.5) -- (0,-1.25);
\draw[blue] (0.35,-0.35) -- (0.75,-0.75);
\draw[blue] (0.55,-0.55) -- (0.7,-0.4);
\draw[blue] (1.25,0) -- (1.25,0.75);
\draw[blue] (1.25,0) -- (1.625,0.375);
\draw[blue] (-1.25,0) -- (-1.25,0.75);
\draw[blue] (-1.25,0.5) -- (-1.05,0.5);
\draw[blue] (-1.75,0) -- (-1.75,-0.2);
\draw[red] (0,0) circle (0.5);
\draw[red,fill=white] (-1.25,0) circle (0.25);
\draw[red,fill=white] (1.25,0) circle (0.25);
\draw[red,fill=white] (0.55,-0.55) circle (0.0625);
\draw[red,fill=white] (-1.25,0.5) circle (0.0625);
\draw[red,fill=white] (-1.11875,0.5) circle (0.025);
\draw[red,fill=white] (-1.75,0) circle (0.0625);
\end{tikzpicture}
\caption{A generic graph expansion of the Airplane replacement system $\mathcal{A}$ (directions omitted for the sake of clarity).}
\label{fig_exp_A_generic}
\end{figure}
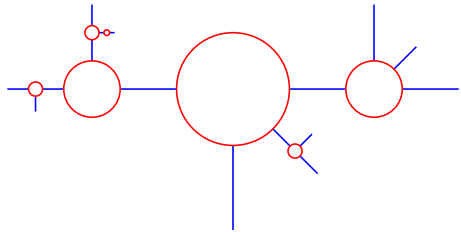

\label{DEF expanding}
A replacement system is said to be \textbf{expanding} if:
\begin{itemize}
    \item neither the base graph nor any replacement graph contains any isolated vertex (i.e., a vertex that is not initial nor terminal for any edge);
    \item in each replacement graph, the initial and terminal vertices are not connected by an edge;
    \item each replacement graph has at least three vertices and two edges.
\end{itemize}
From now on, $\mathcal{X} = (X_0, R, \mathrm{C})$ will always be an expanding replacement system whose colors are $\mathrm{C} = \{1, \dots, k \}$ and whose replacement graphs by $R= \{X_i \mid i \in \mathrm{C} \}$.

Being expanding is a condition that guarantees the existence of the limit space, and every rearrangement group studied so far can be realized with an expanding replacement system.
However, it is worth noting that there are non expanding replacement systems whose limit spaces can be defined in a natural way that allows the construction of their rearrangement groups, but we will not discuss this here.

\subsection{Limit Spaces and Cells}
\label{SUB limit spaces}

Consider the \textbf{full expansion sequence} of $\mathcal{X}$, which is the sequence of graphs obtained by replacing, at each step, every edge with the appropriate replacement graph, starting from the base graph $X_0$.

Whenever a replacement system is expanding, we can define its \textbf{limit space}, which is essentially the limit of the full expansion sequence;
the complete and detailed construction (which consists of taking a quotient of the Cantor space under a so-called \textit{gluing relation} induced by edge adjacency in the graph expansions) is featured in Subsection 1.2 of \cite{belk2016rearrangement}.
Continuing with our example, \cref{fig_limit_spaces} portrays the limit spaces of the Airplane replacement system.
More examples will be portrayed in \cref{SUB examples}.
Limit spaces have nice topological properties, such as being compact and metrizable (Theorem 1.25 of \cite{belk2016rearrangement}).

\begin{remark}\label{TXT gluing relation}
    Each edge of a graph expansion of the base graph can be seen as a finite \textit{address} (a sequence of symbols), as in \cref{fig_exp_A}.
    Each bit of the address corresponds to an edge of a replacement graph, except for the first one which is instead an edge of the base graph.
    Points of the limit space are infinite addresses modulo a \textit{gluing relation}, which is given by edge adjacency.
    Intuitively, if two such addresses approach the same point in the graph expansions, then they are equivalent under the gluing relation, thus they represent the same point in the limit space.
\end{remark}

\begin{figure}\centering
    \includegraphics[width=.425\textwidth]{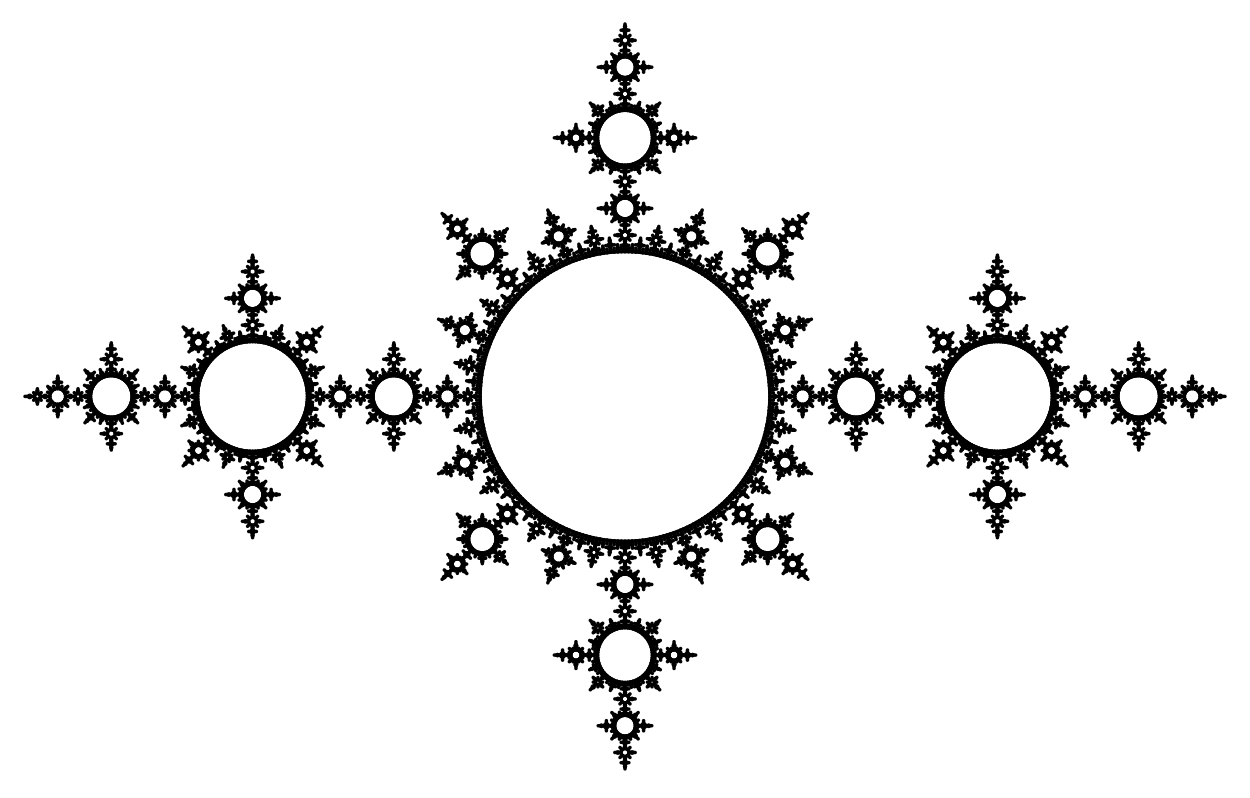}
    \caption{The Airplane limit space. Image from \cite{belk2016rearrangement}.}
    \label{fig_A_limit_space}
\end{figure}

Each limit space is naturally composed of self-similar pieces called \textbf{cells}.
Intuitively, a cell $C(e)$ is the subset of the limit space corresponding to the edge $e$ of some graph expansion, which consists of everything that appears from that edge in subsequent graph expansions.
\cref{fig_cells_A} shows two examples of cells of $\mathcal{A}$.

\begin{figure}\centering
\includegraphics[width=.425\textwidth]{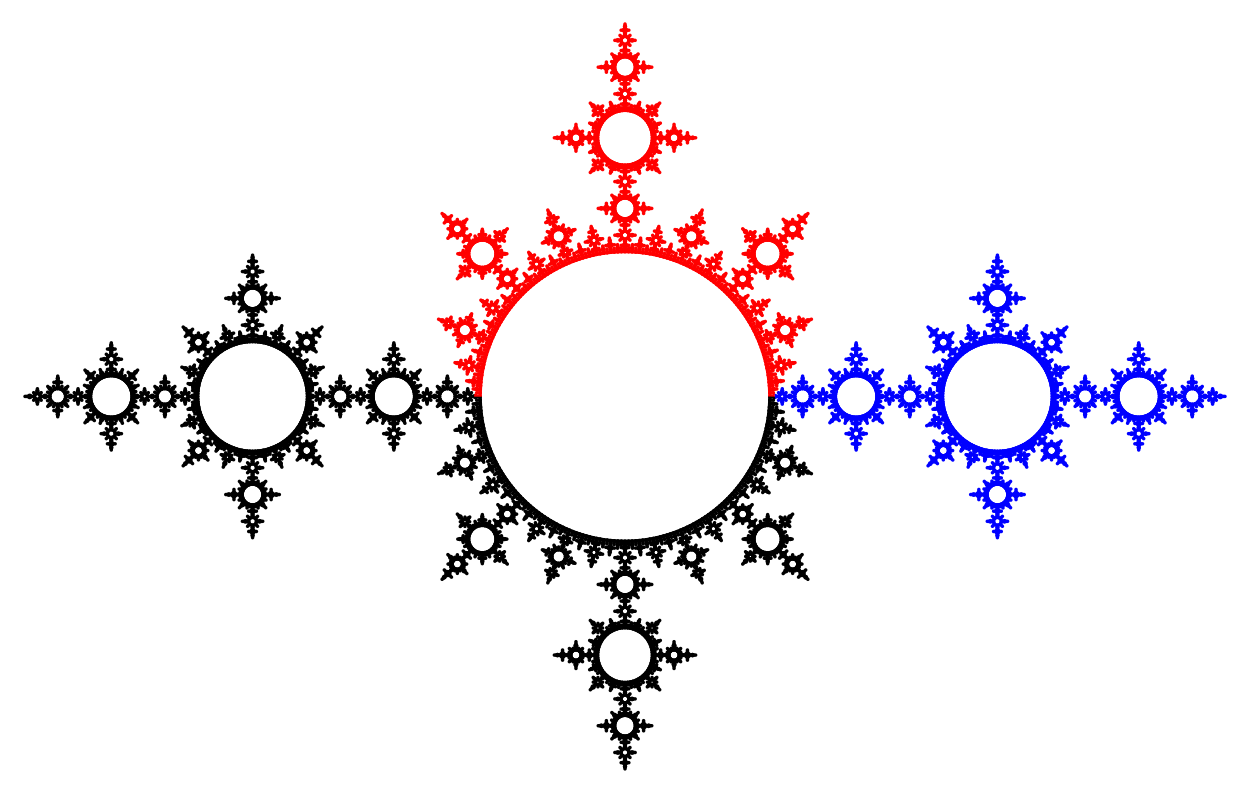}
\caption{The two types of cells in $\mathcal{A}$, distinguished by the color of the generating edge.}
\label{fig_cells_A}
\end{figure}

\subsection{Rearrangements of Limit Spaces}
\label{SUB rearrangements}

There are different \textit{types} of cells $C(e)$, distinguished by two aspects of the generating edge $e$: its color and whether or not it is a loop.
It is not hard to see that there is a \textbf{canonical homeomorphism} between any two cells of the same type.
A canonical homeomorphism between two cells can essentially be thought as a transformation that maps the first cell ``linearly'' to the second, or equivalently as a prefix exchange of addresses (in the sense discussed in \cref{TXT gluing relation}).
The much more detailed definition can be found in \cite{belk2016rearrangement}.

\begin{definition}
\label{DEF cellular partition}
A \textbf{cellular partition} of the limit space $X$ is a cover of $X$ by finitely many cells whose interiors are disjoint.
\end{definition}

Note that there is a natural bijection between the set of graph expansions of a replacement system and the set of cellular partitions.
The bijection is given by mapping each graph expansion to the set consisting of the cells $C(e)$ for every edge $e$ of the graph expansion.

\begin{definition}
\label{DEF rearrangement}
A homeomorphism $f: X \to X$ is called a \textbf{rearrangement} of $X$ if there exists a cellular partition $P$ of $X$ such that $f$ restricts to a canonical homeomorphism on each cell of $P$.
\end{definition}

It can be proved that the rearrangements of a limit space $X$ form a group under composition, called the \textbf{rearrangement group} of $X$.
Despite what may appear from this name, the rearrangement group is not uniquely determined by the limit space $X$, but instead depends on the replacement system $\mathcal{X}$:
distinct replacement systems can produce the same limit space but yield distinct rearrangement groups.

\begin{figure}\centering
\begin{subfigure}[b]{.45\textwidth}
\centering
\begin{tikzpicture}
    \draw[->-=.5] (0,0) -- (4,0);
    \node at (0,0) [circle,fill,inner sep=1.5]{};
    \node at (4,0) [circle,fill,inner sep=1.5]{};
    \draw[white] (0,-1);
\end{tikzpicture}
\caption{The base graph $\Gamma$.}
\end{subfigure}
\begin{subfigure}[b]{.45\textwidth}
\centering
\begin{tikzpicture}
\draw[-angle 60] (0,0) node[circle,fill,inner sep=1.25]{} node[left]{$\iota$} -- (0,1); \draw (0,1) -- (0,2) node[circle,fill,inner sep=1.25]{} node[left]{$\tau$};

\draw[-stealth] (0.4,1) -- (0.8,1);

\draw[-angle 60] (1.2,0) node[circle,fill,inner sep=1.25]{} node[left]{$\iota$} -- (1.2,0.5); \draw (1.2,0.5) -- (1.2,1) node[circle,fill,inner sep=1.25]{}; \draw[-stealth] (1.2,1) -- (1.2,1.5); \draw (1.2,1.5) -- (1.2,2) node[circle,fill,inner sep=1.25]{} node[left]{$\tau$};
\end{tikzpicture}
\caption{The replacement rule $e \to R$.}
\end{subfigure}

\caption{The replacement system for the dyadic subdivision of the unit interval $[0,1]$, whose rearrangement group is Thompson's group $F$.}
\label{fig_replacement_interval}
\end{figure}
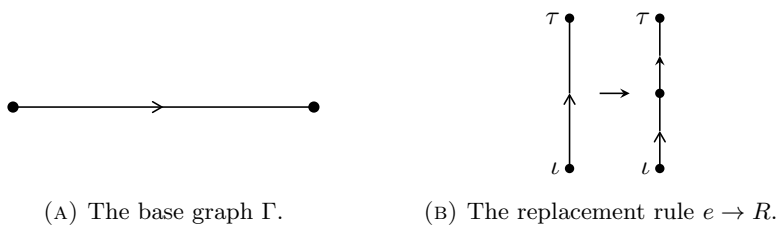

\begin{figure}\centering
\begin{subfigure}{.45\textwidth}
\centering
\begin{tikzpicture}
    \draw[->-=.5] (0,0) circle (1);
    \node at (1,0) [circle,fill,inner sep=1.5]{};
\end{tikzpicture}
\caption{The base graph $\Gamma$.}
\end{subfigure}
\begin{subfigure}{.45\textwidth}
\centering
\begin{tikzpicture}
\draw[-angle 60] (0,0) node[circle,fill,inner sep=1.25]{} node[left]{$\iota$} -- (0,1); \draw (0,1) -- (0,2) node[circle,fill,inner sep=1.25]{} node[left]{$\tau$};

\draw[-stealth] (0.4,1) -- (0.8,1);

\draw[-angle 60] (1.2,0) node[circle,fill,inner sep=1.25]{} node[left]{$\iota$} -- (1.2,0.5); \draw (1.2,0.5) -- (1.2,1) node[circle,fill,inner sep=1.25]{}; \draw[-stealth] (1.2,1) -- (1.2,1.5); \draw (1.2,1.5) -- (1.2,2) node[circle,fill,inner sep=1.25]{} node[left]{$\tau$};
\end{tikzpicture}
\caption{The replacement rule $e \to R$.}
\end{subfigure}

\caption{The replacement system for the dyadic subdivision of the unit circle $S^1$, whose rearrangement group is Thompson's group $T$.}
\label{fig_replacement_circle}
\end{figure}

\begin{figure}\centering
\begin{subfigure}[b]{.45\textwidth}
\centering
\begin{tikzpicture}
    \draw[->-=.5] (0,0) -- (4,0);
    \node at (0,0) [circle,fill,inner sep=1.5]{};
    \node at (4,0) [circle,fill,inner sep=1.5]{};
    \draw[white] (0,-1);
\end{tikzpicture}
\caption{The base graph $\Gamma$.}
\end{subfigure}
\begin{subfigure}[b]{.45\textwidth}
\centering
\begin{tikzpicture}
\draw[-angle 60] (0,0) node[circle,fill,inner sep=1.25]{} node[left]{$\iota$} -- (0,1); \draw (0,1) -- (0,2) node[circle,fill,inner sep=1.25]{} node[left]{$\tau$};

\draw[-stealth] (0.4,1) -- (0.8,1);

\draw[-angle 60] (1.2,0) node[circle,fill,inner sep=1.25]{} node[left]{$\iota$} -- (1.2,0.5); \draw (1.2,0.5) -- (1.2,0.8) node[circle,fill,inner sep=1.25]{}; \draw[-stealth] (1.2,1.2) node[circle,fill,inner sep=1.25]{} -- (1.2,1.7); \draw (1.2,1.7) -- (1.2,2) node[circle,fill,inner sep=1.25]{} node[left]{$\tau$};
\end{tikzpicture}
\caption{The replacement rule $e \to R$.}
\end{subfigure}

\caption{The replacement system for the Cantor space, whose rearrangement group is Thompson's group $V$.}
\label{fig_replacement_cantor}
\end{figure}
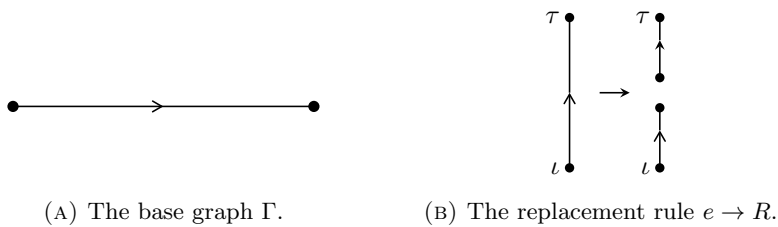

\begin{figure}\centering
\begin{minipage}{.4378\textwidth}\centering
\includegraphics[width=\textwidth]{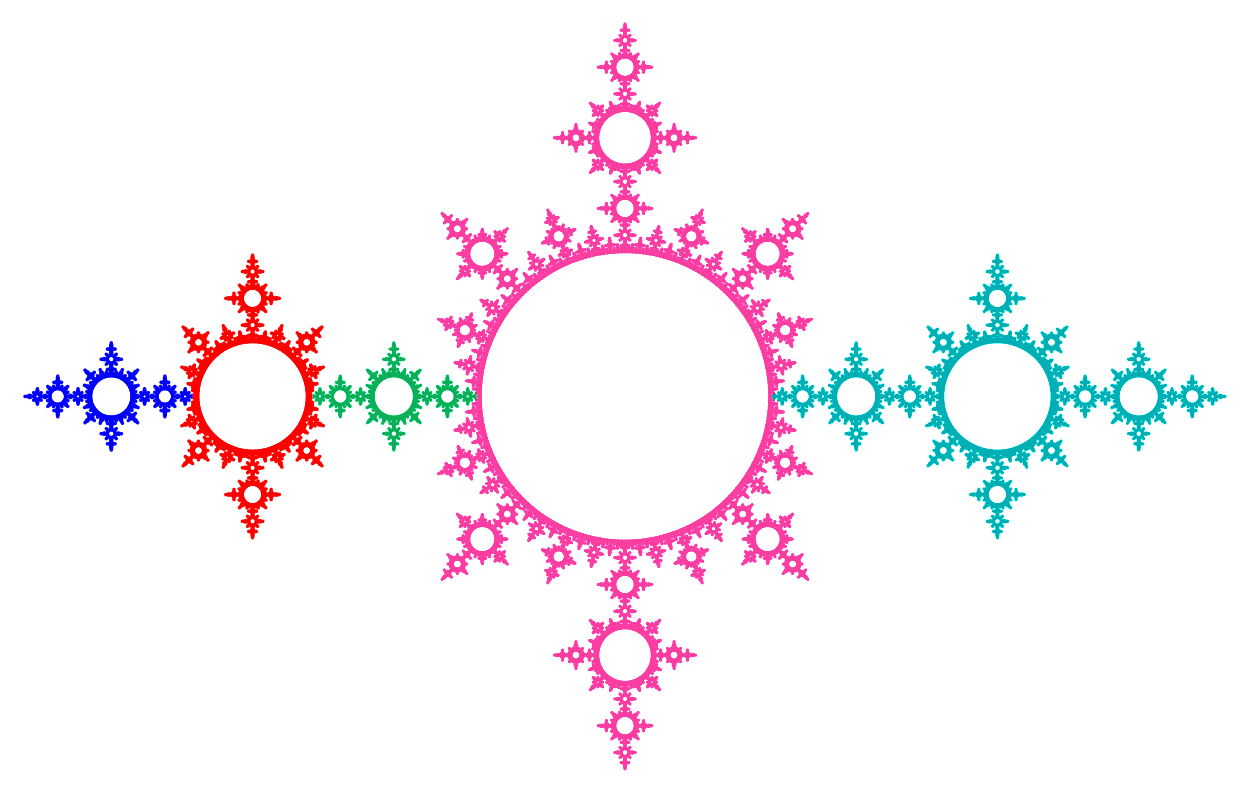}
\end{minipage}%
\begin{minipage}{.095\textwidth}\centering
\begin{tikzpicture}[scale=1.35]
    \draw[-to] (0,0) -- (.3,0);
\end{tikzpicture}
\end{minipage}%
\begin{minipage}{.4378\textwidth}\centering
\includegraphics[width=\textwidth]{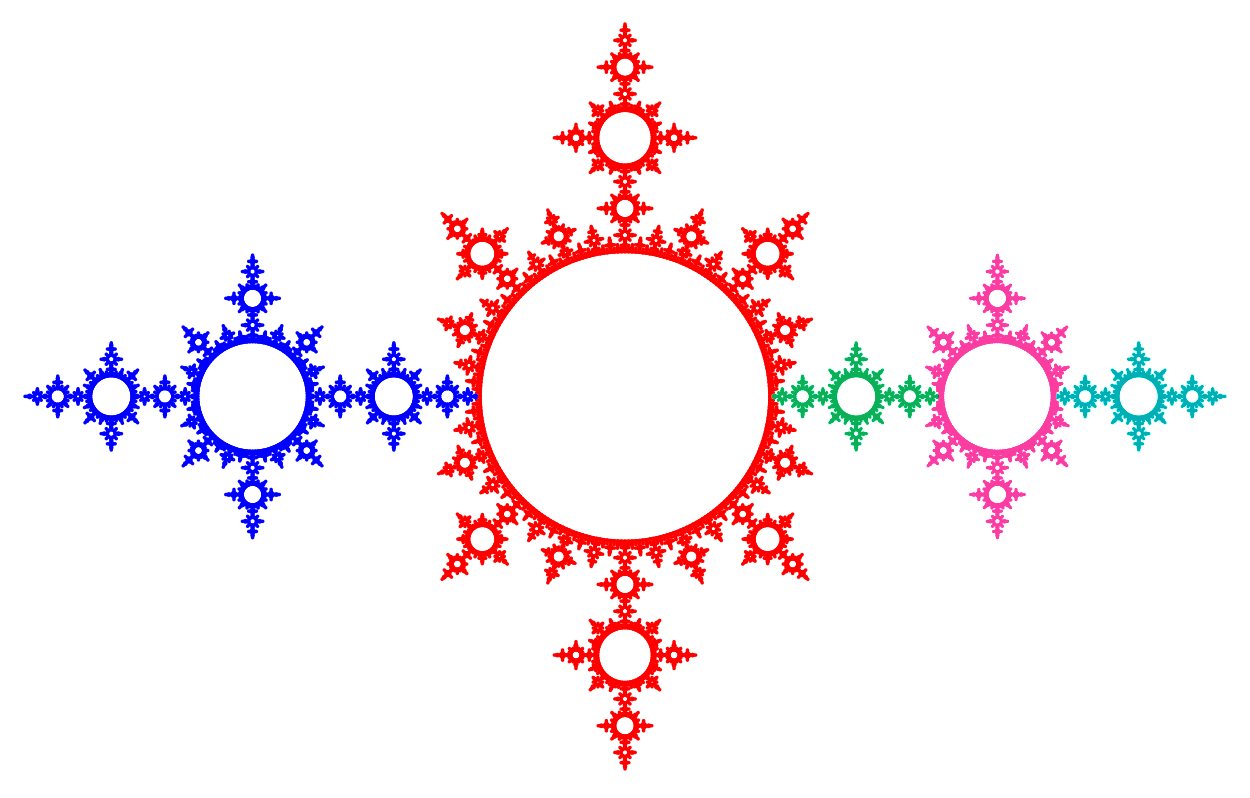}
\end{minipage}\\
\vspace*{10pt}
\begin{tikzpicture}[scale=1.35]
    \draw[blue,->-=.55] (-1.5,0) -- (-2,0);
    \draw[Green,->-=1] (-.75,0) -- (-0.5,0);
    \draw[Green,->-=1] (-.75,0) -- (-1,0);
    \draw[Green,fill=white] (-0.75,0) circle (0.08);
    \draw[TealBlue,->-=.5] (0.5,0) -- (2,0);
    \draw[->-=.5,red] (-1.5,0) to[out=270,in=270,looseness=1.7] (-1,0);
    \draw[->-=.5,red] (-1,0) to[out=90,in=90,looseness=1.7] (-1.5,0);
    \draw[->-=.5,Magenta] (-0.5,0) to[out=270,in=270,looseness=1.7] (0.5,0);
    \draw[->-=.5,Magenta] (0.5,0) to[out=90,in=90,looseness=1.7] (-0.5,0);
    
    \draw[-to] (2.35,0) -- (2.65,0);
    
    \draw[blue,->-=.5] (4.5,0) -- (3,0);
    \draw[red] (5,0) circle (0.5);
    \draw[Green,->-=1] (5.75,0) -- (6,0);
    \draw[Green,->-=1] (5.75,0) -- (5.5,0);
    \draw[Green,fill=white] (5.75,0) circle (0.08);
    \draw[TealBlue,->-=.55] (6.5,0) -- (7,0);
    \draw[->-=.5,red] (4.5,0) to[out=270,in=270,looseness=1.7] (5.5,0);
    \draw[->-=.5,red] (5.5,0) to[out=90,in=90,looseness=1.7] (4.5,0);
    \draw[->-=.5,Magenta] (6,0) to[out=270,in=270,looseness=1.7] (6.5,0);
    \draw[->-=.5,Magenta] (6.5,0) to[out=90,in=90,looseness=1.7] (6,0);
\end{tikzpicture}
\caption{A rearrangement of the Airplane limit space, along with a graph pair diagram that represents it.}
\label{fig_action_alpha}
\end{figure}

\phantomsection\label{TXT graph pair diagrams}
In the setting of rearrangement groups, the trio of Thompson groups $F$, $T$ and $V$ are realized by the replacement systems in Figures \ref{fig_replacement_interval}, \ref{fig_replacement_circle}, \ref{fig_replacement_cantor}, and similar replacement systems also realize the Higman-Thompson groups from \cite{higman1974finitely}.
Similarly to how dyadic rearrangements (the elements of Thompson groups) are specified by permutations between pairs of dyadic subdivisions, rearrangements of a limit space are specified by graph isomorphisms between graph expansions of the replacement systems, called \textbf{graph pair diagrams}.
For example, the rearrangement of the Airplane limit space depicted in \cref{fig_action_alpha} is specified by the graph isomorphism depicted in the same figure.
The colors in that picture mean that each edge of the domain graph is mapped to the edge of the same color in the range graph.

\phantomsection\label{TXT reduced GPDs}
Graph pair diagrams can be \textit{expanded} by expanding an edge in the domain graph and its image in the range graph, resulting in a new graph pair diagram that is more "redundant" then the original one:
it makes the graphs more complex, but it does not add any new information, as the cell corresponding to the edge being expanded in the domain is mapped canonically by the rearrangement.
It is important to note that each rearrangement is represented by a unique \textbf{reduced} graph pair diagram, where reduced means that it is not the result of an expansion of any other graph pair diagram.

\begin{remark}\label{RMK undirected}
A replacement graph $R_c$ may admit a graph automorphism $\phi$ that switches its initial and terminal vertices $\iota$ and $\tau$, as is the case for the blue replacement graph of the Airplane replacement system (\cref{fig_A_replacement_rule}).
When this happens, we can deal with $c$-colored edges as if they were not directed edges, meaning that a $c$-colored edge from $i$ to $t$ can be mapped to a $c$-colored edge from $t$ to $i$, as happens to the edge highlighted in green in \cref{fig_undirected_alpha}.
This is because these edges can be expanded and the resulting graphs are isomorphic independently from the original orientation, so a graph isomorphism that reverses the orientation of an edge is implicitly applying the graph isomorphism $\phi$ to the expansion of that edge (if multiple isomorphisms switches $\iota$ and $\tau$, one needs to consider one such isomorphism $\phi$ and keep it fixed).
In this case, we say that the color $c$ is \textbf{undirected}.
\end{remark}

\begin{figure}\centering
    \begin{tikzpicture}[scale=1.35]
        \draw[blue] (-2,0) -- (-1.5,0);
        \draw[Green] (-1,0) -- (-0.5,0);
        \draw[TealBlue] (0.5,0) -- (2,0);
        \draw[->-=.5,red] (-1.5,0) to[out=270,in=270,looseness=1.7] (-1,0);
        \draw[->-=.5,red] (-1,0) to[out=90,in=90,looseness=1.7] (-1.5,0);
        \draw[->-=.5,Magenta] (-0.5,0) to[out=270,in=270,looseness=1.7] (0.5,0);
        \draw[->-=.5,Magenta] (0.5,0) to[out=90,in=90,looseness=1.7] (-0.5,0);
        
        \draw[-to] (2.35,0) -- (2.65,0);
        
        \draw[blue] (3,0) -- (4.5,0);
        \draw[Green] (5.5,0) -- (6,0);
        \draw[TealBlue] (6.5,0) -- (7,0);
        \draw[->-=.5,red] (4.5,0) to[out=270,in=270,looseness=1.7] (5.5,0);
        \draw[->-=.5,red] (5.5,0) to[out=90,in=90,looseness=1.7] (4.5,0);
        \draw[->-=.5,Magenta] (6,0) to[out=270,in=270,looseness=1.7] (6.5,0);
        \draw[->-=.5,Magenta] (6.5,0) to[out=90,in=90,looseness=1.7] (6,0);
    \end{tikzpicture}
    \caption{The same rearrangement of \cref{fig_action_alpha} represented using undirected blue edges (see \cref{RMK undirected}.}
    \label{fig_undirected_alpha}
\end{figure}
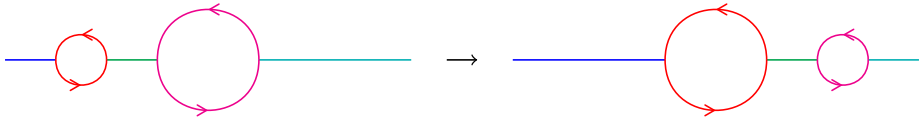

\subsection{Examples of Notable Rearrangement Groups}
\label{SUB examples}

Besides the Airplane, there are many more examples of rearrangement groups coming from previously known fractals.
\cref{fig replacement B,,fig replacement Vic,,fig replacement BB} depict the Basilica, Vicsek and Bubble Bath replacement systems, respectively.
These names come from the fractals that are homeomorphic to their limit spaces, which are shown in \cref{fig_limit_spaces}.
The Basilica replacement system can be generalized by adding more loops to the central vertex of the replacement graph, which results in the so called Rabbit replacement systems (Example 2.3 of \cite{belk2016rearrangement}).
In a similar fashion, the Vicsek replacement system can be generalized by adding more edges originating from the same main ``crossing'' of the replacement graph (Example 2.1 of \cite{belk2016rearrangement}).
Finally, it is worth noting that the Bubble Bath replacement system and a degree-3 variation of the Vicsek replacement system produce rearrangement groups that are similar but actually larger than the ones described in the dissertations \cite{BubbleBath,Dendrite}, respectively:
rearrangements need not preserve the orientation of the limit spaces, thus rearrangement groups properly contain the groups described in these two dissertations.

\begin{figure}\centering
\begin{subfigure}[b]{.4\textwidth}\centering
\begin{tikzpicture}[scale=.75]
    \draw[->-=.5] (0,0) node[circle,fill,inner sep=1.25]{} to[out=60,in=120,looseness=1.5] (2,0);
    \draw[->-=.5] (2,0) node[circle,fill,inner sep=1.25]{} to[out=240,in=300,looseness=1.5] (0,0);
    \draw (0,0) to[in=270,out=240,looseness=1.5] (-1,0); \draw[->-=0] (-1,0) to[in=120,out=90,looseness=1.5] (0,0);
    \draw (2,0) to[in=90,out=60,looseness=1.5] (3,0); \draw[->-=0] (3,0) to[in=300,out=270,looseness=1.5] (2,0);
\end{tikzpicture}
\caption{The base graph.}
\end{subfigure}
\begin{subfigure}[b]{.5\textwidth}\centering
\begin{tikzpicture}[scale=1]
    \draw[->-=.5] (-.25,0) node[circle,fill,inner sep=1.25]{} node[above]{$\iota$} -- (1,0) node[circle,fill,inner sep=1.25]{};
    \draw[->-=.5] (1,0) -- (2.25,0) node[circle,fill,inner sep=1.25]{} node[above]{$\tau$};
    \draw (1,0) to[out=140,in=180,looseness=1.5] (1,0.85); \draw[->-=0] (1,0.85) to[out=0,in=40,looseness=1.75] (1,0);
    \draw[white] (0,-.333);
\end{tikzpicture}
\label{fig_B_replacement_graph}
\caption{The replacement graph.}
\end{subfigure}
\caption{The Basilica replacement system.}
\label{fig replacement B}
\end{figure}
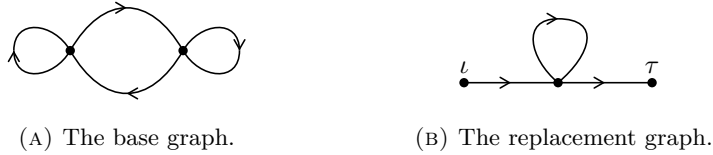

\begin{figure}\centering
\begin{subfigure}[b]{.4\textwidth}\centering
\begin{tikzpicture}[scale=.75]
    \draw[->-=.5] (0,0) node[circle,fill,inner sep=1.25]{} -- (1.5,0) node[circle,fill,inner sep=1.25]{};
    \draw[->-=.5] (0,0) -- (-1.5,0) node[circle,fill,inner sep=1.25]{};
    \draw[->-=.5] (0,0) -- (0,1.5) node[circle,fill,inner sep=1.25]{};
    \draw[->-=.5] (0,0) -- (0,-1.5) node[circle,fill,inner sep=1.25]{};
\end{tikzpicture}
\caption{The base graph.}
\end{subfigure}
\begin{subfigure}[b]{.5\textwidth}\centering
\begin{tikzpicture}[scale=.6]
    \draw[->-=.5] (0,0) node[circle,fill,inner sep=1.25]{} -- (0,1.5) node[circle,fill,inner sep=1.25]{};
    \draw[->-=.5] (0,0) -- (0,-1.5) node[circle,fill,inner sep=1.25]{};
    \draw[->-=.5] (0,0) -- (1.5,0) node[above]{$\tau$} node[circle,fill,inner sep=1.25]{};
    \draw[->-=.5] (0,0) -- (-1.5,0) node[circle,fill,inner sep=1.25]{};
    \draw[->-=.5] (-3,0) node[above]{$\iota$} node[circle,fill,inner sep=1.25]{} -- (-1.5,0);
\end{tikzpicture}
\caption{The replacement graph.}
\end{subfigure}
\caption{The Vicsek replacement system.}
\label{fig replacement Vic}
\end{figure}
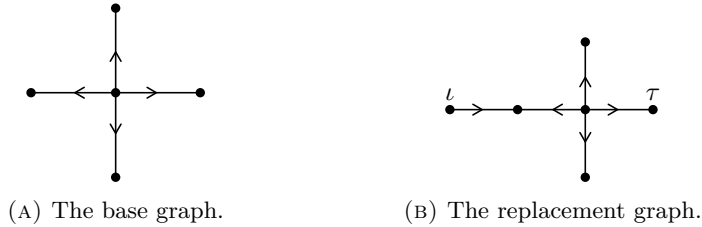

\begin{figure}\centering
\begin{subfigure}[b]{.4\textwidth}\centering
\begin{tikzpicture}[scale=.9]
    \draw[->-=.5] (0,0) node[circle,fill,inner sep=1.25]{} to[out=180,in=180,looseness=1.8] (0,2) node[circle,fill,inner sep=1.25]{};
    \draw[->-=.5] (0,0) -- (0,2);
    \draw[->-=.5] (0,0) to[out=0,in=0,looseness=1.8] (0,2);
\end{tikzpicture}
\caption{The base graph.}
\end{subfigure}
\begin{subfigure}[b]{.5\textwidth}\centering
\begin{tikzpicture}[scale=.8]
    \draw[->-=.5] (-0.5,0) -- (-2,0) node[above]{$\iota$} node[black,circle,fill,inner sep=1.25]{};
    \draw[->-=.5] (0.5,0) -- (2,0) node[above]{$\tau$} node[black,circle,fill,inner sep=1.25]{};
    \draw[->-=.5] (0.5,0) node[circle,fill,inner sep=1.25]{} to[out=90,in=90,looseness=1.7] (-0.5,0);
    \draw[->-=.5] (-0.5,0) node[black,circle,fill,inner sep=1.25]{} to[out=270,in=270,looseness=1.7] (0.5,0) node[black,circle,fill,inner sep=1.25]{};
    \draw[white] (0,-1);
\end{tikzpicture}
\caption{The replacement graph.}
\end{subfigure}
\caption{The Bubble Bath replacement system.}
\label{fig replacement BB}
\end{figure}
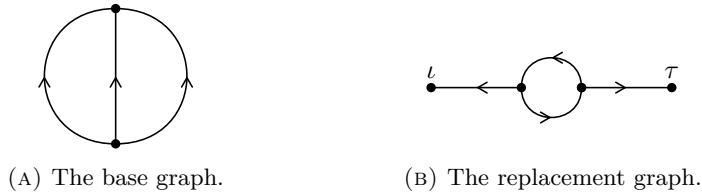

\begin{figure}\centering
\begin{subfigure}[c]{.5\textwidth}\centering
    \includegraphics[width=\textwidth]{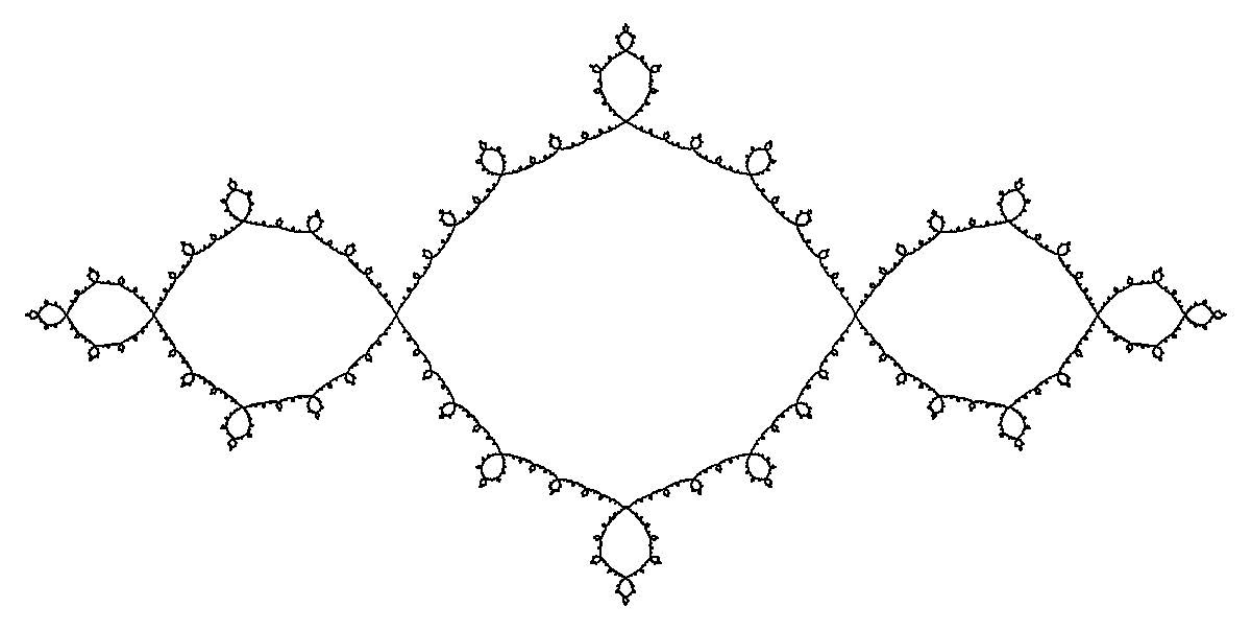}
\end{subfigure}
\hspace{15pt}
\begin{subfigure}[c]{.35\textwidth}\centering
    \includegraphics[width=\textwidth]{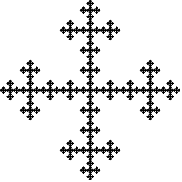}
\end{subfigure}
\\
\begin{subfigure}[c]{.45\textwidth}\centering    \includegraphics[width=\textwidth]{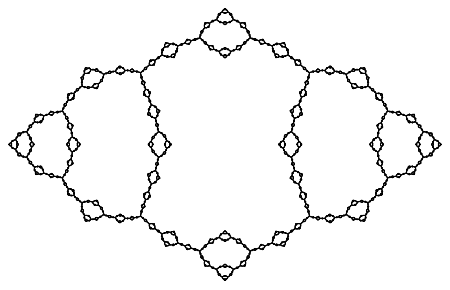}
\end{subfigure}
\hspace{15pt}
\caption{The Basilica, Vicsek and Bubble Bath limit spaces. Images made by James Belk for \cite{Belk_2015} and \cite{belk2016rearrangement}.}
\label{fig_limit_spaces}
\end{figure}

Other examples are the Houghton groups $H_n$ from \cite{Houghton1978TheFC} and the groups $QV$, $\tilde{Q}V$, $QT$ and $QF$ from \cite{QV2} (the first one was denoted by $QAut(\mathcal{T}_{2,c})$ when it appeared in \cite{QV} and later in \cite{QV1}).
The replacement systems for the Houghton groups and the ones for $QV$, $QT$ and $QF$ are depicted in \cref{fig Houghton} and \cref{fig QV}, respectively, and one can obtain those for $\tilde{Q}V$ and $\tilde{Q}T$ from $QV$ and $QT$, respectively, by adding a blue edge to the base graph.
In both replacement systems, the black replacement graphs was chosen because it supports trivial rearrangements (see Example 2.5 of \cite{belk2016rearrangement}).
More natural replacement systems for all of these groups would be the ones where the black replacement graph consists of a sole black edge;
these replacement systems are not expanding, but it possible to define their limit spaces:
the sequences of edges that do not expand would result in single isolated points.
Additionally, changing the base graph of $QV$ would provide a Higman-Thompson kind of generalization for $QV$ that would be natural to denote by $QV_{n,k}$ and the same goes for $QT$ and $QF$ (see also Remark 2.6 of \cite{mallery2022houghtonlike}), and different replacement graphs for the color black would provide generalizations of both the $H_n$'s and the groups $QV$, $\tilde{Q}V$, $QT$ and $QF$.
Many variations of these groups could be defined and studied, but this is beyond the scope of this work.

\begin{figure}\centering
\begin{subfigure}[b]{.325\textwidth}\centering
\begin{tikzpicture}[scale=.75]
    \draw[->-=.5,Green] (180:-1) node[black,circle,fill,inner sep=1.25]{} -- (180: -2) node[black,circle,fill,inner sep=1.25]{};
    \draw[->-=.5,red] (225:-1) node[black,circle,fill,inner sep=1.25]{} -- (225: -2) node[black,circle,fill,inner sep=1.25]{};
    \draw[->-=.5,Turquoise] (270:-1) node[black,circle,fill,inner sep=1.25]{} -- (270: -2) node[black,circle,fill,inner sep=1.25]{};
    \draw[dotted,gray] (315:-1) -- (315: -2);
    \draw (0:-1.5) node{$\dots$};
    \draw[dotted,gray] (45:-1) -- (45: -2);
    \draw[->-=.5,blue] (90:-1) node[black,circle,fill,inner sep=1.25]{} -- (90: -2) node[black,circle,fill,inner sep=1.25]{};
    \draw[->-=.5,Orange] (135:-1) node[black,circle,fill,inner sep=1.25]{} -- (135: -2) node[black,circle,fill,inner sep=1.25]{};
\end{tikzpicture}
\caption{The base graph.}
\end{subfigure}
\begin{subfigure}[b]{.65\textwidth}\centering
\centering
    \begin{subfigure}{.48\textwidth}
    \centering
    \begin{tikzpicture}[scale=.75]
        \draw[->-=.5] (0,0) node[black,circle,fill,inner sep=1.25]{} node[black,above]{$\iota$} -- (1.5,-.8);
        \draw[->-=.5] (1.5,-.8) node[black,circle,fill,inner sep=1.25]{} -- (3,0) node[black,circle,fill,inner sep=1.25]{};
        \draw[->-=.5] (0,0) -- (1,.75) node[black,circle,fill,inner sep=1.25]{};
        \draw[->-=.5] (1,.75) -- (2,.75) node[black,circle,fill,inner sep=1.25]{};
        \draw[->-=.5] (2,.75) -- (3,0) node[black,above]{$\tau$};
        \draw[white] (0,-1);
    \end{tikzpicture}
    \end{subfigure}
    \begin{subfigure}{.48\textwidth}
    \centering
    \begin{tikzpicture}[scale=1]
        \draw[->-=.5] (0,0) node[black,circle,fill,inner sep=1.25]{} node[black,above]{$\iota$} -- (1,0) node[black,circle,fill,inner sep=1.25]{};
        \draw[->-=.5,Green] (2,0) node[black,circle,fill,inner sep=1.25]{} -- (3,0) node[black,circle,fill,inner sep=1.25]{} node[black,above]{$\tau$};
        \draw[white] (0,-.75);
    \end{tikzpicture}
    \end{subfigure}
\caption{The black and the \textcolor{Green}{green} replacement graphs, respectively. The other colors are analogous to the green one.}
\end{subfigure}
\caption{Replacement systems for the Houghton groups $H_n$. There are precisely $n+1$ colors, one of which is black.}
\label{fig Houghton}
\end{figure}

\begin{figure}\centering
\begin{subfigure}[c]{\textwidth}\centering
\begin{subfigure}[b]{.275\textwidth}\centering
\begin{tikzpicture}
    \draw[->-=.5,blue] (0,0) node[black,circle,fill,inner sep=1.25]{} -- (2,0) node[black,circle,fill,inner sep=1.25]{};
    \draw[white] (0,-.75);
\end{tikzpicture}
\caption{The base graph.}
\end{subfigure}
\begin{subfigure}[b]{.7\textwidth}\centering
\centering
    \begin{subfigure}{.4\textwidth}
    \centering
    \begin{tikzpicture}[scale=.75]
        \draw[->-=.5] (0,0) node[black,circle,fill,inner sep=1.25]{} node[black,above]{$\iota$} -- (1.5,-.8);
        \draw[->-=.5] (1.5,-.8) node[black,circle,fill,inner sep=1.25]{} -- (3,0) node[black,circle,fill,inner sep=1.25]{};
        \draw[->-=.5] (0,0) -- (1,.75) node[black,circle,fill,inner sep=1.25]{};
        \draw[->-=.5] (1,.75) -- (2,.75) node[black,circle,fill,inner sep=1.25]{};
        \draw[->-=.5] (2,.75) -- (3,0) node[black,above]{$\tau$};
        \draw[white] (0,-1);
    \end{tikzpicture}
    \end{subfigure}
    \begin{subfigure}{.55\textwidth}
    \centering
    \begin{tikzpicture}[scale=1]
        \draw[->-=.5,blue] (0,0) node[black,circle,fill,inner sep=1.25]{} node[black,above]{$\iota$} -- (1,0) node[black,circle,fill,inner sep=1.25]{};
        \draw[->-=.5] (1.5,0) node[black,circle,fill,inner sep=1.25]{} -- (2.5,0) node[black,circle,fill,inner sep=1.25]{};
        \draw[->-=.5,blue] (3,0) node[black,circle,fill,inner sep=1.25]{} -- (4,0) node[black,circle,fill,inner sep=1.25]{} node[black,above]{$\tau$};
        \draw[white] (0,-.75);
    \end{tikzpicture}
    \end{subfigure}
\caption{The black and the \textcolor{blue}{blue} replacement graphs, respectively.}
\end{subfigure}
\end{subfigure}
\vspace{15pt}
\\
\begin{subfigure}[c]{\textwidth}\centering
\begin{subfigure}[b]{.275\textwidth}\centering
\begin{tikzpicture}[scale=.667]
    \draw[->-=.5,blue] (0,0) circle (1);
    \node at (1,0) [circle,fill,inner sep=1.5]{};
\end{tikzpicture}
\caption{The base graph.}
\end{subfigure}
\begin{subfigure}[b]{.7\textwidth}\centering
\centering
    \begin{subfigure}{.4\textwidth}
    \centering
    \begin{tikzpicture}[scale=.75]
        \draw[->-=.5] (0,0) node[black,circle,fill,inner sep=1.25]{} node[black,above]{$\iota$} -- (1.5,-.8);
        \draw[->-=.5] (1.5,-.8) node[black,circle,fill,inner sep=1.25]{} -- (3,0) node[black,circle,fill,inner sep=1.25]{};
        \draw[->-=.5] (0,0) -- (1,.75) node[black,circle,fill,inner sep=1.25]{};
        \draw[->-=.5] (1,.75) -- (2,.75) node[black,circle,fill,inner sep=1.25]{};
        \draw[->-=.5] (2,.75) -- (3,0) node[black,above]{$\tau$};
        \draw[white] (0,-1);
    \end{tikzpicture}
    \end{subfigure}
    \begin{subfigure}{.55\textwidth}
    \centering
    \begin{tikzpicture}[scale=1]
        \draw[->-=.5,blue] (0,0) node[black,circle,fill,inner sep=1.25]{} node[black,above]{$\iota$} -- (1.5,0) node[black,circle,fill,inner sep=1.25]{};
        \draw[->-=.5,blue] (1.5,0) -- (3,0) node[black,circle,fill,inner sep=1.25]{} node[black,above]{$\tau$};
        \draw[->-=.5] (0.75,0.75) node[black,circle,fill,inner sep=1.25]{} -- (2.25,0.75) node[black,circle,fill,inner sep=1.25]{};
        \draw[white] (0,-.333);
    \end{tikzpicture}
    \end{subfigure}
\caption{The black and the \textcolor{blue}{blue} replacement graphs, respectively.}
\end{subfigure}
\end{subfigure}
\vspace{15pt}
\\
\begin{subfigure}[c]{\textwidth}\centering
\begin{subfigure}[b]{.275\textwidth}\centering
\begin{tikzpicture}
    \draw[->-=.5,blue] (0,0) node[black,circle,fill,inner sep=1.25]{} -- (2,0) node[black,circle,fill,inner sep=1.25]{};
    \draw[white] (0,-.75);
\end{tikzpicture}
\caption{The base graph.}
\end{subfigure}
\begin{subfigure}[b]{.7\textwidth}\centering
\centering
    \begin{subfigure}{.4\textwidth}
    \centering
    \begin{tikzpicture}[scale=.75]
        \draw[->-=.5] (0,0) node[black,circle,fill,inner sep=1.25]{} node[black,above]{$\iota$} -- (1.5,-.8);
        \draw[->-=.5] (1.5,-.8) node[black,circle,fill,inner sep=1.25]{} -- (3,0) node[black,circle,fill,inner sep=1.25]{};
        \draw[->-=.5] (0,0) -- (1,.75) node[black,circle,fill,inner sep=1.25]{};
        \draw[->-=.5] (1,.75) -- (2,.75) node[black,circle,fill,inner sep=1.25]{};
        \draw[->-=.5] (2,.75) -- (3,0) node[black,above]{$\tau$};
        \draw[white] (0,-1);
    \end{tikzpicture}
    \end{subfigure}
    \begin{subfigure}{.55\textwidth}
    \centering
    \begin{tikzpicture}[scale=1]
        \draw[->-=.5,blue] (0,0) node[black,circle,fill,inner sep=1.25]{} node[black,above]{$\iota$} -- (1.5,0) node[black,circle,fill,inner sep=1.25]{};
        \draw[->-=.5,blue] (1.5,0) -- (3,0) node[black,circle,fill,inner sep=1.25]{} node[black,above]{$\tau$};
        \draw[->-=.5] (0.75,0.75) node[black,circle,fill,inner sep=1.25]{} -- (2.25,0.75) node[black,circle,fill,inner sep=1.25]{};
        \draw[white] (0,-.333);
    \end{tikzpicture}
    \end{subfigure}
\caption{The black and the \textcolor{blue}{blue} replacement graphs, respectively.}
\end{subfigure}
\end{subfigure}
\caption{From top to bottom, the three replacement systems for the groups $QV$, $QT$ and $QF$.}
\label{fig QV}
\end{figure}

Finally, topological full groups of (one-sided) subshifts of finite type (see \cite{Matui} or Definition 6.9 from \cite{belk2022type}) can be realized as rearrangement groups: each type is a color; replacement graphs consist of disjoints edges, one for every edge in the graph that defines the subshift, each colored by the terminal vertex of the edge.
However, the application of the results from this work to such groups is limited to those whose replacement rules are reduction confluent (\cref{DEF red-conf}).
Describing a class of topological full groups of subshifts of finite type that satisfiy this condition (or such that the issues caused by the lack of this condition can be circumvented as described in \cref{SUB non conf}) is left for further investigation.

\subsection{Forest Pair Diagrams}
\label{SUB FPDs}

So far we have described rearrangements using graph pair diagrams, as in \cite{belk2016rearrangement}.
Similarly to how Thompson groups can also be described by using \textit{tree pair diagrams}, here we will introduce a new way of representing rearrangements using \textit{forest pair diagrams}.
This will require the addition of labels that describe the graph structure of the graph expansions.

\subsubsection{Forest Expansions}
\label{SUB forest expansions}

We first ``translate'' the data codified by the base and the replacement graphs into rooted forests and trees that are labeled on the edges, as is explained in the following paragraphs.
As an example of this construction, the reader can refer to \cref{fig_A_forest}, which represents the Airplane replacement system from \cref{fig_replacement_A} using forests and trees.

For each graph among $X_0, X_1, \dots, X_k$, fix an ordering of its edges.
In order to avoid confusion between the edges of graph expansions and those of forest expansions, we will call the latter by \textbf{branches}.
Each of the forests that we will describe are rooted, labeled on the branches and equipped with an ordering of their roots and with a rotation system.
Recall that a \textbf{rotation system} on a graph is an assignment of a circular order to the edges incident on each vertex;
in all of our figures we will indicate the rotation system by the counterclockwise order of the edges around each vertex.

\begin{figure}
    \centering
    \begin{subfigure}[b]{\textwidth}
    \centering
    \begin{tikzpicture}[font=\small,scale=1.2]
        \draw (0,0) node[black,circle,fill,inner sep=1.25]{} -- node[red,left,align=center]{c\\b} (0,-1) node[black,circle,fill,inner sep=1.25]{};
        \draw (1.25,0) node[black,circle,fill,inner sep=1.25]{} -- node[blue,left,align=center]{b\\a} (1.25,-1) node[black,circle,fill,inner sep=1.25]{};
        \draw (2.5,0) node[black,circle,fill,inner sep=1.25]{} -- node[red,left,align=center]{b\\c} (2.5,-1) node[black,circle,fill,inner sep=1.25]{};
        \draw (3.75,0) node[black,circle,fill,inner sep=1.25]{} -- node[blue,left,align=center]{c\\d} (3.75,-1) node[black,circle,fill,inner sep=1.25]{};
        \begin{scope}[xshift=7cm,yshift=-.5cm,scale=.8]
        \draw[blue] (-0.5,0) node[black,anchor=south east]{b} -- (-2,0) node[black,above]{a} node[black,circle,fill,inner sep=1.25]{}; \draw[blue] (0.5,0) node[black,anchor=south west]{c} -- (2,0) node[black,above]{d} node[black,circle,fill,inner sep=1.25]{};
        \draw[->-=.5,red] (0.5,0) node[circle,fill,inner sep=1.25]{} to[out=90,in=90,looseness=1.7] (-0.5,0);
        \draw[->-=.5,red] (-0.5,0) node[black,circle,fill,inner sep=1.25]{} to[out=270,in=270,looseness=1.7] (0.5,0) node[black,circle,fill,inner sep=1.25]{};
        \end{scope}
    \end{tikzpicture}
    \caption{The base forest compared to the base graph}
    \label{fig_A_base_forest}
    \end{subfigure}
    \begin{subfigure}[b]{\textwidth}
    \centering
    \vspace{5pt}
    \begin{tikzpicture}[font=\small,scale=1.2]
        \draw (0,0) node[black,circle,fill,inner sep=1.25]{} -- node[red,left,align=center]{$\iota$\\$\tau$} (0,-1) node[black,circle,fill,inner sep=1.25]{};
        \draw (0,-1) to[out=180,in=90,looseness=1.2] (-1.25,-2) node[black,circle,fill,inner sep=1.25]{} node[red,xshift=-.3cm,yshift=.3cm,align=center]{$\iota$\\x};
        \draw (0,-1) -- (0,-2) node[black,circle,fill,inner sep=1.25]{} node[blue,xshift=-.3cm,yshift=.3cm,align=center]{x\\y};
        \draw (0,-1) to[out=0,in=90,looseness=1.2] (1.25,-2) node[black,circle,fill,inner sep=1.25]{} node[red,xshift=-.3cm,yshift=.3cm,align=center]{x\\$\tau$};
        \begin{scope}[xshift=4cm,yshift=-2cm]
        \draw[->-=.5,red] (0,0) node[black,circle,fill,inner sep=1.25]{} node[black,left]{$\iota$} -- (0,1) node[black,left]{x};
        \draw[->-=.5,red] (0,1) -- (0,2) node[black,circle,fill,inner sep=1.25]{} node[black,left]{$\tau$};
        \draw[blue] (0,1) node[black,circle,fill,inner sep=1.25]{} -- (1,1) node[black,circle,fill,inner sep=1.25]{} node[black,right]{y};
        \end{scope}
    \end{tikzpicture}
    \vspace{5pt}
    \\
    \begin{tikzpicture}[font=\small,scale=1.2]
        \draw (0,0) node[black,circle,fill,inner sep=1.25]{} -- node[blue,left,align=center]{$\iota$\\$\tau$} (0,-1) node[black,circle,fill,inner sep=1.25]{};
        \draw (0,-1) to[out=180,in=90,looseness=1.2] (-1.5,-2) node[black,circle,fill,inner sep=1.25]{} node[blue,xshift=-.3cm,yshift=.3cm,align=center]{x\\$\iota$};
        \draw (0,-1) to[out=195,in=90,looseness=1.2] (-.5,-2) node[black,circle,fill,inner sep=1.25]{} node[red,xshift=-.3cm,yshift=.3cm,align=center]{x\\y};
        \draw (0,-1) to[out=345,in=90,looseness=1.2] (.5,-2) node[black,circle,fill,inner sep=1.25]{} node[blue,xshift=-.3cm,yshift=.3cm,align=center]{y\\$\tau$};
        \draw (0,-1) to[out=0,in=90,looseness=1.2] (1.5,-2) node[black,circle,fill,inner sep=1.25]{} node[red,xshift=-.3cm,yshift=.3cm,align=center]{y\\x};
        \begin{scope}[xshift=4cm,yshift=-1cm,scale=.72]
        \draw[blue] (-0.5,0) node[black,anchor=south east]{x} -- (-2,0) node[black,above]{$\iota$} node[black,circle,fill,inner sep=1.25]{}; \draw[blue] (0.5,0) node[black,anchor=south west]{y} -- (2,0) node[black,above]{$\tau$} node[black,circle,fill,inner sep=1.25]{};
        \draw[->-=.5,red] (0.5,0) node[circle,fill,inner sep=1.25]{} to[out=90,in=90,looseness=1.7] (-0.5,0);
        \draw[->-=.5,red] (-0.5,0) node[black,circle,fill,inner sep=1.25]{} to[out=270,in=270,looseness=1.7] (0.5,0) node[black,circle,fill,inner sep=1.25]{};
        \end{scope}
    \end{tikzpicture}
    \caption{The replacement trees compared to the replacement graphs}
    \label{fig_A_replacement-trees}
    \end{subfigure}
    \caption{The Airplane replacement system represented with forests.}
    \label{fig_A_forest}
\end{figure}
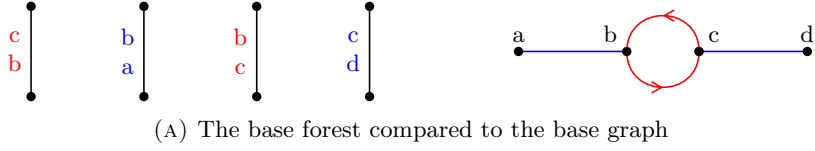
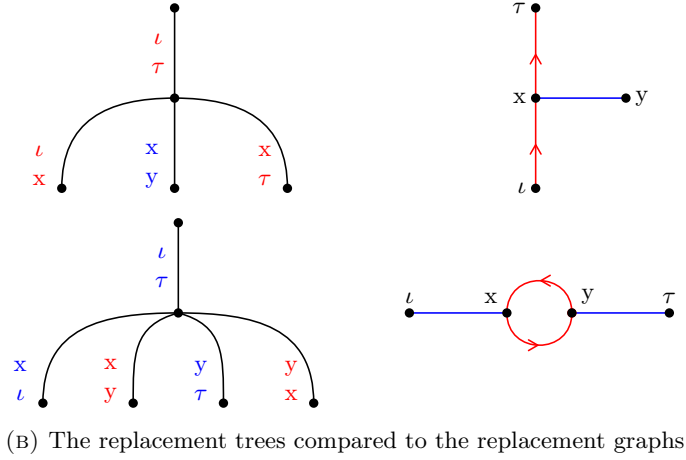

The \textbf{base forest} $F_0$ consists solely of a root for each edge of the base graph $X_0$, in their given order, and a single branch departing from each of these roots.
Each branch is colored by the color of the corresponding edge of $X_0$ and labeled by a triple $(v,w,z)$, where $v$ and $w$ are the origin and the terminus of the edge, respectively, and $z \in \mathbb{Z}^+$ distinguishes the occurrence of parallel edges.
More precisely, the first edge going from $v$ to $w$ produces the label $(v,w,1)$, the second such edge produces $(v,w,2)$ and so on.
When dealing with replacement systems whose graph expansions are devoid of parallel edges, for the sake of brevity we omit the index $z = 1$ and we simply label branches by $(v,w)$, as is the case for the Airplane replacement system in \cref{fig_A_forest}.
None of the replacement systems that have been introduced have parallel edges, so we give an example of this third index of the labeling in \cref{fig parallel edges}.

\begin{figure}
    \centering
    \begin{subfigure}[b]{.475\textwidth}
    \centering
    \begin{tikzpicture}[font=\small,scale=1]
        \draw[->-=.5] (-2,0) node[black,above]{$\iota$} node[black,circle,fill,inner sep=1.25]{} -- (-0.5,0) node[black,anchor=south east]{a} node[black,circle,fill,inner sep=1.25]{};
        \draw[->-=.5] (-0.5,0) node[circle,fill,inner sep=1.25]{} to[out=90,in=90,looseness=1.7] (0.5,0);
        \draw[->-=.5] (-0.5,0) node[black,circle,fill,inner sep=1.25]{} to[out=270,in=270,looseness=1.7] (0.5,0) node[black,circle,fill,inner sep=1.25]{};
        \draw[->-=.5] (0.5,0) node[black,anchor=south west]{b} -- (2,0) node[black,above]{$\tau$} node[black,circle,fill,inner sep=1.25]{};
    \end{tikzpicture}
    \end{subfigure}
    \begin{subfigure}[b]{.475\textwidth}
    \centering
    \begin{tikzpicture}[font=\small,scale=1.2]
        \draw (0,0) node[black,circle,fill,inner sep=1.25]{} -- node[left,align=center]{$\iota$\\$\tau$} (0,-1) node[black,circle,fill,inner sep=1.25]{};
        \draw (0,-1) to[out=180,in=90,looseness=1.2] (-1.5,-2) node[black,circle,fill,inner sep=1.25]{} node[xshift=-.3cm,yshift=.3cm,align=center]{$\iota$\\x};
        \draw (0,-1) to[out=195,in=90,looseness=1.2] (-.5,-2) node[black,circle,fill,inner sep=1.25]{} node[xshift=-.3cm,yshift=.3cm,align=center]{a\\b\\1};
        \draw (0,-1) to[out=345,in=90,looseness=1.2] (.5,-2) node[black,circle,fill,inner sep=1.25]{} node[xshift=-.3cm,yshift=.3cm,align=center]{a\\b\\2};
        \draw (0,-1) to[out=0,in=90,looseness=1.2] (1.5,-2) node[black,circle,fill,inner sep=1.25]{} node[xshift=-.3cm,yshift=.3cm,align=center]{y\\$\tau$};
    \end{tikzpicture}
    \end{subfigure}
    \caption{A graph with parallel edges and its replacement tree.}
    \label{fig parallel edges}
\end{figure}

Also, for undirected colors (as in \cref{RMK undirected}) we naturally consider $(v,w,z)$ to be the same as $(w,v,z)$.
As explained in \cref{RMK undirected}, this identification is implying a graph isomorphism on the expansion of the edges $(v,w)$ and $(w,v)$.

\phantomsection\label{TXT replacement trees}
For each color $i \in \mathrm{C}$, we call $i$-th \textbf{replacement tree} and denote by $T_i$ the tree consisting of a top branch labeled by $(\iota,\tau,\epsilon)$ (where $\epsilon$ is just a temporary symbol) which splits into a bottom branch for each edge of $X_i$, in their given order.
The top branch is colored by $i$, whereas each bottom branch is colored by the color of the corresponding edge and labeled by a triple $(v,w,z)$, where $v$ and $w$ are the origin and the terminus of the corresponding edge, respectively, and $z \in \mathbb{Z}^+$ is an index that distinguishes parallel edges, as seen earlier for the base forest.
In this labeling, we use special symbols $\iota$ and $\tau$ to denote the initial and terminal vertices of the graph $X_i$.
Every symbol that appears among the labels of $T_i$ should be thought as a temporary symbol that will be properly changed each time we produce some new expansion, as explained below.

Suppose that we want to expand an edge $e$ of the base graph $X_0$.
In the base forest $F_0$, consider the branch $l$ corresponding to the edge $e$, and suppose that $(v,w,z)$ is its label and $c$ is its color.
The \textit{forest expansion} of $F_0$ by the branch $l$ is obtained by attaching below $l$ a copy of the replacement tree $T_c$ where the symbols $\iota$, $\tau$ and $\epsilon$ have been replaced by $v$, $w$ and $z$, respectively, and every other symbol of $T_c$ has been chosen so that vertices use new symbols that do not appear in $F_0$ (otherwise we would end up denoting different vertices by the same symbol).
The forest obtained in this way is denoted by $F_0 \triangleleft l$, and it is called a \textbf{simple forest expansion} of the base forest $F_0$.
For example, \cref{fig_A_forest_expansions} depicts two simple forest expansions of the Airplane limit space.

We can inductively expand any bottom branch of any forest expansion of $F_0$ in the same way we just described:
attach the replacement tree of the same color as the branch and adjust its labels.
A sequence $F_0 \triangleleft l_1 \triangleleft \cdots \triangleleft l_m$ of simple forest expansions is called a \textbf{forest expansion} of $F_0$.

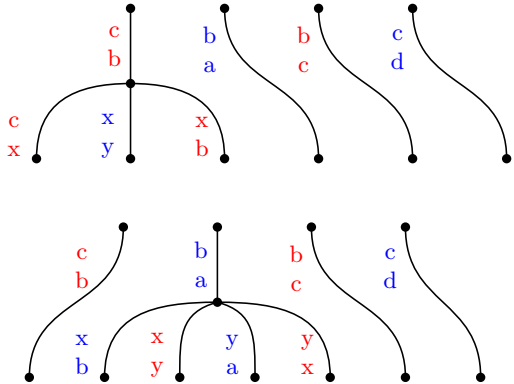
\begin{figure}
    \centering
    \begin{subfigure}[b]{\textwidth}
    \centering
    \begin{tikzpicture}[font=\small,scale=1]
        \draw (0,0) node[black,circle,fill,inner sep=1.25]{} -- node[red,left,align=center]{c\\b} (0,-1) node[black,circle,fill,inner sep=1.25]{};
        \draw (1.25,0) node[black,circle,fill,inner sep=1.25]{} node[blue,xshift=-.2cm,yshift=-.55cm,align=center]{b\\a} to[out=270,in=90,looseness=1.2] (2.5,-2) node[black,circle,fill,inner sep=1.25]{};
        \draw (2.5,0) node[black,circle,fill,inner sep=1.25]{} node[red,xshift=-.2cm,yshift=-.55cm,align=center]{b\\c} to[out=270,in=90,looseness=1.2] (3.75,-2) node[black,circle,fill,inner sep=1.25]{};
        \draw (3.75,0) node[black,circle,fill,inner sep=1.25]{} node[blue,xshift=-.2cm,yshift=-.55cm,align=center]{c\\d} to[out=270,in=90,looseness=1.2] (5,-2) node[black,circle,fill,inner sep=1.25]{};
        \draw (0,-1) to[out=180,in=90,looseness=1.2] (-1.25,-2) node[black,circle,fill,inner sep=1.25]{} node[red,xshift=-.3cm,yshift=.3cm,align=center]{c\\x};
        \draw (0,-1) -- (0,-2) node[black,circle,fill,inner sep=1.25]{} node[blue,xshift=-.3cm,yshift=.3cm,align=center]{x\\y};
        \draw (0,-1) to[out=0,in=90,looseness=1.2] (1.25,-2) node[black,circle,fill,inner sep=1.25]{} node[red,xshift=-.3cm,yshift=.3cm,align=center]{x\\b};
    \end{tikzpicture}
    \\
    \vspace{20pt}
    \begin{tikzpicture}[font=\small,scale=1]
        \draw (0,0) node[black,circle,fill,inner sep=1.25]{} node[red,xshift=-.55cm,yshift=-.55cm,align=center]{c\\b} to[out=270,in=90,looseness=1.2] (-1.25,-2) node[black,circle,fill,inner sep=1.25]{};
        \draw (1.25,0) node[black,circle,fill,inner sep=1.25]{} -- node[blue,left,align=center]{b\\a} (1.25,-1) node[black,circle,fill,inner sep=1.25]{};
        \draw (2.5,0) node[black,circle,fill,inner sep=1.25]{} node[red,xshift=-.2cm,yshift=-.55cm,align=center]{b\\c} to[out=270,in=90,looseness=1.2] (3.75,-2) node[black,circle,fill,inner sep=1.25]{};
        \draw (3.75,0) node[black,circle,fill,inner sep=1.25]{} node[blue,xshift=-.2cm,yshift=-.55cm,align=center]{c\\d} to[out=270,in=90,looseness=1.2] (4.75,-2) node[black,circle,fill,inner sep=1.25]{};
        \begin{scope}[xshift=1.25cm]
        \draw (0,-1) to[out=180,in=90,looseness=1.2] (-1.5,-2) node[black,circle,fill,inner sep=1.25]{} node[blue,xshift=-.3cm,yshift=.3cm,align=center]{x\\b};
        \draw (0,-1) to[out=195,in=90,looseness=1.2] (-.5,-2) node[black,circle,fill,inner sep=1.25]{} node[red,xshift=-.3cm,yshift=.3cm,align=center]{x\\y};
        \draw (0,-1) to[out=345,in=90,looseness=1.2] (.5,-2) node[black,circle,fill,inner sep=1.25]{} node[blue,xshift=-.3cm,yshift=.3cm,align=center]{y\\a};
        \draw (0,-1) to[out=0,in=90,looseness=1.2] (1.5,-2) node[black,circle,fill,inner sep=1.25]{} node[red,xshift=-.3cm,yshift=.3cm,align=center]{y\\x};
        \end{scope}
    \end{tikzpicture}
    \end{subfigure}
    \caption{Two simple forest expansions of the Airplane replacement system, equivalent to the graph expansions of \cref{fig_exp_A}.}
    \label{fig_A_forest_expansions}
\end{figure}

These forest expansions correspond exactly to graph expansions of the replacement system, in this sense:
the labeling of the bottom branches of $T_i$ (or $F_0$) determines uniquely the graph $X_i$ (or $X_0$), since a graph with no isolated vertices is entirely described by its edges, and each edge is specified by the ordered pair of vertices $v,w$ along with a number $z$ distinguishing parallel edges;
thus, when expanding a bottom branch labeled by $(v,w,z)$, the bottom branches of the resulting forest describe the corresponding graph expansion:
essentially, each $c$-colored bottom branch labeled by $(v,w,z)$ of a forest expansion represents the $c$-colored $z$-th edge starting from $v$ and terminating at $w$ of the graph expansion.
Since there is an obvious natural bijection between bottom branches and leaves, we will refer to the graph described in this way by the labeling of the bottom branches by the expression \textbf{leaf graph}.

\phantomsection\label{TXT renaming symbols}
Finally, suppose that two forests are the same up to a \textbf{renaming of symbols} (where by \textit{symbol} we mean the name of vertices and the index that distinguishes between parallel edges).
In this case we consider them to be equal, since they represent the same graph expansion, only with vertices being named by different symbols.
For example, in \cref{fig_A_base_forest} we could change each a to $\alpha$, b to $\beta$, c to $\gamma$ and d to $\delta$, as the data encoded in the forest would clearly be the same.

\begin{remark}
Under the point of view of forest expansions, the limit space (defined in \cref{SUB limit spaces}) can be seen as the quotient under the gluing relation of the boundary of the infinite forest obtained by expanding every branch from the base forest.
For instance, in the trivial case in which there is no gluing at all (meaning that replacement graphs consist of single disjoint edges and so the rearrangement group is a topological full group of a subshift of finite type), this gives the usual representation of the Cantor space as the boundary of a forest.
\end{remark}

\subsubsection{Forest Pair Diagrams}

It is known that a graph isomorphism is entirely described by its action on the edges of the graph, when dealing with graphs that do not have any isolated vertex.
This requirement holds for graph expansions of expanding replacement systems (\cref{DEF expanding}), hence a permutation of leaves (and thus of bottom branches) entirely describes a graph isomorphism between two graph expansions.
With this in mind, we can define a new description of rearrangements that is equivalent to that of the graph pair diagrams (which were defined at \cpageref{TXT graph pair diagrams}).

\begin{definition}
A \textbf{forest pair diagram} is a triple $(F_D, F_R, \sigma)$, where $F_D$ and $F_R$, called \textit{domain forest} and \textit{range forest}, respectively, are forest expansions of $F_0$ and $\sigma$ is a bijection from the set of leaves of $F_D$ to that of $F_R$ that is also a graph isomorphism between the leaf graphs of $F_D$ and $F_R$.
\end{definition}

An example is displayed in \cref{fig_forest_alpha_old}.
Notice that $\sigma$ is a graph isomorphism only when we suppose that blue edges are undirected (\cref{RMK undirected}), since the direction of the blue edge $(x,b)$ would otherwise be inverted by $\sigma$.
In case one wants to avoid using undirected edges, the forest pair diagram would need an expansion of $(x,b)$ in the domain forest and $(x,c)$ in the range forest, which is precisely the same difference between the graph pair diagrams of Figures \ref{fig_undirected_alpha} and \ref{fig_action_alpha}.

\begin{figure}
    \centering
    \begin{minipage}{.485\textwidth}\centering
    \begin{tikzpicture}[font=\small,scale=.875]
        \draw (0,0) node[black,circle,fill,inner sep=1.25]{} node[red,xshift=-.65cm,yshift=-.55cm,align=center]{c\\b} to[out=270,in=90,looseness=1.2] (-1.25,-2) node[black,circle,fill,inner sep=1.25]{} node[yshift=-.35cm]{$1$};
        \draw (1.25,0) node[black,circle,fill,inner sep=1.25]{} -- node[blue,left,align=center]{b\\a} (1.25,-1) node[black,circle,fill,inner sep=1.25]{};
        \draw (2.5,0) node[black,circle,fill,inner sep=1.25]{} node[red,xshift=-.2cm,yshift=-.55cm,align=center]{b\\c} to[out=270,in=90,looseness=1.2] (3.75,-2) node[black,circle,fill,inner sep=1.25]{} node[yshift=-.35cm]{$6$};
        \draw (3.75,0) node[black,circle,fill,inner sep=1.25]{} node[blue,xshift=-.2cm,yshift=-.55cm,align=center]{c\\d} to[out=270,in=90,looseness=1.2] (4.75,-2) node[black,circle,fill,inner sep=1.25]{} node[yshift=-.35cm]{$7$};
        \begin{scope}[xshift=1.25cm]
        \draw (0,-1) to[out=180,in=90,looseness=1.2] (-1.5,-2) node[black,circle,fill,inner sep=1.25]{} node[blue,xshift=-.3cm,yshift=.3cm,align=center]{x\\b} node[yshift=-.35cm]{$2$};
        \draw (0,-1) to[out=195,in=90,looseness=1.2] (-.5,-2) node[black,circle,fill,inner sep=1.25]{} node[red,xshift=-.3cm,yshift=.3cm,align=center]{x\\y} node[yshift=-.35cm]{$3$};
        \draw (0,-1) to[out=345,in=90,looseness=1.2] (.5,-2) node[black,circle,fill,inner sep=1.25]{} node[blue,xshift=-.3cm,yshift=.3cm,align=center]{y\\a} node[yshift=-.35cm]{$4$};
        \draw (0,-1) to[out=0,in=90,looseness=1.2] (1.5,-2) node[black,circle,fill,inner sep=1.25]{} node[red,xshift=-.3cm,yshift=.3cm,align=center]{y\\x} node[yshift=-.35cm]{$5$};
        \end{scope}
    \end{tikzpicture}
    \end{minipage}%
    \begin{minipage}{.03\textwidth}\centering
    \begin{tikzpicture}[scale=1]
        \draw[-to] (0,0) -- (.3,0);
    \end{tikzpicture}
    \end{minipage}%
    \begin{minipage}{.485\textwidth}\centering
    \begin{tikzpicture}[font=\small,scale=.875]
        \draw (0,0) node[black,circle,fill,inner sep=1.25]{} node[red,xshift=-.55cm,yshift=-.55cm,align=center]{c\\b} to[out=270,in=90,looseness=1.2] (-.75,-2) node[black,circle,fill,inner sep=1.25]{} node[yshift=-.35cm]{$\sigma(3)$};
        \draw (1.25,0) node[black,circle,fill,inner sep=1.25]{} node[blue,xshift=-.625cm,yshift=-.55cm,align=center]{b\\a} to[out=270,in=90,looseness=1.2] (.25,-2) node[black,circle,fill,inner sep=1.25]{} node[yshift=-.35cm]{$\sigma(4)$};
        \draw (2.5,0) node[black,circle,fill,inner sep=1.25]{} node[red,xshift=-.7cm,yshift=-.55cm,align=center]{b\\c} to[out=270,in=90,looseness=1.2] (1.25,-2) node[black,circle,fill,inner sep=1.25]{} node[yshift=-.35cm]{$\sigma(5)$};
        \draw (3.75,0) node[black,circle,fill,inner sep=1.25]{} node[blue,xshift=-.2cm,yshift=-.55cm,align=center]{c\\d} -- (3.75,-1) node[black,circle,fill,inner sep=1.25]{};
        \begin{scope}[xshift=3.75cm]
        \draw (0,-1) to[out=180,in=90,looseness=1.2] (-1.5,-2) node[black,circle,fill,inner sep=1.25]{} node[blue,xshift=-.3cm,yshift=.3cm,align=center]{x\\c}node[yshift=-.35cm]{$\sigma(2)$};
        \draw (0,-1) to[out=195,in=90,looseness=1.2] (-.5,-2) node[black,circle,fill,inner sep=1.25]{} node[red,xshift=-.3cm,yshift=.3cm,align=center]{x\\y} node[yshift=-.35cm]{$\sigma(6)$};
        \draw (0,-1) to[out=345,in=90,looseness=1.2] (.5,-2) node[black,circle,fill,inner sep=1.25]{} node[blue,xshift=-.3cm,yshift=.3cm,align=center]{y\\d}node[yshift=-.35cm]{$\sigma(7)$};
        \draw (0,-1) to[out=0,in=90,looseness=1.2] (1.5,-2) node[black,circle,fill,inner sep=1.25]{} node[red,xshift=-.3cm,yshift=.3cm,align=center]{y\\x} node[yshift=-.35cm]{$\sigma(1)$};
        \end{scope}
    \end{tikzpicture}
    \end{minipage}
    \caption{A forest pair diagram of the Airplane replacement system that is equivalent to the graph pair diagram in \cref{fig_action_alpha}.}
    \label{fig_forest_alpha_old}
\end{figure}
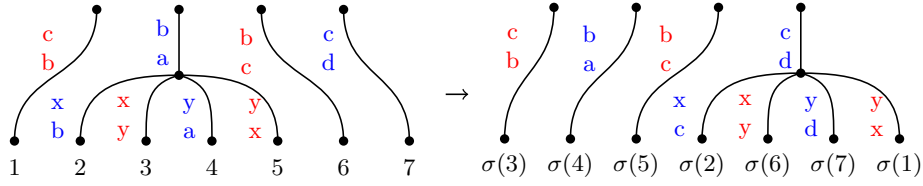

\begin{remark}
For the trio of Thompson groups $F$, $T$ and $V$ and for the Higman-Thompson groups, forest pair diagrams produce the usual tree pair diagrams:
labels here imply that the bijection between the leaves of the domain and the range must be trivial for $F$, cyclic for $T$ or any for $V$.
This suggests that most rearrangement groups sit between Thompson groups $F$ and $V$.

Indeed, every rearrangement group embeds into Thompson's group $V$.
This is because every rearrangement group naturally embeds in the rearrangement group obtained by ``ungluing'' all of the edges in the replacement system (which is, for example, a transformation that would go from $F$ to $V$ or from $T$ to $V$).
This results in a topological full group of a subshift of finite type (see \cite{Matui} or Definition 6.9 from \cite{belk2022type}).
The standard binary encoding of the subshifts of finite type (Corollary 2.12 of \cite{grigorchuk2000automata}) gives a homeomorphism from the subshift to the Cantor space that conjugates the topological full group of the subshift into $V$.
\end{remark}

\phantomsection\label{TXT FPDs convention}
Since we can rename the symbols of a forest expansion (as described at \cpageref{TXT renaming symbols}), we can always simply name each vertex in the range forest $F_R$ by the name of its preimage in $F_D$ under $\sigma$, which for example results in \cref{fig_forest_alpha} for the same element of \cref{fig_forest_alpha_old}.
Again, $(x,b)$ is being mapped to $(b,x)$, which is allowed because blue edges here are undirected (\cref{RMK undirected}).
Then the image under $\sigma$ of a leaf of $F_D$ labeled by $(v,w,z)$ is precisely the leaf of $F_R$ labeled by the same triple, and so the isomorphism $\sigma$ between the leaf graphs is described by the labeling itself and thus the forest pair diagram is entirely determined by the pair $(F_D, F_R)$.
We will always be using this simpler notation from now on, when referring to forest pair diagrams.

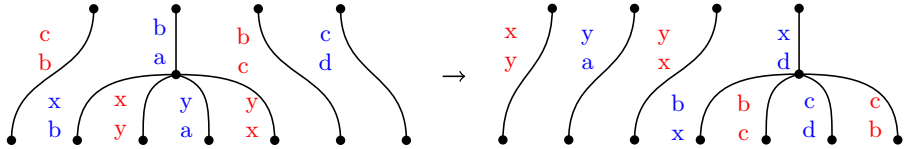
\begin{figure}
    \centering
    \begin{minipage}{.485\textwidth}\centering
    \begin{tikzpicture}[font=\small,scale=.875]
        \draw (0,0) node[black,circle,fill,inner sep=1.25]{} node[red,xshift=-.65cm,yshift=-.55cm,align=center]{c\\b} to[out=270,in=90,looseness=1.2] (-1.25,-2) node[black,circle,fill,inner sep=1.25]{};
        \draw (1.25,0) node[black,circle,fill,inner sep=1.25]{} -- node[blue,left,align=center]{b\\a} (1.25,-1) node[black,circle,fill,inner sep=1.25]{};
        \draw (2.5,0) node[black,circle,fill,inner sep=1.25]{} node[red,xshift=-.2cm,yshift=-.55cm,align=center]{b\\c} to[out=270,in=90,looseness=1.2] (3.75,-2) node[black,circle,fill,inner sep=1.25]{};
        \draw (3.75,0) node[black,circle,fill,inner sep=1.25]{} node[blue,xshift=-.2cm,yshift=-.55cm,align=center]{c\\d} to[out=270,in=90,looseness=1.2] (4.75,-2) node[black,circle,fill,inner sep=1.25]{};
        \begin{scope}[xshift=1.25cm]
        \draw (0,-1) to[out=180,in=90,looseness=1.2] (-1.5,-2) node[black,circle,fill,inner sep=1.25]{} node[blue,xshift=-.3cm,yshift=.3cm,align=center]{x\\b};
        \draw (0,-1) to[out=195,in=90,looseness=1.2] (-.5,-2) node[black,circle,fill,inner sep=1.25]{} node[red,xshift=-.3cm,yshift=.3cm,align=center]{x\\y};
        \draw (0,-1) to[out=345,in=90,looseness=1.2] (.5,-2) node[black,circle,fill,inner sep=1.25]{} node[blue,xshift=-.3cm,yshift=.3cm,align=center]{y\\a};
        \draw (0,-1) to[out=0,in=90,looseness=1.2] (1.5,-2) node[black,circle,fill,inner sep=1.25]{} node[red,xshift=-.3cm,yshift=.3cm,align=center]{y\\x};
        \end{scope}
    \end{tikzpicture}
    \end{minipage}%
    \begin{minipage}{.03\textwidth}\centering
    \begin{tikzpicture}[scale=1]
        \draw[-to] (0,0) -- (.3,0);
    \end{tikzpicture}
    \end{minipage}%
    \begin{minipage}{.485\textwidth}\centering
    \begin{tikzpicture}[font=\small,scale=.875]
        \draw (0,0) node[black,circle,fill,inner sep=1.25]{} node[red,xshift=-.55cm,yshift=-.55cm,align=center]{x\\y} to[out=270,in=90,looseness=1.2] (-.75,-2) node[black,circle,fill,inner sep=1.25]{};
        \draw (1.25,0) node[black,circle,fill,inner sep=1.25]{} node[blue,xshift=-.625cm,yshift=-.55cm,align=center]{y\\a} to[out=270,in=90,looseness=1.2] (.25,-2) node[black,circle,fill,inner sep=1.25]{};
        \draw (2.5,0) node[black,circle,fill,inner sep=1.25]{} node[red,xshift=-.7cm,yshift=-.55cm,align=center]{y\\x} to[out=270,in=90,looseness=1.2] (1.25,-2) node[black,circle,fill,inner sep=1.25]{};
        \draw (3.75,0) node[black,circle,fill,inner sep=1.25]{} node[blue,xshift=-.2cm,yshift=-.55cm,align=center]{x\\d} -- (3.75,-1) node[black,circle,fill,inner sep=1.25]{};
        \begin{scope}[xshift=3.75cm]
        \draw (0,-1) to[out=180,in=90,looseness=1.2] (-1.5,-2) node[black,circle,fill,inner sep=1.25]{} node[blue,xshift=-.3cm,yshift=.3cm,align=center]{b\\x};
        \draw (0,-1) to[out=195,in=90,looseness=1.2] (-.5,-2) node[black,circle,fill,inner sep=1.25]{} node[red,xshift=-.3cm,yshift=.3cm,align=center]{b\\c};
        \draw (0,-1) to[out=345,in=90,looseness=1.2] (.5,-2) node[black,circle,fill,inner sep=1.25]{} node[blue,xshift=-.3cm,yshift=.3cm,align=center]{c\\d};
        \draw (0,-1) to[out=0,in=90,looseness=1.2] (1.5,-2) node[black,circle,fill,inner sep=1.25]{} node[red,xshift=-.3cm,yshift=.3cm,align=center]{c\\b};
        \end{scope}
    \end{tikzpicture}
    \end{minipage}
    \caption{A forest pair diagram for the same element represented \cref{fig_forest_alpha_old} after a renaming of symbols such that the permutation $\sigma$ of the bottom branches is specified by the labeling.}
    \label{fig_forest_alpha}
\end{figure}

\phantomsection\label{TXT renaming FPDs}
If $(F_D, F_R)$ is a forest pair diagram, we can rename symbols of both $F_D$ and $F_R$ in a coherent way, by which we mean that a symbol $a$ is changed to a symbol $b$ in the domain forest $F_D$ if and only if the symbol $a$ is changed to the symbol $b$ in the range forest $F_R$.
If two forest pair diagrams differ from such a renaming of symbols, they clearly represent the same rearrangement.

\phantomsection\label{TXT reduced FPDs}
Similarly to how graph pair diagrams can be expanded (as seen at \cpageref{TXT reduced GPDs}), forest pair diagrams can also be expanded by simply expanding a bottom branch of $F_D$ and its image in $F_R$, using the same symbols for both expansions.
Again, this results in a new (and more redundant) forest pair diagram that represents the same rearrangement, and each rearrangement is represented by a unique reduced forest pair diagram.

Composition of forest pair diagrams is very similar to the usual composition of tree pair diagrams in Thompson groups, the only difference being that one must also rename the symbols of one of the two forest pair diagrams in such a way that the range forest of the first and the domain forest of the second have the same labels in their bottom branches.
Since we will only be composing strand diagrams, we will give more details about this in \cref{SUB SDs composition}, when we describe the composition of strand diagrams.

\section{Strand Diagrams and Replacement Groupoids}
\label{SEC strand diagrams}

Following the ideas of \cite{Belk2007ConjugacyAD}, in this Section we introduce strand diagrams that represent elements of a rearrangement group (\cref{SUB SDs}).
While doing so, it will be natural and useful to introduce a groupoid consisting of generalized rearrangements obtained by allowing different base graph for the domain and the range expansions, while still keeping the same replacement rules (\cref{SUB replacement groupoid}).
Although this is defined here in terms of strand diagrams, this is essentially the same groupoid that was introduced in Subsection 3.1 of \cite{belk2016rearrangement}.

Then we will introduce reduction rules to find a unique minimal reduced diagram for each element (\cref{SUB SDs reduction}), we will see how strand diagrams relate to rearrangements (\cref{SUB SDs are rearrangements}) and we will describe how to compose two strand diagrams (\cref{SUB SDs composition}).

\subsection{Generic Strand Diagrams}
\label{SUB SDs}

\phantomsection\label{TXT glue forests}
Let $(F_D, F_R)$ be a forest pair diagram.
As done in \cref{fig_strand_alpha}, draw $F_D$ and $F_R$ with the range forest turned upside down below the domain forest and join each leaf of $F_D$ to its image in $F_R$, which is the unique leaf of $F_R$ with the same label, using forest pair diagrams as discussed at \cpageref{TXT FPDs convention}.
The result is the \textbf{strand diagram} corresponding to $(F_D, F_R)$, consisting of ``strands'' starting at the top, merging and splitting in the middle and hanging at the bottom, along with labels such as those of forest expansions.
Observe that we can recover $F_D$ and $F_R$ by ``cutting'' the strand diagram in the unique way that separates the merges from the splits.

\begin{figure}\centering
    \begin{subfigure}[b]{.475\textwidth}\centering
    \begin{tikzpicture}[font=\small,scale=.875]
        \draw (0,0) node[black,circle,fill,inner sep=1.25]{} node[red,xshift=-.65cm,yshift=-.55cm,align=center]{c\\b} to[out=270,in=90,looseness=1.2] (-1.25,-2);
        \draw (1.25,0) node[black,circle,fill,inner sep=1.25]{} -- node[blue,left,align=center]{b\\a} (1.25,-1) node[black,circle,fill,inner sep=1.25]{};
        \draw (2.5,0) node[black,circle,fill,inner sep=1.25]{} node[red,xshift=-.2cm,yshift=-.55cm,align=center]{b\\c} to[out=270,in=90,looseness=1.2] (3.75,-2);
        \draw (3.75,0) node[black,circle,fill,inner sep=1.25]{} node[blue,xshift=-.2cm,yshift=-.55cm,align=center]{c\\d} to[out=270,in=90,looseness=1.2] (4.75,-2);
        \begin{scope}[xshift=1.25cm]
        \draw (0,-1) to[out=180,in=90,looseness=1.2] (-1.5,-2);
        \draw (0,-1) to[out=195,in=90,looseness=1.2] (-.5,-2);
        \draw (0,-1) to[out=345,in=90,looseness=1.2] (.5,-2);
        \draw (0,-1) to[out=0,in=90,looseness=1.2] (1.5,-2);
        \end{scope}
        \draw (-1.25,-2) to[out=270,in=90,looseness=.8] (4.75,-4);
        \draw (-.25,-2) to[out=270,in=90,looseness=1.2] (1.75,-4);
        \draw (.75,-2) to[out=270,in=90,looseness=1] (-1.25,-4);
        \draw (1.75,-2) to[out=270,in=90,looseness=1] (-.25,-4);
        \draw (2.75,-2) to[out=270,in=90,looseness=1] (.75,-4);
        \draw (3.75,-2) to[out=270,in=90,looseness=1.2] (2.75,-4);
        \draw (4.75,-2) to[out=270,in=90,looseness=1.2] (3.75,-4);
        \begin{scope}[yshift=-4cm]
        \draw (-1.25,0) to[out=270,in=90,looseness=1.2] (0,-2) node[red,xshift=-.45cm,yshift=.3cm,align=center]{x\\y} node[black,circle,fill,inner sep=1.25]{};
        \draw (-.25,0) to[out=270,in=90,looseness=1.2] (1.25,-2) node[blue,xshift=-.45cm,yshift=.3cm,align=center]{y\\a} node[black,circle,fill,inner sep=1.25]{};
        \draw (.75,0) to[out=270,in=90,looseness=1.2] (2.5,-2) node[red,xshift=-.45cm,yshift=.3cm,align=center]{y\\x} node[black,circle,fill,inner sep=1.25]{};
        \draw (3.75,-1) node[black,circle,fill,inner sep=1.25]{} node[blue,xshift=-.2cm,yshift=-.45cm,align=center]{x\\d} -- (3.75,-2) node[black,circle,fill,inner sep=1.25]{};
        \begin{scope}[xshift=3.75cm]
        \draw (0,-1) to[out=180,in=270,looseness=1.2] (-2,0) node[blue,xshift=-.25cm,yshift=0cm,align=center]{x\\b};
        \draw (0,-1) to[out=180,in=270,looseness=1.2] (-1,0);
        \draw (0,-1) to[out=90,in=270,looseness=1.2] (0,0);
        \draw (0,-1) to[out=0,in=270,looseness=1.2] (1,0);
        \end{scope}
        \end{scope}
    \end{tikzpicture}
    \caption{}
    \label{fig_strand_alpha}
    \end{subfigure}
    \begin{subfigure}[b]{.475\textwidth}\centering
    \begin{tikzpicture}[font=\small,scale=.875]
        \draw (0,0) node[black,circle,fill,inner sep=1.25]{} node[red,xshift=-.65cm,yshift=-.55cm,align=center]{p\\q} to[out=270,in=90,looseness=1.2] (-1.25,-2);
        \draw (1.25,0) node[black,circle,fill,inner sep=1.25]{} -- node[blue,left,align=center]{q\\e} (1.25,-1) node[black,circle,fill,inner sep=1.25]{};
        \draw (2.5,0) node[black,circle,fill,inner sep=1.25]{} node[red,xshift=-.2cm,yshift=-.55cm,align=center]{q\\p} to[out=270,in=90,looseness=1.2] (3.75,-2);
        \draw (3.75,0) node[black,circle,fill,inner sep=1.25]{} node[blue,xshift=-.2cm,yshift=-.55cm,align=center]{p\\f} to[out=270,in=90,looseness=1.2] (4.75,-2);
        \begin{scope}[xshift=1.25cm]
        \draw (0,-1) to[out=180,in=90,looseness=1.2] (-1.5,-2);
        \draw (0,-1) to[out=195,in=90,looseness=1.2] (-.5,-2);
        \draw (0,-1) to[out=345,in=90,looseness=1.2] (.5,-2);
        \draw (0,-1) to[out=0,in=90,looseness=1.2] (1.5,-2);
        \end{scope}
        \draw (-1.25,-2) to[out=270,in=90,looseness=.8] (4.75,-4);
        \draw (-.25,-2) to[out=270,in=90,looseness=1.2] (1.75,-4);
        \draw (.75,-2) to[out=270,in=90,looseness=1] (-1.25,-4);
        \draw (1.75,-2) to[out=270,in=90,looseness=1] (-.25,-4);
        \draw (2.75,-2) to[out=270,in=90,looseness=1] (.75,-4);
        \draw (3.75,-2) to[out=270,in=90,looseness=1.2] (2.75,-4);
        \draw (4.75,-2) to[out=270,in=90,looseness=1.2] (3.75,-4);
        \begin{scope}[yshift=-4cm]
        \draw (-1.25,0) to[out=270,in=90,looseness=1.2] (0,-2) node[red,xshift=-.45cm,yshift=.3cm,align=center]{v\\w} node[black,circle,fill,inner sep=1.25]{};
        \draw (-.25,0) to[out=270,in=90,looseness=1.2] (1.25,-2) node[blue,xshift=-.45cm,yshift=.3cm,align=center]{w\\e} node[black,circle,fill,inner sep=1.25]{};
        \draw (.75,0) to[out=270,in=90,looseness=1.2] (2.5,-2) node[red,xshift=-.45cm,yshift=.3cm,align=center]{w\\v} node[black,circle,fill,inner sep=1.25]{};
        \draw (3.75,-1) node[black,circle,fill,inner sep=1.25]{} node[blue,xshift=-.2cm,yshift=-.45cm,align=center]{v\\f} -- (3.75,-2) node[black,circle,fill,inner sep=1.25]{};
        \begin{scope}[xshift=3.75cm]
        \draw (0,-1) to[out=180,in=270,looseness=1.2] (-2,0) node[blue,xshift=-.25cm,yshift=0cm,align=center]{v\\b};
        \draw (0,-1) to[out=180,in=270,looseness=1.2] (-1,0);
        \draw (0,-1) to[out=90,in=270,looseness=1.2] (0,0);
        \draw (0,-1) to[out=0,in=270,looseness=1.2] (1,0);
        \end{scope}
        \end{scope}
    \end{tikzpicture}
    \caption{}
    \label{fig_strand_renaming}
    \end{subfigure}
    \caption{The strand diagram for the forest pair diagram of \cref{fig_forest_alpha} (A) and a renaming of its symbols (B).}
    \label{fig_strand}
\end{figure}
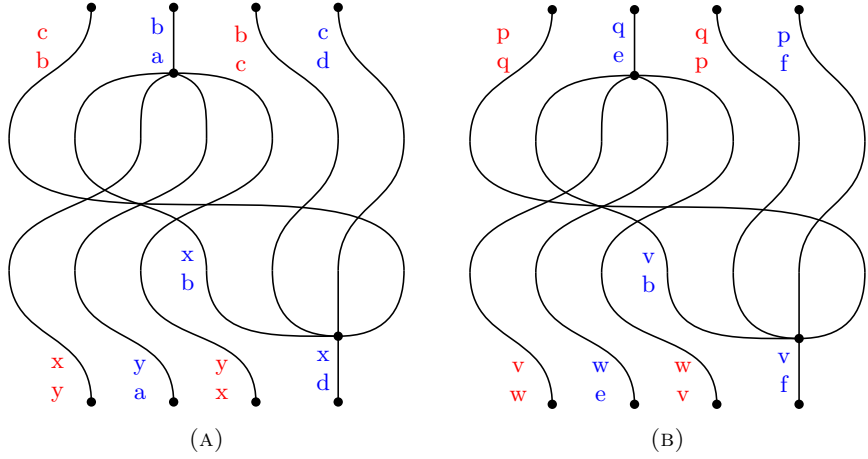

Now, these only cover those strand diagrams that are obtained by gluing forest pair diagrams.
In general, the definition of a strand diagram is the one given below.
However, we will see in \cref{SUB SDs are rearrangements} that, up to reductions (described in \cref{SUB SDs reduction}), these are really all of the strand diagrams that are needed to describe rearrangements.

\begin{definition}\label{DEF SDs}
A \textbf{strand diagram} is a finite acyclic directed graph whose edges are colored and labeled by ordered triples of symbols, together with a rotation system, such that every vertex is either a univalent source, a univalent sink, a split or a merge, and with a given ordering of both the sources and the sinks.

By \textbf{split} we mean a vertex that is the terminus of a single edge and the origin of at least two edges.
Conversely, by \textbf{merge} we mean a vertex that is the origin of a single edge and the terminus of at least two edges.
\end{definition}

In order to avoid confusion between the edges of graph expansions, those of forest expansions (which we called branches) and those of strand diagrams, we will refer to the latter by the term \textbf{strands}.

We consider two strand diagrams equal if there exists a graph isomorphism between them that is compatible with the corresponding rotation systems, and if the two strand diagrams differ by a renaming of the labeling symbols.
An example of such a renaming is depicted in \cref{fig_strand_renaming}.

As done in \cref{fig_strand}, it is convenient to depict sources at the top and sinks at the bottom, both aligned and ordered from left to right in their given order;
this allows us to hide the orientation of the strands, which is always implied to descend from the sources to the sinks.
Moreover, for the sake of clarity we color the label associated to the strand instead of the strand itself, as was done with forest pair diagrams.

The set of \textit{all} strand diagrams is far too large and varied to wield meaningful information about rearrangements (in fact, this general definition does not even take into account any information of the replacement system).
This is why we turn our attention to replacement groupoids: classes of strand diagrams determined by the replacement rules, as described in the following Subsection.

\subsection{Replacement Groupoids}
\label{SUB replacement groupoid}

If $X$ is a replacement graph, we say that a split (merge) is an $\mathbf{X}$\textbf{-split} ($\mathbf{X}$\textbf{-merge}) if, up to renaming symbols, it is a copy of the replacement tree associated to $X$ (described at \cpageref{TXT replacement trees}).
More precisely, an $X$-split consists of a top strand that splits into as many bottom strands as there are edges in $X$, in their given ordering;
if the top strand is labeled by $(\iota,\tau,\epsilon)$, then each bottom strand is labeled by $(v,w,z)$ where $v$ and $w$, respectively, are the origin and terminus of the corresponding edge of $X$, and $z$ is an index that distinguishes parallel edges, up to renaming vertices and indices, with $\iota$ and $\tau$ the initial and terminal vertices of $X$, respectively.
An $X$-merge is the same, with inverted direction of strands.

We call \textbf{branching strand} the unique top strand of a split or the unique bottom strand of a merge.
Additionally, we say that a symbol is \textbf{generated} by a split (merge) if it appears among the symbols included in the split (merge) excluding the branching strand, i.e., if it represents a new vertex in the graph expansion.

\begin{definition}\label{DEF R-branching}
Let $(R, \mathrm{C})$ be a set of replacement rules (\cref{DEF replacement}), with $R = \{ X_i \mid i \in \mathrm{C} \}$.
We say that a strand diagram is \textbf{$\bm{R}$-branching} if:
\begin{enumerate}
    \item every split and every merge is an $X_i$-split or an $X_i$-merge for some $i \in \mathrm{C}$;
    \item whenever a sequence of $k$ (possibly $k=1$) merges is immediately followed by a sequence of $k$ splits that mirror each other, the branching strands of both sequences are labeled exactly in the same way, in the same order (more precisely, by sequences that ``mirror'' each other we mean that the sequence of merges is followed by a sequence of splits that is precisely its inverse if read without labels);
    \item the same symbol cannot be generated by different splits whose branching strands have different labels;
    likewise, the same symbol cannot be generated by different merges whose branching strands have different labels.
\end{enumerate}
\end{definition}

These conditions are invariant under renaming symbols, so this definition makes sense.
\cref{fig_notRbranching} depicts examples of strand diagrams that are not $R$-branching.

\begin{figure}\centering
    \begin{subfigure}[t]{.33\textwidth}\centering
    \begin{tikzpicture}[font=\small,scale=1]
        \draw (0,0) to[out=180,in=270,looseness=1.2] (-1.25,1) node[black,circle,fill,inner sep=1.25]{} node[red,xshift=-.3cm,yshift=-.3cm,align=center]{v\\a};
        \draw (0,0) -- (0,1) node[black,circle,fill,inner sep=1.25]{} node[blue,xshift=-.3cm,yshift=-.3cm,align=center]{a\\b};
        \draw (0,0) to[out=0,in=270,looseness=1.2] (1.25,1) node[black,circle,fill,inner sep=1.25]{} node[red,xshift=-.3cm,yshift=-.3cm,align=center]{a\\w};
        \draw (0,0) node[black,circle,fill,inner sep=1.25]{} -- node[red,left,align=center]{v\\w} (0,-1) node[black,circle,fill,inner sep=1.25]{};
        \draw (0,-1) to[out=180,in=90,looseness=1.2] (-1.25,-2) node[black,circle,fill,inner sep=1.25]{} node[red,xshift=-.3cm,yshift=.3cm,align=center]{v\\x};
        \draw (0,-1) -- (0,-2) node[black,circle,fill,inner sep=1.25]{} node[blue,xshift=-.3cm,yshift=.3cm,align=center]{x\\y};
        \draw (0,-1) to[out=0,in=90,looseness=1.2] (1.25,-2) node[black,circle,fill,inner sep=1.25]{} node[red,xshift=-.3cm,yshift=.3cm,align=center]{x\\w};
    \end{tikzpicture}
    \caption{Condition (2) does not hold.}
    \label{fig_notRbranching_2}
    \end{subfigure}
    \hfill
    \begin{subfigure}[t]{.57\textwidth}\centering
    \begin{tikzpicture}[font=\small,scale=1]
        \draw (0,0) node[black,circle,fill,inner sep=1.25]{} -- node[red,left,align=center]{x\\y} (0,-1) node[black,circle,fill,inner sep=1.25]{};
        \draw (0,-1) to[out=180,in=90,looseness=1.2] (-1,-2) node[black,circle,fill,inner sep=1.25]{} node[red,xshift=-.3cm,yshift=.3cm,align=center]{x\\a};
        \draw (0,-1) -- (0,-2) node[black,circle,fill,inner sep=1.25]{} node[blue,xshift=-.3cm,yshift=.3cm,align=center]{a\\b};
        \draw (0,-1) to[out=0,in=90,looseness=1.2] (1,-2) node[black,circle,fill,inner sep=1.25]{} node[red,xshift=-.3cm,yshift=.3cm,align=center]{a\\y};
        \begin{scope}[xshift=3.5cm]
            \draw (0,0) node[black,circle,fill,inner sep=1.25]{} -- node[blue,left,align=center]{v\\w} (0,-1) node[black,circle,fill,inner sep=1.25]{};
            \draw (0,-1) to[out=180,in=90,looseness=1.2] (-1.5,-2) node[black,circle,fill,inner sep=1.25]{} node[blue,xshift=-.3cm,yshift=.3cm,align=center]{b\\v};
            \draw (0,-1) to[out=195,in=90,looseness=1.2] (-.5,-2) node[black,circle,fill,inner sep=1.25]{} node[red,xshift=-.3cm,yshift=.3cm,align=center]{b\\c};
            \draw (0,-1) to[out=345,in=90,looseness=1.2] (.5,-2) node[black,circle,fill,inner sep=1.25]{} node[blue,xshift=-.3cm,yshift=.3cm,align=center]{c\\w};
            \draw (0,-1) to[out=0,in=90,looseness=1.2] (1.5,-2) node[black,circle,fill,inner sep=1.25]{} node[red,xshift=-.3cm,yshift=.3cm,align=center]{c\\b};
        \end{scope}
    \end{tikzpicture}
    \caption{Condition (3) does not hold because the symbol b appears on two different splits.}
    \label{fig_notRbranching_3}
    \end{subfigure}
    \caption{Strand diagrams that are not $R$-branching, using the replacement rules of the Airplane replacement system.}
    \label{fig_notRbranching}
\end{figure}
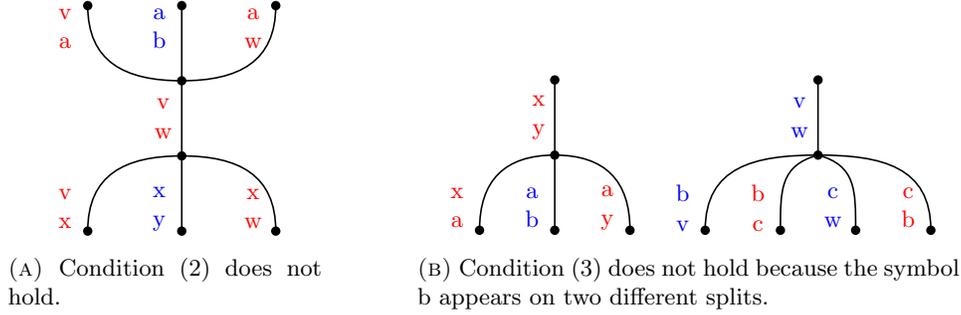

\begin{remark}
Although these conditions may seem artificial at first, they arise from the following observations, while keeping in mind that splits and merges correspond to expansions in the domain and in the range, respectively.
\begin{enumerate}
    \item The first condition essentially states that splits and merges are shaped and labeled like the branchings of forest expansions based on the replacement rules $R$;
    \item The second condition ensures that Type 2 reductions, which will be described in the following \cref{SUB SDs reduction}, can be performed whenever merges are followed by splits (this situation does not arise when pasting the domain and the range of a forest pair diagram as explained at \cpageref{TXT glue forests});
    \item The third condition tells us that, in both the underlying domain and range graph expansions, vertices generated by an edge expansion have distinct names that are not used for other vertices.
\end{enumerate}
These conditions are verified in strand diagrams obtained by gluing together the domain and the range forests of a forest pair diagrams as explained at \cpageref{TXT glue forests}.
Thus strand diagrams obtained by gluing the two forests of a a forest pair diagram based on the replacement rules $(R, \mathrm{C})$ produces an $R$-branching strand diagram.
\end{remark}

In \cref{LEM cut SDs} and the discussion that follows it we will see that $R$-branching strand diagrams, once they have been reduced (as explained later in the next \cref{SUB SDs reduction}), can be ``cut'' in a unique way that returns a forest pair diagram.
This means that the family of $R$-branching strand diagrams is really the one that we should be looking into in order to study rearrangement groups.

\begin{definition}
The \textbf{replacement groupoid} associated to the set of replacement rules $(R, \mathrm{C})$ is the set of all $R$-branching strand diagrams.
\end{definition}

The reasons as to why this is a groupoid are given later in \cref{SUB SDs composition}, when we define the composition of $R$-branching strand diagrams.

\subsection{Reductions of Strand Diagrams}
\label{SUB SDs reduction}

An $X$-merge being followed or preceded by a properly aligned $X$-split produce a reduction of the strand diagram in the following way.

\begin{definition}
A \textbf{reduction} of an $R$-branching strand diagram is either of the two types of moves shown in \cref{fig_SD_reductions}.
An $R$-branching strand diagram is \textbf{reduced} if no reduction can be performed on it.
Two $R$-branching strand diagrams are \textbf{equivalent} if one can be obtained from the other by a sequence of reductions and inverse reductions.
\end{definition}

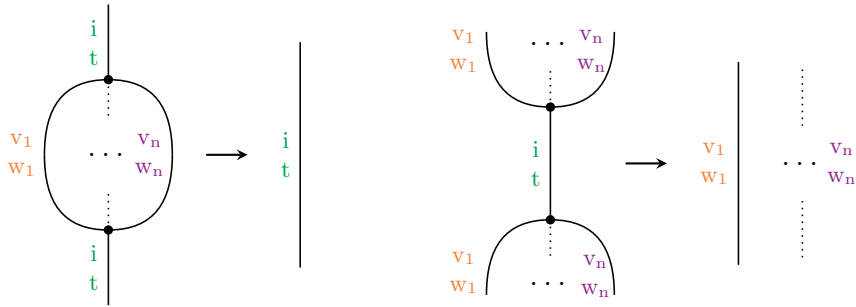
\begin{figure}\centering
    \begin{subfigure}[t]{.45\textwidth}\centering
    \begin{tikzpicture}[font=\small]
        \draw (0,0) -- node[Green,left,align=center]{i\\t} (0,-1) node[black,circle,fill,inner sep=1.25]{};
        \draw (0,-1) to[out=180,in=90,looseness=1.2] (-.85,-2) node[Orange,xshift=-.3cm,yshift=0cm,align=center]{v\textsubscript{1}\\w\textsubscript{1}};
        \draw[dotted] (0,-1) -- (0,-1.5);
        \draw (0,-1) to[out=0,in=90,looseness=1.2] (.85,-2) node[Plum,xshift=-.3cm,yshift=0cm,align=center]{v\textsubscript{n}\\w\textsubscript{n}};
        \draw (0,-2) node{\Large$\dots$};
        \draw (0,-3) to[out=180,in=270,looseness=1.2] (-.85,-2);
        \draw[dotted] (0,-3) -- (0,-2.5);
        \draw (0,-3) to[out=0,in=270,looseness=1.2] (.85,-2);
        \draw (0,-4) -- node[Green,left,align=center]{i\\t} (0,-3) node[black,circle,fill,inner sep=1.25]{};
        \draw[thick,-stealth] (1.3,-2) -- (1.85,-2);
        \draw (2.55,-.5) -- node[Green,left,align=center]{i\\t} (2.55,-3.5);
    \end{tikzpicture}
    \label{fig_SD_red1}
    \caption{Type 1 reduction: an $X$-split on top of an $X$-merge produces a single strand.}
    \end{subfigure}
    \hfill
    \begin{subfigure}[t]{.5\textwidth}\centering
    \begin{tikzpicture}[font=\small]
        \draw (0,-1) to[out=180,in=270,looseness=1.2] (-.85,0)  node[Orange,xshift=-.3cm,yshift=-.25cm,align=center]{v\textsubscript{1}\\w\textsubscript{1}};
        \draw[dotted] (0,-1) -- (0,-.5);
        \draw (0,-1) to[out=0,in=270,looseness=1.2] (.85,0)  node[Plum,xshift=-.3cm,yshift=-.25cm,align=center]{v\textsubscript{n}\\w\textsubscript{n}};
        \draw (0,-.15) node{\Large$\dots$};
        \draw (0,-2.5) node[black,circle,fill,inner sep=1.25]{} -- node[Green,left,align=center]{i\\t} (0,-1) node[black,circle,fill,inner sep=1.25]{};
        \draw (0,-2.5) to[out=180,in=90,looseness=1.2] (-.85,-3.5) node[Orange,xshift=-.3cm,yshift=.3cm,align=center]{v\textsubscript{1}\\w\textsubscript{1}};
        \draw[dotted] (0,-2.5) -- (0,-3);
        \draw (0,-2.5) to[out=0,in=90,looseness=1.2] (.85,-3.5) node[Plum,xshift=-.25cm,yshift=.25cm,align=center]{v\textsubscript{n}\\w\textsubscript{n}};
        \draw (0,-3.35) node{\Large$\dots$};
        \draw[thick,-stealth] (1,-1.75) -- (1.55,-1.75);
        \draw (2.5,-.4) -- node[Orange,left,align=center]{v\textsubscript{1}\\w\textsubscript{1}} (2.5,-3.1);
        \draw[dotted] (3.35,-.5) -- (3.35,-1.25);
        \draw (3.35,-1.75) node{\Large$\dots$};
        \draw[dotted] (3.35,-3) -- (3.35,-2.25);
        \draw (4.2,-.4) -- node[Plum,left,align=center]{v\textsubscript{n}\\w\textsubscript{n}} (4.2,-3.1);
    \end{tikzpicture}
    \caption{Type 2 reduction: an $X$-merge on top of an $X$-split produces multiple strands.}
    \label{fig_SD_red2}
    \end{subfigure}
    \caption{A schematic depiction of reductions of $R$-branching strand diagrams. Each strand should be labeled and colored according to the replacement rules.}
    \label{fig_SD_reductions}
\end{figure}

Observe that the reduction of an $R$-branching strand diagram is an $R$-branching strand diagram.
Also note that reductions of a forest pair diagram (discussed at \cpageref{TXT reduced FPDs}) correspond to Type 1 reductions of the associated strand diagram.
Instead, Type 2 reductions cannot appear when gluing together forests from a forest pair diagram (as seen at \cpageref{TXT glue forests}), but they often emerge in the composition of $R$-branching strand diagrams, which is described later in \cref{SUB SDs composition}.

\begin{lemma}
\label{LEM reduced SD}
Every $R$-branching strand diagram is equivalent to a unique reduced $R$-branching strand diagram.
\end{lemma}

\begin{proof}
Using the standard argument, by Newman's Diamond Lemma from \cite{NewmanDiamond} it suffices to prove that the directed graph whose set of vertices is the set of $R$-branching strand diagrams and whose edges are reductions satisfies the two following properties:
\begin{itemize}
    \item it is \textbf{terminating}, i.e., there is no infinite directed path (or, equivalently, there is no infinite chain of reductions of $R$-branching strand diagrams);
    \item it is \textbf{locally confluent}, i.e., if $f \longrightarrow g$ and $f \longrightarrow h$ are two edges (which correspond to reductions of the same $R$-branching strand diagram $f$), then there exist finite possibly empty directed paths from $g$ and from $h$ to a common vertex (which correspond to finite sequences of reductions from $g$ and $h$ to a common $R$-branching strand diagram).
\end{itemize}

Since each reduction strictly decreases the number of strands in a diagram, it is clear that the directed graph is terminating, so we only need to prove the local confluence.

Suppose $f, g$ and $h$ are $R$-branching strand diagrams such that $f \overset{A}{\longrightarrow} g$ and $f \overset{B}{\longrightarrow} h$ are distinct reductions.
If both $A$ and $B$ are reductions of the same type, then they must be disjoint, by which we mean that they concern different splits and merges of $f$.
In this case it is clear that the reduction $B$ can be applied to $g$ and the reduction $A$ can be applied to $h$, producing the same $R$-branching strand diagram.
If instead the reductions $A$ and $B$ are of different type, say that $A$ is of Type 1 and $B$ is of Type 2, then they are either disjoint, in which case we are done for the same reason we just discussed, or we are in one of the situations described in \cref{fig_strand_diamond}.
Observe that, in either of the two situations, whether we decide to perform the reduction $A$ or $B$, the resulting strand diagram is the same, i.e., $g = h$, and so we are done.
\end{proof}

\begin{figure}\centering
    \begin{subfigure}[t]{.4\textwidth}\centering
    \begin{tikzpicture}[font=\small,scale=1]
        \draw (0,0) to[out=180,in=270,looseness=1.2] (-.85,1) node[Orange,xshift=-.3cm,yshift=-.3cm,align=center]{v\textsubscript{1}\\w\textsubscript{1}};
        \draw[dotted] (0,0) -- (0,.5);
        \draw (0,.85) node{\Large$\dots$};
        \draw (0,0) to[out=0,in=270,looseness=1.2] (.85,1) node[Plum,xshift=-.3cm,yshift=-.3cm,align=center]{v\textsubscript{1}\\w\textsubscript{1}};
        \draw (0,0) node[black,circle,fill,inner sep=1.25]{} -- node[Green,left,align=center]{i\\t} (0,-1) node[black,circle,fill,inner sep=1.25]{};
        \draw (0,-1) to[out=180,in=90,looseness=1.2] (-.85,-2) node[Orange,xshift=-.3cm,yshift=0cm,align=center]{v\textsubscript{1}\\w\textsubscript{1}};
        \draw[dotted] (0,-1) -- (0,-1.5);
        \draw (0,-1) to[out=0,in=90,looseness=1.2] (.85,-2) node[Plum,xshift=-.3cm,yshift=0cm,align=center]{v\textsubscript{n}\\w\textsubscript{n}};
        \draw (0,-2) node{\Large$\dots$};
        \draw (0,-3) to[out=180,in=270,looseness=1.2] (-.85,-2);
        \draw[dotted] (0,-3) -- (0,-2.5);
        \draw (0,-3) to[out=0,in=270,looseness=1.2] (.85,-2);
        \draw (0,-4) -- node[Green,left,align=center]{i\\t} (0,-3) node[black,circle,fill,inner sep=1.25]{};
    \end{tikzpicture}
    \end{subfigure}
    \begin{subfigure}[t]{.4\textwidth}\centering
    \begin{tikzpicture}[font=\small,scale=1]
        \draw (0,0) -- node[Green,left,align=center]{i\\t} (0,-1) node[black,circle,fill,inner sep=1.25]{};
        \draw (0,-1) to[out=180,in=90,looseness=1.2] (-.85,-2) node[Orange,xshift=-.3cm,yshift=0cm,align=center]{v\textsubscript{1}\\w\textsubscript{1}};
        \draw[dotted] (0,-1) -- (0,-1.5);
        \draw (0,-1) to[out=0,in=90,looseness=1.2] (.85,-2) node[Plum,xshift=-.3cm,yshift=0cm,align=center]{v\textsubscript{n}\\w\textsubscript{n}};
        \draw (0,-2) node{\Large$\dots$};
        \draw (0,-3) to[out=180,in=270,looseness=1.2] (-.85,-2);
        \draw[dotted] (0,-3) -- (0,-2.5);
        \draw (0,-3) to[out=0,in=270,looseness=1.2] (.85,-2);
        \draw (0,-4) node[black,circle,fill,inner sep=1.25]{} -- node[Green,left,align=center]{i\\t} (0,-3) node[black,circle,fill,inner sep=1.25]{};
        \draw (0,-4) to[out=180,in=90,looseness=1.2] (-.85,-5) node[Orange,xshift=-.3cm,yshift=.3cm,align=center]{v\textsubscript{1}\\w\textsubscript{1}};
        \draw[dotted] (0,-4) -- (0,-4.5);
        \draw (0,-4.85) node{\Large$\dots$};
        \draw (0,-4) to[out=0,in=90,looseness=1.2] (.85,-5) node[Plum,xshift=-.3cm,yshift=.3cm,align=center]{v\textsubscript{n}\\w\textsubscript{n}};
    \end{tikzpicture}
    \end{subfigure}
    \caption{Strand diagrams where two non-disjoint reductions of Type 1 and 2 are possible.}
    \label{fig_strand_diamond}
\end{figure}
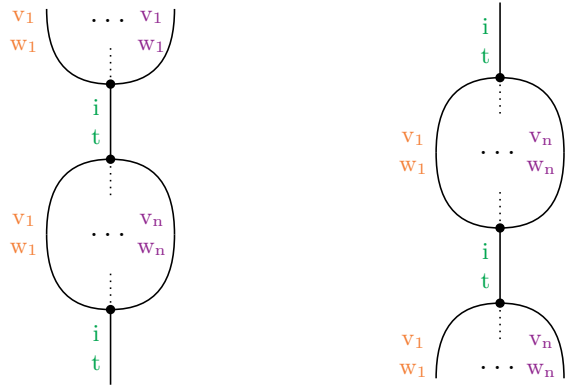

\subsection{Rearrangements as Strand Diagrams}
\label{SUB SDs are rearrangements}

In this Subsection we will see that the replacement groupoid associated to $R$ contains every rearrangement of any replacement system based on the same replacement rules $(R, \mathrm{C})$, independently from the base graph.
After describing the composition of $R$-branching strand diagrams in \cref{SUB SDs composition}, this will imply that every rearrangement group is a subgroupoid of the replacement groupoid associated to its replacement rules.

\begin{lemma}\label{LEM cut SDs}
Each reduced $R$-branching strand diagram can be cut in a unique way such that the two resulting parts of the diagram are two forest expansions of the replacement systems $(A_0, R, \mathrm{C})$ and $(B_0, R, \mathrm{C})$, where $A_0$ and $B_0$ are the base graphs corresponding to the labelings of the sources and sinks, respectively.
Moreover, there is a unique way of gluing them back together into the original strand diagram, and this is described by a unique bijection between the leaves of the two forest expansions that preserves labels, meaning that this permutation is a graph isomorphism between the leaf graph of the domain forest and the one of the range forest.
\end{lemma}

\begin{proof}
Let $f$ be a reduced $R$-branching strand diagram.
Consider the family of all directed paths starting from a source and ending at a sink of the diagram.
We claim that each of these paths contains a unique strand $s$ such that the strands preceding $s$ in the path all belong to splits and not merges, whereas the strands following $s$ all belong to merges and not splits.
Indeed, suppose that some such path does not contain any such strand $s$.
Then there is some strand that is preceded by a merge and followed by a split.
Because of the second requirement of the definition of $R$-branching (\cref{DEF R-branching}), these merge and split must be labeled in the same way.
This is a contradiction, because it gives place to a reduction of type 2, but we required $f$ to be reduced.

The uniqueness of such $s$ is evident: the ``cut'' consists of severing each of the unique strands $s$ we just found.
This results in two forest expansions as in the statement of the Lemma: each of the two forests is shaped and labeled like a forest expansion because of the first and third requirement of the definition of $R$-branching (\cref{DEF R-branching}).
\end{proof}

This Lemma essentially tells us that each $R$-branching strand diagram corresponds to a generalized rearrangement between expansions obtained using the the same replacement rules $(R, \mathrm{C})$, but possibly two distinct base graphs, as we do not require the top and bottom strands to correspond to the same graph.
After being reduced, such a diagram can be cut in a unique way that produces a generalized forest pair diagram where the roots of the two forests need not represent the same base graph.

\phantomsection\label{TXT X-SDs}
Now, consider an $R$-branching strand diagram whose sources and sinks are in the same amount, with the same labels (up to renaming symbols) and colors, in the same order.
Then sources and sinks both identify the same graph $X_0$, along with the same ordering of its edges given by the ordering of sources and sinks.
Such an $R$-branching strand diagram is called an \textbf{$\bm{\mathcal{X}}$-strand diagram}, where $\mathcal{X}$ is the replacement system $(X_0, R, \mathrm{C})$.
By the discussion in the previous paragraph, it is clear that each reduced $\mathcal{X}$-strand diagram corresponds to a unique rearrangement of $\mathcal{X}$, as in this case the two graph expansions are realized starting from the same base graph.

Conversely, we previously noted that rearrangements can be represented by forest pair diagrams and we have already seen at the beginning of this Section (\cpageref{TXT glue forests}) that we can glue the domain and the range forests of an $\mathcal{X}$-forest pair diagram to obtain a strand diagram.
This results in an $\mathcal{X}$-strand diagram.
Then, given a replacement system $\mathcal{X} = (X_0, R, \mathrm{C})$, we have that, up to reductions, rearrangements of $\mathcal{X}$ correspond uniquely to $\mathcal{X}$-strand diagrams.

With the aid of \cref{LEM cut SDs}, we are now finally ready to define the composition of $R$-branching strand diagrams.

\subsection{Composition of Strand Diagrams}
\label{SUB SDs composition}

Let $f$ and $g$ be $R$-branching strand diagrams.
We can compute their composition $f \circ g$ when the following requirements are met:
\begin{enumerate}[label=(\Alph*)]
    \item the number of sources of $f$ equals the number of sinks of $g$;
    \item the sources of $f$ and the sinks of $g$, in their given order, share the same color and, up to renaming symbols, the same labels;
\end{enumerate}
Observe that these conditions mean that the graphs represented by the sources of $f$ and the sinks of $g$ are the same, up to renaming their vertices, and the ordering of their edges is the same.

When these requirements hold, we can always find a renaming of the symbols of $f$ such that the diagram obtained by gluing the sinks of $g$ with the sources of $f$ in their given order is an $R$-branching strand diagram (\cref{DEF R-branching}).
Indeed, because of requirements A and B we can rename the symbols labeling the sources of $f$ so that they match the labeling of the sinks of $g$, and then we can continue renaming the diagram for $f$ downwards from its sources in the following way.
Because of \cref{LEM cut SDs}, we can identify a ``top half'' of $f$ consisting of splits and devoid of merges.
It is clear that, once the splits of $f$ have been renamed, the renaming of its sinks will be uniquely determined, so it suffices to rename the ``top half'' of $f$.
In order to do so, starting from the sources of $f$, rename each split distinguishing between two cases depending on the label $(v,w,z)$ of the branching strand of the split:

\begin{itemize}
    \item if $(v,w,z)$ also labels some branching strand of a merge of $g$ (which must then be unique), then rename the bottom strands of the split of $f$ using the same symbols that also appear in the top strands of the merge of $g$;
    \item if $(v,w,z)$ does not label any branching strand of any merge of $g$, then rename the bottom strands of the split of $f$ using new symbols that do not appear elsewhere.
\end{itemize}
Then the composition $f \circ g$ is the strand diagram obtained by renaming $f$ as described above, gluing the sinks of $g$ with the sources of $f$ in their given order and then reducing as explained in \cref{SUB SDs reduction}.
Indeed, the strand diagram obtained by gluing $g$ and $f$ is $R$-branching (\cref{DEF R-branching}), as the first condition is trivially satisfied and the second and third hold because of the steps described above.

An example is given in \cref{fig_strand_composition}: as shown there, when the above-mentioned requirements of $f$ and $g$ are met, it is convenient to draw $g$ right above $f$ with the respective sinks and sources aligned, and it can be useful to write on the right of each strand the renamed symbols when computing the renaming of symbols.

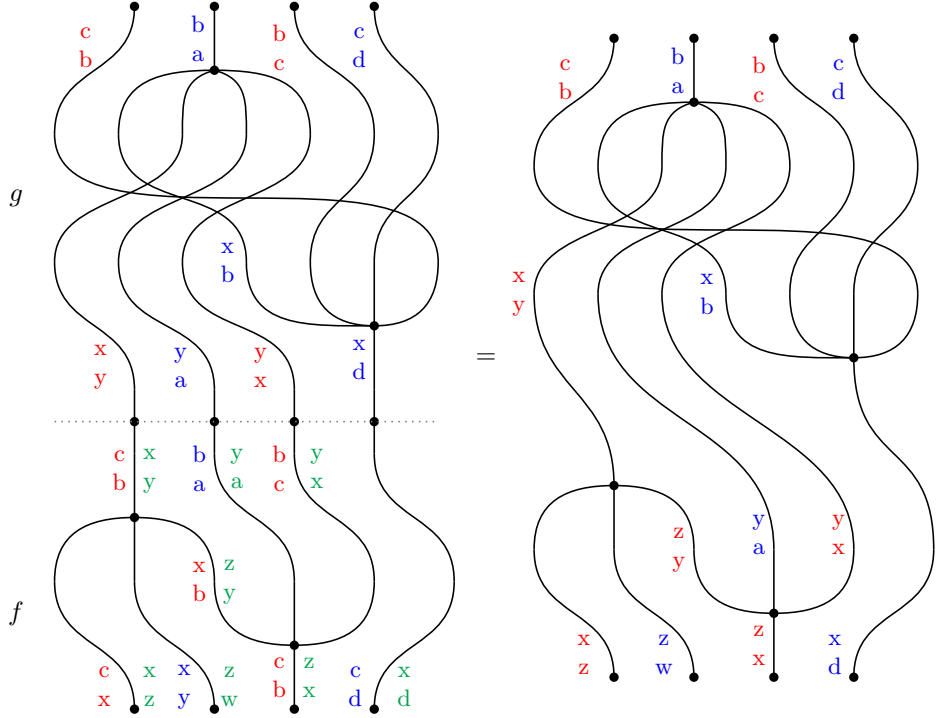
\begin{figure}\centering
\begin{tikzpicture}[font=\small,scale=.85]
\begin{scope} 
    \draw (-1.85,-3) node{\normalsize$g$};
    \draw (0,0) node[black,circle,fill,inner sep=1.25]{} node[red,xshift=-.65cm,yshift=-.55cm,align=center]{c\\b} to[out=270,in=90,looseness=1.2] (-1.25,-2);
    \draw (1.25,0) node[black,circle,fill,inner sep=1.25]{} -- node[blue,left,align=center]{b\\a} (1.25,-1) node[black,circle,fill,inner sep=1.25]{};
    \draw (2.5,0) node[black,circle,fill,inner sep=1.25]{} node[red,xshift=-.2cm,yshift=-.55cm,align=center]{b\\c} to[out=270,in=90,looseness=1.2] (3.75,-2);
    \draw (3.75,0) node[black,circle,fill,inner sep=1.25]{} node[blue,xshift=-.2cm,yshift=-.55cm,align=center]{c\\d} to[out=270,in=90,looseness=1.2] (4.75,-2);
    \begin{scope}[xshift=1.25cm]
    \draw (0,-1) to[out=180,in=90,looseness=1.2] (-1.5,-2);
    \draw (0,-1) to[out=195,in=90,looseness=1.2] (-.5,-2);
    \draw (0,-1) to[out=345,in=90,looseness=1.2] (.5,-2);
    \draw (0,-1) to[out=0,in=90,looseness=1.2] (1.5,-2);
    \end{scope}
    \draw (-1.25,-2) to[out=270,in=90,looseness=.8] (4.75,-4);
    \draw (-.25,-2) to[out=270,in=90,looseness=1.2] (1.75,-4);
    \draw (.75,-2) to[out=270,in=90,looseness=1] (-1.25,-4);
    \draw (1.75,-2) to[out=270,in=90,looseness=1] (-.25,-4);
    \draw (2.75,-2) to[out=270,in=90,looseness=1] (.75,-4);
    \draw (3.75,-2) to[out=270,in=90,looseness=1.2] (2.75,-4);
    \draw (4.75,-2) to[out=270,in=90,looseness=1.2] (3.75,-4);
    \begin{scope}[yshift=-4cm]
    \draw (-1.25,0) to[out=270,in=90,looseness=1.2] (0,-2) node[red,xshift=-.45cm,yshift=.3cm,align=center]{x\\y};
    \draw (-.25,0) to[out=270,in=90,looseness=1.2] (1.25,-2) node[blue,xshift=-.45cm,yshift=.3cm,align=center]{y\\a};
    \draw (.75,0) to[out=270,in=90,looseness=1.2] (2.5,-2) node[red,xshift=-.45cm,yshift=.3cm,align=center]{y\\x};
    \draw (3.75,-1) node[black,circle,fill,inner sep=1.25]{} node[blue,xshift=-.2cm,yshift=-.45cm,align=center]{x\\d} -- (3.75,-2);
    \begin{scope}[xshift=3.75cm]
    \draw (0,-1) to[out=180,in=270,looseness=1.2] (-2,0) node[blue,xshift=-.25cm,yshift=0cm,align=center]{x\\b};
    \draw (0,-1) to[out=180,in=270,looseness=1.2] (-1,0);
    \draw (0,-1) to[out=90,in=270,looseness=1.2] (0,0);
    \draw (0,-1) to[out=0,in=270,looseness=1.2] (1,0);
    \end{scope}
    \end{scope}
\end{scope}
    \draw (0,-6) -- node[black,circle,fill,inner sep=1.25]{} (0,-7);
    \draw (1.25,-6) -- node[black,circle,fill,inner sep=1.25]{} (1.25,-7);
    \draw (2.5,-6) -- node[black,circle,fill,inner sep=1.25]{} (2.5,-7);
    \draw (3.75,-6) -- node[black,circle,fill,inner sep=1.25]{} (3.75,-7);
    \draw[dotted,gray] (-1.25,-6.5) -- (4.75,-6.5);
\begin{scope}[yshift=-7cm] 
    \draw (-1.85,-2.5) node{\normalsize$f$};
    \draw (0,0) -- node[red,xshift=-.2cm,yshift=.2cm,align=center]{c\\b} node[Green,xshift=.2cm,yshift=.2cm,align=center]{x\\y} (0,-1) node[black,circle,fill,inner sep=1.25]{};
    \draw (1.25,0) node[blue,xshift=-.2cm,yshift=-.2cm,align=center]{b\\a} node[Green,xshift=.3cm,yshift=-.2cm,align=center]{y\\a} to[out=270,in=90,looseness=1.2] (2.5,-2);
    \draw (2.5,0) node[red,xshift=-.2cm,yshift=-.2cm,align=center]{b\\c} node[Green,xshift=.3cm,yshift=-.2cm,align=center]{y\\x} to[out=270,in=90,looseness=1.2] (3.75,-2);
    \draw (3.75,0) to[out=270,in=90,looseness=1.2] (5,-2);
    \draw (0,-1) to[out=180,in=90,looseness=1.2] (-1.25,-2);
    \draw (0,-1) -- (0,-2);
    \draw (0,-1) to[out=0,in=90,looseness=1.2] (1.25,-2) node[red,xshift=-.2cm,yshift=0cm,align=center]{x\\b} node[Green,xshift=.2cm,yshift=0cm,align=center]{z\\y};
    \begin{scope}[yshift=-2cm]
    \draw (0,-2) node[red,xshift=-.4cm,yshift=.3cm,align=center]{c\\x} node[Green,xshift=.2cm,yshift=.3cm,align=center]{x\\z} node[black,circle,fill,inner sep=1.25]{} to[out=90,in=270,looseness=1.2] (-1.25,0);
    \draw (1.25,-2) node[blue,xshift=-.4cm,yshift=.3cm,align=center]{x\\y} node[Green,xshift=.2cm,yshift=.3cm,align=center]{z\\w} node[black,circle,fill,inner sep=1.25]{} to[out=90,in=270,looseness=1.2] (0,0);
    \draw (2.5,-2) node[black,circle,fill,inner sep=1.25]{} -- node[red,xshift=-.2cm,align=center]{c\\b} node[Green,xshift=.2cm,align=center]{z\\x} (2.5,-1) node[black,circle,fill,inner sep=1.25]{};
    \draw (3.75,-2) node[black,circle,fill,inner sep=1.25]{} node[blue,xshift=-.25cm,yshift=.3cm,align=center]{c\\d} node[Green,xshift=.4cm,yshift=.3cm,align=center]{x\\d} to[out=90,in=270,looseness=1.2] (5,0);
    \draw (2.5,-1) to[out=180,in=270,looseness=1.2] (1.25,0);
    \draw (2.5,-1) -- (2.5,0);
    \draw (2.5,-1) to[out=0,in=270,looseness=1.2] (3.75,0);
    \end{scope}
\end{scope}
    \draw (5.5,-5.5) node{\normalsize$=$};
\begin{scope}[yshift=-.5cm,xshift=7.5cm] 
    \draw (0,0) node[black,circle,fill,inner sep=1.25]{} node[red,xshift=-.65cm,yshift=-.55cm,align=center]{c\\b} to[out=270,in=90,looseness=1.2] (-1.25,-2);
    \draw (1.25,0) node[black,circle,fill,inner sep=1.25]{} -- node[blue,left,align=center]{b\\a} (1.25,-1) node[black,circle,fill,inner sep=1.25]{};
    \draw (2.5,0) node[black,circle,fill,inner sep=1.25]{} node[red,xshift=-.2cm,yshift=-.55cm,align=center]{b\\c} to[out=270,in=90,looseness=1.2] (3.75,-2);
    \draw (3.75,0) node[black,circle,fill,inner sep=1.25]{} node[blue,xshift=-.2cm,yshift=-.55cm,align=center]{c\\d} to[out=270,in=90,looseness=1.2] (4.75,-2);
    \begin{scope}[xshift=1.25cm]
    \draw (0,-1) to[out=180,in=90,looseness=1.2] (-1.5,-2);
    \draw (0,-1) to[out=195,in=90,looseness=1.2] (-.5,-2);
    \draw (0,-1) to[out=345,in=90,looseness=1.2] (.5,-2);
    \draw (0,-1) to[out=0,in=90,looseness=1.2] (1.5,-2);
    \end{scope}
    \draw (-1.25,-2) to[out=270,in=90,looseness=.8] (4.75,-4);
    \draw (-.25,-2) to[out=270,in=90,looseness=1.2] (1.75,-4);
    \draw (.75,-2) to[out=270,in=90,looseness=1] (-1.25,-4) node[red,xshift=-.2cm,yshift=0cm,align=center]{x\\y};
    \draw (1.75,-2) to[out=270,in=90,looseness=1] (-.25,-4);
    \draw (2.75,-2) to[out=270,in=90,looseness=1] (.75,-4);
    \draw (3.75,-2) to[out=270,in=90,looseness=1.2] (2.75,-4);
    \draw (4.75,-2) to[out=270,in=90,looseness=1.2] (3.75,-4);
    \begin{scope}[yshift=-4cm]
    \draw (-1.25,0) to[out=270,in=90,looseness=1.2] (0,-3) node[black,circle,fill,inner sep=1.25]{};
    \draw (-.25,0) to[out=270,in=90,looseness=1.1] (2.5,-4) node[blue,xshift=-.2cm,yshift=.2cm,align=center]{y\\a};
    \draw (.75,0) to[out=270,in=90,looseness=1] (3.75,-4) node[red,xshift=-.2cm,yshift=.2cm,align=center]{y\\x};
    \draw (3.75,-1) node[black,circle,fill,inner sep=1.25]{} to[out=270,in=90,looseness=1.4] (5,-4);
    \begin{scope}[xshift=3.75cm]
    \draw (0,-1) to[out=180,in=270,looseness=1.2] (-2,0) node[blue,xshift=-.25cm,yshift=0cm,align=center]{x\\b};
    \draw (0,-1) to[out=180,in=270,looseness=1.2] (-1,0);
    \draw (0,-1) to[out=90,in=270,looseness=1.2] (0,0);
    \draw (0,-1) to[out=0,in=270,looseness=1.2] (1,0);
    \end{scope}
    \end{scope}
    \begin{scope}[yshift=-6cm]
    \draw (0,-1) to[out=180,in=90,looseness=1.2] (-1.25,-2);
    \draw (0,-1) -- (0,-2);
    \draw (0,-1) to[out=0,in=90,looseness=1.2] (1.25,-2) node[red,xshift=-.2cm,yshift=0cm,align=center]{z\\y};
    \begin{scope}[yshift=-2cm]
    \draw (0,-2) node[red,xshift=-.4cm,yshift=.3cm,align=center]{x\\z} node[black,circle,fill,inner sep=1.25]{} to[out=90,in=270,looseness=1.2] (-1.25,0);
    \draw (1.25,-2) node[blue,xshift=-.4cm,yshift=.3cm,align=center]{z\\w} node[black,circle,fill,inner sep=1.25]{} to[out=90,in=270,looseness=1.2] (0,0);
    \draw (2.5,-2) node[black,circle,fill,inner sep=1.25]{} -- node[red,xshift=-.2cm,align=center]{z\\x} (2.5,-1) node[black,circle,fill,inner sep=1.25]{};
    \draw (3.75,-2) node[black,circle,fill,inner sep=1.25]{} node[blue,xshift=-.25cm,yshift=.3cm,align=center]{x\\d} to[out=90,in=270,looseness=1.2] (5,0);
    \draw (2.5,-1) to[out=180,in=270,looseness=1.2] (1.25,0);
    \draw (2.5,-1) -- (2.5,0);
    \draw (2.5,-1) to[out=0,in=270,looseness=1.2] (3.75,0);
    \end{scope}
    \end{scope}
\end{scope}
\end{tikzpicture}
\caption{A composition $f \circ g$ of strand diagrams.}
\label{fig_strand_composition}
\end{figure}

With this definition of composition between $R$-branching strand diagrams, it is easy to check that the inverse of a diagram is the diagram obtained by reversing the direction of each strand, which makes sinks into sources and sources into sinks.
Essentially, the inverse can be computed just by drawing the original strand diagram ``upside-down''.
It is also clear that this composition is associative, thus the replacement groupoid corresponding to $R$ is really a groupoid.

\subsection{Generators of the Replacement Groupoid}
\label{SUB groupoid generators}

It is worth mentioning that, as a straightforward application of \cref{LEM cut SDs}, we can describe a generating set for the replacement groupoid.
Given a set of replacement graphs $R$, we call \textbf{split diagram} (\textbf{merge diagram}) the $R$-branching strand diagram consisting of a sole split (merge) surrounded by any amount of straight strands.
We call \textbf{permutation diagram} a strand diagram without splits or merges (the reason for this name is that the action of such a diagram is  that of a bijection between the set of sources and the set of sinks).

\begin{proposition}
The replacement groupoid associated to $R$ is generated by the infinite set consisting of every permutation diagram and every split diagram (or equivalently every merge diagram).
Moreover, each reduced $R$-branching strand diagram can be written uniquely as a product $M \circ P \circ S$, where $M$ is a product of merge diagrams, $P$ is a permutation diagram and $S$ is a product of split diagrams.
\end{proposition}

This was first noted in \cite{belk2016rearrangement}:
split, merge and permutation diagrams correspond, respectively, to the simple expansion morphisms, simple contraction morphisms and base isomorphisms of that paper.
However, seeing this in terms of strand diagrams gives an explicit and visual description of these generators which can be handy for computations.

\section{Closed Strand Diagrams}
\label{SEC closed strand diagrams}

As was done in \cite{Belk2007ConjugacyAD}, we now proceed to ``close'' our strand diagrams (\cref{SUB closed strand diagrams}) and we describe new transformations that can be performed on these closed diagrams (\cref{SUB CSD transformations}).
Then we study the possible uniqueness of reduced closed strand diagrams (\cref{SUB confluent reductions}), which will be crucial when stating a general result about the conjugacy problem in the following \cref{SEC conjugacy problem}.
Finally, in \cref{SUB stable and vanishing} we describe what happens when certain transformations move the ``base line'' of a closed diagram back to where it started and possibly produce new configurations of labels.
This lays the groundwork for the algorithm to solve the conjugacy problem in \cref{SUB algorithm}.

\subsection{Definition of Closed Diagrams}
\label{SUB closed strand diagrams}

Fix a replacement system $\mathcal{X} = (X_0, R, \mathrm{C})$ and its rearrangement group $G_\mathcal{X}$.
As we have seen in the previous Section at \cpageref{TXT X-SDs}, we can represent rearrangements as $\mathcal{X}$-strand diagrams, which are $R$-branching strand diagrams (\cref{DEF R-branching}) whose sources and sinks both represent the base graph $X_0$ of $\mathcal{X}$ (along with its given ordering of edges).
Precisely because of this last condition, we can essentially ``close'' an $\mathcal{X}$-strand diagram around a unique ``hole'' by attaching each source to a sink in their given order as done in \cref{fig_CSD}.
Labels above and below need not be the same, but they will be the same up to some renaming of symbols.
A formal definition would essentially be the same of that of a strand diagram (\cref{DEF SDs}), except that we do not require acyclicity, we do not allow univalent vertices (sources and sinks) and instead we allow the existence of special vertices of in-degree and out-degree equal to 1 originating from the gluing.

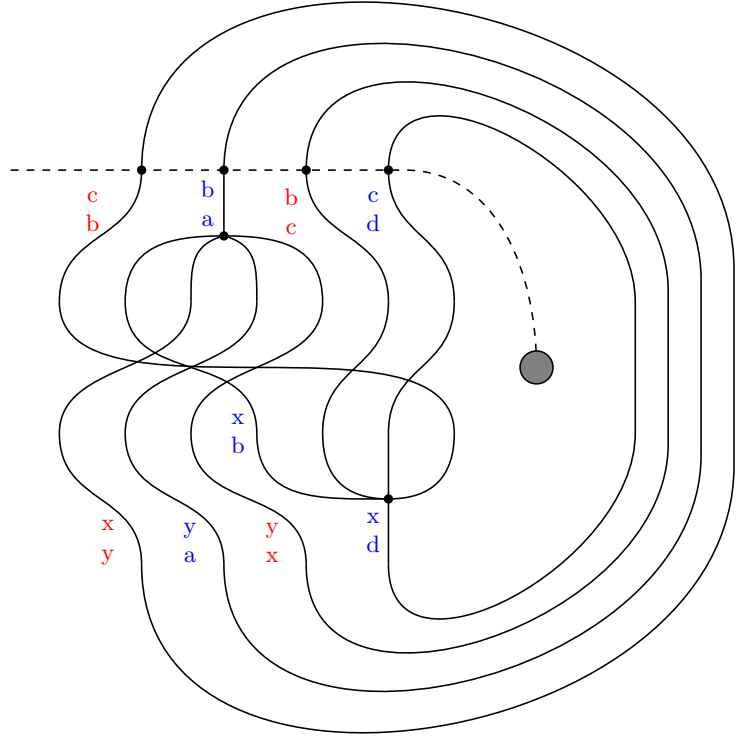
\begin{figure}\centering
\begin{tikzpicture}[font=\small,scale=.875]
    \useasboundingbox (-8,-5.5) rectangle (3,5.5);
    %
    \begin{scope}[xshift=-6cm,yshift=3cm]
    \draw (0,0) node[black,circle,fill,inner sep=1.25]{} node[red,xshift=-.65cm,yshift=-.55cm,align=center]{c\\b} to[out=270,in=90,looseness=1.2] (-1.25,-2);
    \draw (1.25,0) node[black,circle,fill,inner sep=1.25]{} -- node[blue,left,align=center]{b\\a} (1.25,-1) node[black,circle,fill,inner sep=1.25]{};
    \draw (2.5,0) node[black,circle,fill,inner sep=1.25]{} node[red,xshift=-.2cm,yshift=-.55cm,align=center]{b\\c} to[out=270,in=90,looseness=1.2] (3.75,-2);
    \draw (3.75,0) node[black,circle,fill,inner sep=1.25]{} node[blue,xshift=-.2cm,yshift=-.55cm,align=center]{c\\d} to[out=270,in=90,looseness=1.2] (4.75,-2);
    \begin{scope}[xshift=1.25cm]
    \draw (0,-1) to[out=180,in=90,looseness=1.2] (-1.5,-2);
    \draw (0,-1) to[out=195,in=90,looseness=1.2] (-.5,-2);
    \draw (0,-1) to[out=345,in=90,looseness=1.2] (.5,-2);
    \draw (0,-1) to[out=0,in=90,looseness=1.2] (1.5,-2);
    \end{scope}
    \draw (-1.25,-2) to[out=270,in=90,looseness=.8] (4.75,-4);
    \draw (-.25,-2) to[out=270,in=90,looseness=1.2] (1.75,-4);
    \draw (.75,-2) to[out=270,in=90,looseness=1] (-1.25,-4);
    \draw (1.75,-2) to[out=270,in=90,looseness=1] (-.25,-4);
    \draw (2.75,-2) to[out=270,in=90,looseness=1] (.75,-4);
    \draw (3.75,-2) to[out=270,in=90,looseness=1.2] (2.75,-4);
    \draw (4.75,-2) to[out=270,in=90,looseness=1.2] (3.75,-4);
    \begin{scope}[yshift=-4cm]
    \draw (-1.25,0) to[out=270,in=90,looseness=1.2] (0,-2) node[red,xshift=-.45cm,yshift=.3cm,align=center]{x\\y};
    \draw (-.25,0) to[out=270,in=90,looseness=1.2] (1.25,-2) node[blue,xshift=-.45cm,yshift=.3cm,align=center]{y\\a};
    \draw (.75,0) to[out=270,in=90,looseness=1.2] (2.5,-2) node[red,xshift=-.45cm,yshift=.3cm,align=center]{y\\x};
    \draw (3.75,-1) node[black,circle,fill,inner sep=1.25]{} node[blue,xshift=-.2cm,yshift=-.45cm,align=center]{x\\d} -- (3.75,-2);
    \begin{scope}[xshift=3.75cm]
    \draw (0,-1) to[out=180,in=270,looseness=1.2] (-2,0) node[blue,xshift=-.25cm,yshift=0cm,align=center]{x\\b};
    \draw (0,-1) to[out=180,in=270,looseness=1.2] (-1,0);
    \draw (0,-1) to[out=90,in=270,looseness=1.2] (0,0);
    \draw (0,-1) to[out=0,in=270,looseness=1.2] (1,0);
    \end{scope}
    \end{scope}
    \end{scope}
    \draw[dashed] (0,0) to[out=90,in=0] (-2,3) -- (-8,3);
    \draw[fill=gray] (0,0) circle (.25);
    \draw (-2.25,-3) to[out=270,in=270,looseness=1.2] (1.5,-1) -- (1.5,1) to[out=90,in=90,looseness=1.2] (-2.25,3);
    \draw (-3.5,-3) to[out=270,in=270,looseness=1.2] (2,-1.167) -- (2,1.167) to[out=90,in=90,looseness=1.2] (-3.5,3);
    \draw (-4.75,-3) to[out=270,in=270,looseness=1.2] (2.5,-1.333) -- (2.5,1.333) to[out=90,in=90,looseness=1.2] (-4.75,3);
    \draw (-6,-3) to[out=270,in=270,looseness=1.2] (3,-1.5) -- (3,1.5) to[out=90,in=90,looseness=1.2] (-6,3);
\end{tikzpicture}
\caption{The closed strand diagram obtained from the strand diagram depicted in \cref{fig_strand}. The base line is represented by the dashed line.}
\label{fig_CSD}
\end{figure}

More generally, we can close in this manner any $R$-branching strand diagram whose sources and sinks both represent isomorphic graphs, the isomorphism being given by the ordering of sources and sinks.
Diagrams obtained in this manner are called \textbf{$\bm{R}$-branching closed strand diagrams} (or simply \textit{closed strand diagrams} or \textit{closed diagrams}, if there is no ambiguity in omitting $R$).
Those vertices that were glued (originally sources and sinks) are called \textbf{base points} and the ordered tuple of gluing points is called the \textbf{base line}.
Since sinks and sources represent the same graph, it is natural to refer to that graph as the \textbf{base graph} of the closed diagram.
If $f$ is an $R$-branching strand diagram, we denote by $\llbracket f \rrbracket$ its closure.

A closed strand diagram contains the same data as the original $R$-branching strand diagram.
However, as we will see in the next \cref{SEC conjugacy problem}, conjugacy has a natural way of being represented by reductions of parallel strands, shifts of the base line and permutations of the base points.

\subsection{Transformations of Closed Diagrams}
\label{SUB CSD transformations}

Given an $R$-branching closed strand diagram, we can apply three different kinds of transformations: permutations, shifts and reductions.
Essentially, the first two allow to permute the base points and move the base line, respectively, and they will be called \textit{similarities}.
Reductions, on the other hand, include the same Type 1 and Type 2 reductions of strand diagrams defined in \cref{SUB SDs reduction} and also a new Type 3 reduction that corresponds to reducing the base graph of the rearrangement.

Each of these transformations encode some kind of (possibly trivial) modification of the base graph that corresponds some (possibly trivial) conjugation of the original rearrangement by some element of the replacement groupoid, as we will see soon.

\subsubsection{Similarities: Permutations and Shifts}

\phantomsection\label{TXT permutations}
The simplest transformation is the \textbf{permutation} of the base line, which consists simply of changing the order of the base points.
For example, \cref{fig_CSD_perm} depict the result of a permutation of the diagram portrayed in \cref{fig_CSD} (the explicit permutation is $(1 3 2) \in S_4$, where each number from $1$ to $4$ denotes a base point in their given ordering).

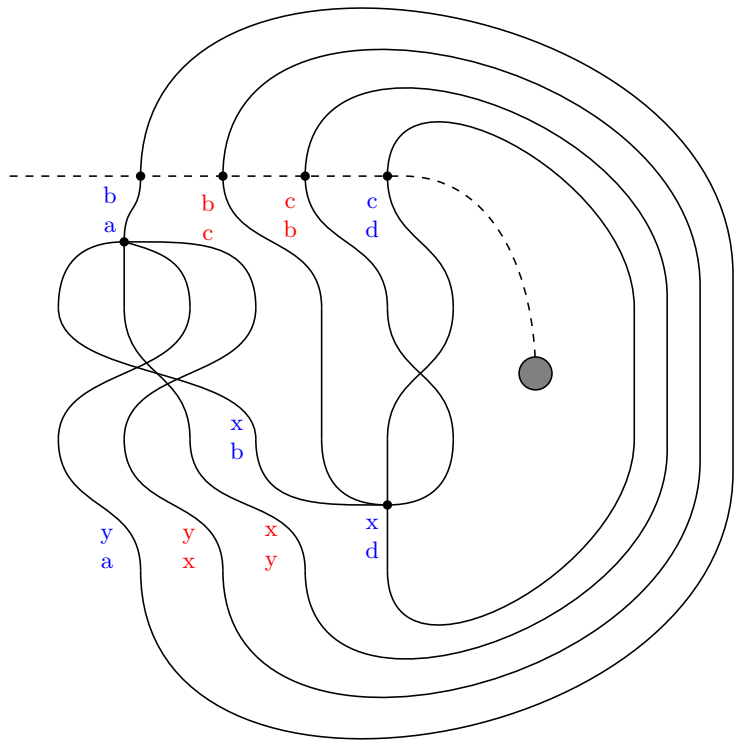
\begin{figure}\centering
\begin{tikzpicture}[font=\small,scale=.875]
    \useasboundingbox (-8,-5.5) rectangle (3,5.5);
    %
    \begin{scope}[xshift=-6cm,yshift=3cm]
    \draw (0,0) node[black,circle,fill,inner sep=1.25]{} to[out=270,in=90,looseness=1.5] node[blue,xshift=-.3cm,align=center]{b\\a} (-.25,-1) node[black,circle,fill,inner sep=1.25]{};
    \draw (1.25,0) node[black,circle,fill,inner sep=1.25]{} node[red,xshift=-.2cm,yshift=-.55cm,align=center]{b\\c} to[out=270,in=90,looseness=1.2] (2.75,-2);
    \draw (2.5,0) node[black,circle,fill,inner sep=1.25]{} node[red,xshift=-.2cm,yshift=-.55cm,align=center]{c\\b} to[out=270,in=90,looseness=1.2] (3.75,-2);
    \draw (3.75,0) node[black,circle,fill,inner sep=1.25]{} node[blue,xshift=-.2cm,yshift=-.55cm,align=center]{c\\d} to[out=270,in=90,looseness=1.2] (4.75,-2);
    \draw (-.25,-1) to[out=180,in=90,looseness=1.2] (-1.25,-2);
    \draw (-.25,-1) -- (-.25,-2);
    \draw (-.25,-1) to[out=345,in=90,looseness=1.2] (.75,-2);
    \draw (-.25,-1) to[out=0,in=90,looseness=1.2] (1.75,-2);
    \draw (-1.25,-2) to[out=270,in=90,looseness=.8] (1.75,-4);
    \draw (-.25,-2) to[out=270,in=90,looseness=1.2] (.75,-4);
    \draw (.75,-2) to[out=270,in=90,looseness=1] (-1.25,-4);
    \draw (1.75,-2) to[out=270,in=90,looseness=1] (-.25,-4);
    \draw (2.75,-2) to[out=270,in=90,looseness=1] (2.75,-4);
    \draw (3.75,-2) to[out=270,in=90,looseness=1.2] (4.75,-4);
    \draw (4.75,-2) to[out=270,in=90,looseness=1.2] (3.75,-4);
    \begin{scope}[yshift=-4cm]
    \draw (-1.25,0) to[out=270,in=90,looseness=1.2] (0,-2) node[blue,xshift=-.45cm,yshift=.3cm,align=center]{y\\a};
    \draw (-.25,0) to[out=270,in=90,looseness=1.2] (1.25,-2) node[red,xshift=-.45cm,yshift=.3cm,align=center]{y\\x};
    \draw (.75,0) to[out=270,in=90,looseness=1.2] (2.5,-2) node[red,xshift=-.45cm,yshift=.3cm,align=center]{x\\y};
    \draw (3.75,-1) node[black,circle,fill,inner sep=1.25]{} node[blue,xshift=-.2cm,yshift=-.45cm,align=center]{x\\d} -- (3.75,-2);
    \begin{scope}[xshift=3.75cm]
    \draw (0,-1) to[out=180,in=270,looseness=1.2] (-2,0) node[blue,xshift=-.25cm,yshift=0cm,align=center]{x\\b};
    \draw (0,-1) to[out=180,in=270,looseness=1.2] (-1,0);
    \draw (0,-1) to[out=90,in=270,looseness=1.2] (0,0);
    \draw (0,-1) to[out=0,in=270,looseness=1.2] (1,0);
    \end{scope}
    \end{scope}
    \end{scope}
    \draw[dashed] (0,0) to[out=90,in=0] (-2,3) -- (-8,3);
    \draw[fill=gray] (0,0) circle (.25);
    \draw (-2.25,-3) to[out=270,in=270,looseness=1.2] (1.5,-1) -- (1.5,1) to[out=90,in=90,looseness=1.2] (-2.25,3);
    \draw (-3.5,-3) to[out=270,in=270,looseness=1.2] (2,-1.167) -- (2,1.167) to[out=90,in=90,looseness=1.2] (-3.5,3);
    \draw (-4.75,-3) to[out=270,in=270,looseness=1.2] (2.5,-1.333) -- (2.5,1.333) to[out=90,in=90,looseness=1.2] (-4.75,3);
    \draw (-6,-3) to[out=270,in=270,looseness=1.2] (3,-1.5) -- (3,1.5) to[out=90,in=90,looseness=1.2] (-6,3);
\end{tikzpicture}
\caption{The diagram obtained by performing a permutation on the closed strand diagram represented in \cref{fig_CSD}.}
\label{fig_CSD_perm}
\end{figure}

Clearly permutations transform $R$-branching closed diagrams into $R$-branching closed diagrams.
Transformations of this kind do not change the base graph of the diagram, but they change the order of its edges.

As we will see in \cref{PROP conjugator}, we need this kind of transformations in order to represent conjugations by permutation diagrams (which were defined in \cref{SUB groupoid generators}).

\medskip

\phantomsection\label{TXT shifts}
A \textbf{shift} of the base line is the transformation of an $R$-branching closed strand diagram shown in \cref{fig_CSD_shift}, which essentially consists of moving the base line in such a way that it crosses exactly one split or one merge.
An example is depicted in \cref{fig_CSD_shift_ex}.
We say that a shift is \textbf{expanding} and that it \textit{expands} the split (merge) that it crosses if the number of base points is increased by the shift.
If instead the number of base points is decreased, we say that the shift is \textbf{reducing} and that it \textit{reduces} the split (merge) that it crosses.
Observe that, when a shift expands a split (merge), new symbols may need to be generated below the split (above the merge) in order to make sure that the resulting diagram is $R$-branching (\cref{DEF R-branching}), as is the case with the letters v and w in \cref{fig_CSD_shift_ex}.

\begin{figure}\centering
\begin{subfigure}{.475\textwidth}\centering
    \begin{tikzpicture}[font=\small,scale=.85]
        \draw[dashed] (-1.25,0) -- (1.25,0);
        \draw (0,1) -- node[Green,left,align=center]{a\\b} (0,0) node[black,circle,fill,inner sep=1.25]{} -- node[Green,left,align=center]{v\\w} (0,-1) node[black,circle,fill,inner sep=1.25]{};
        \draw (0,-1) to[out=180,in=90,looseness=1.2] (-.85,-2) node[Orange,xshift=-.3cm,yshift=0cm,align=center]{v\textsubscript{1}\\w\textsubscript{1}} -- (-.85,-3.5);
        \draw[dotted] (0,-1) -- (0,-1.5);
        \draw (0,-1) to[out=0,in=90,looseness=1.2] (.85,-2) node[Plum,xshift=-.3cm,yshift=0cm,align=center]{v\textsubscript{n}\\w\textsubscript{n}} -- (.85,-3.5);
        \draw (0,-2) node{\Large$\dots$};
        \draw[thick,-stealth] (1.45,-1.25) -- (2,-1.25);
        \draw[thick,-stealth] (2,-1.25) -- (1.45,-1.25);
        \begin{scope}[xshift=3.5cm]
        \draw[dashed] (-1.5,-2.5) -- (1.5,-2.5);
        \draw (0,1) -- (0,0) -- node[Green,left,align=center]{a\\b} (0,-1) node[black,circle,fill,inner sep=1.25]{};
        \draw (0,-1) to[out=180,in=90,looseness=1.2] (-.85,-2) node[Orange,xshift=-.3cm,yshift=0cm,align=center]{a\textsubscript{1}\\b\textsubscript{1}} -- (-.85,-2.5) node[black,circle,fill,inner sep=1.25]{} -- (-.85,-3.5) node[Orange,xshift=-.3cm,yshift=.4cm,align=center]{v\textsubscript{1}\\w\textsubscript{1}};
        \draw[dotted] (0,-1) -- (0,-1.5);
        \draw (0,-1) to[out=0,in=90,looseness=1.2] (.85,-2) node[Plum,xshift=-.3cm,yshift=0cm,align=center]{a\textsubscript{n}\\b\textsubscript{n}} -- (.85,-2.5) node[black,circle,fill,inner sep=1.25]{} -- (.85,-3.5) node[Plum,xshift=-.3cm,yshift=.4cm,align=center]{v\textsubscript{n}\\w\textsubscript{n}};
        \draw (0,-2) node{\Large$\dots$};
        \end{scope}
    \end{tikzpicture}
\end{subfigure}
\begin{subfigure}{.475\textwidth}\centering
    \begin{tikzpicture}[font=\small,scale=.85,yscale=-1]
        \draw[dashed] (-1.25,0) -- (1.25,0);
        \draw (0,1) -- node[Green,left,align=center]{v\\w} (0,0) node[black,circle,fill,inner sep=1.25]{} -- node[Green,left,align=center]{a\\b} (0,-1) node[black,circle,fill,inner sep=1.25]{};
        \draw (0,-1) to[out=180,in=90,looseness=1.2] (-.85,-2) node[Orange,xshift=-.3cm,yshift=0cm,align=center]{a\textsubscript{1}\\b\textsubscript{1}} -- (-.85,-3.5);
        \draw[dotted] (0,-1) -- (0,-1.5);
        \draw (0,-1) to[out=0,in=90,looseness=1.2] (.85,-2) node[Plum,xshift=-.3cm,yshift=0cm,align=center]{a\textsubscript{n}\\b\textsubscript{n}} -- (.85,-3.5);
        \draw (0,-2) node{\Large$\dots$};
        \draw[thick,-stealth] (1.45,-1.25) -- (2,-1.25);
        \draw[thick,-stealth] (2,-1.25) -- (1.45,-1.25);
        \begin{scope}[xshift=3.5cm]
        \draw[dashed] (-1.5,-2.5) -- (1.5,-2.5);
        \draw (0,1) -- (0,0) -- node[Green,left,align=center]{v\\w} (0,-1) node[black,circle,fill,inner sep=1.25]{};
        \draw (0,-1) to[out=180,in=90,looseness=1.2] (-.85,-2) node[Orange,xshift=-.3cm,yshift=0cm,align=center]{v\textsubscript{1}\\w\textsubscript{1}} -- (-.85,-2.5) node[black,circle,fill,inner sep=1.25]{} -- (-.85,-3.5) node[Orange,xshift=-.3cm,yshift=-.4cm,align=center]{a\textsubscript{1}\\b\textsubscript{1}};
        \draw[dotted] (0,-1) -- (0,-1.5);
        \draw (0,-1) to[out=0,in=90,looseness=1.2] (.85,-2) node[Plum,xshift=-.3cm,yshift=0cm,align=center]{v\textsubscript{n}\\w\textsubscript{n}} -- (.85,-2.5) node[black,circle,fill,inner sep=1.25]{} -- (.85,-3.5) node[Plum,xshift=-.3cm,yshift=-.4cm,align=center]{a\textsubscript{n}\\b\textsubscript{n}};
        \draw (0,-2) node{\Large$\dots$};
        \end{scope}
    \end{tikzpicture}
\end{subfigure}
\caption{From left to right: two expanding shifts. From right to left: two reducing shifts.}
\label{fig_CSD_shift}
\end{figure}
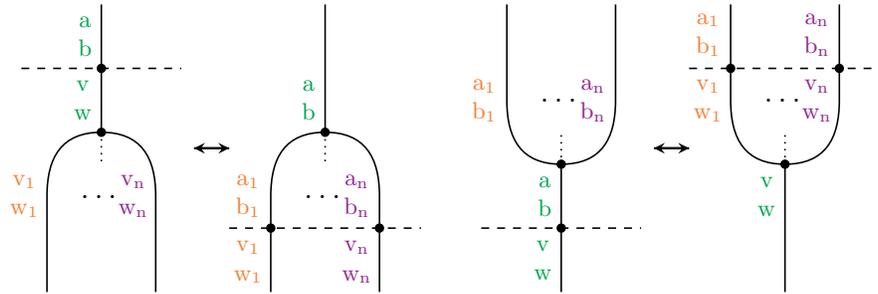

\begin{figure}\centering
\begin{tikzpicture}[font=\small,scale=.875]
    \useasboundingbox (-8,-4.5) rectangle (3,6.5);
    %
    \begin{scope}[xshift=-6cm,yshift=3cm]
    \draw (0,1) to[out=270,in=90,looseness=1.2] (-1.25,-1) node[black,circle,fill,inner sep=1.25]{};
    \draw (1.25,1) -- (1.25,0) node[black,circle,fill,inner sep=1.25]{};
    \draw (2.5,1) to[out=270,in=90,looseness=1.2] (3.75,-1) node[black,circle,fill,inner sep=1.25]{};
    \draw (3.75,1) to[out=270,in=90,looseness=1.2] (4.75,-1) node[black,circle,fill,inner sep=1.25]{};
    \begin{scope}[xshift=1.25cm]
    \draw (0,0) to[out=180,in=90,looseness=1.2] (-1.5,-1) node[black,circle,fill,inner sep=1.25]{} node[blue,xshift=-.2cm,yshift=.35cm,align=center]{v\\y};
    \draw (0,0) to[out=195,in=90,looseness=1.2] (-.5,-1) node[black,circle,fill,inner sep=1.25]{} node[red,xshift=-.2cm,yshift=.35cm,align=center]{v\\w};
    \draw (0,0) to[out=345,in=90,looseness=1.2] (.5,-1) node[black,circle,fill,inner sep=1.25]{} node[blue,xshift=-.25cm,yshift=.35cm,align=center]{w\\a};
    \draw (0,0) to[out=0,in=90,looseness=1.2] (1.5,-1) node[black,circle,fill,inner sep=1.25]{} node[red,xshift=-.35cm,yshift=.35cm,align=center]{w\\v};
    \end{scope}
    \draw (-1.25,-1) to[out=270,in=90,looseness=.8] (4.75,-3);
    \draw (-.25,-1) to[out=270,in=90,looseness=1.2] (1.75,-3);
    \draw (.75,-1) to[out=270,in=90,looseness=1] (-1.25,-3);
    \draw (1.75,-1) to[out=270,in=90,looseness=1] (-.25,-3);
    \draw (2.75,-1) to[out=270,in=90,looseness=1] (.75,-3);
    \draw (3.75,-1) to[out=270,in=90,looseness=1.2] (2.75,-3);
    \draw (4.75,-1) to[out=270,in=90,looseness=1.2] (3.75,-3);
    \begin{scope}[yshift=-3cm]
    \draw (-1.25,0) node[red,xshift=-.25cm,yshift=0cm,align=center]{x\\y} to[out=270,in=90,looseness=1.2] (0,-2);
    \draw (-.25,0) node[blue,xshift=-.25cm,yshift=0cm,align=center]{y\\a} to[out=270,in=90,looseness=1.2] (1.25,-2);
    \draw (.75,0) node[red,xshift=-.25cm,yshift=0cm,align=center]{y\\x} to[out=270,in=90,looseness=1.2] (2.5,-2);
    \draw (3.75,-1) node[black,circle,fill,inner sep=1.25]{} node[blue,xshift=-.2cm,yshift=-.45cm,align=center]{x\\d} -- (3.75,-2);
    \begin{scope}[xshift=3.75cm]
    \draw (0,-1) to[out=180,in=270,looseness=1.2] (-2,0) node[blue,xshift=-.25cm,yshift=0cm,align=center]{x\\b};
    \draw (0,-1) to[out=180,in=270,looseness=1.2] (-1,0) node[red,xshift=-.25cm,yshift=0cm,align=center]{b\\c};
    \draw (0,-1) to[out=90,in=270,looseness=1.2] (0,0) node[blue,xshift=-.25cm,yshift=-0cm,align=center]{c\\d};
    \draw (0,-1) to[out=0,in=270,looseness=1.2] (1,0) node[red,xshift=-.25cm,yshift=0cm,align=center]{c\\b};
    \end{scope}
    \end{scope}
    \end{scope}
    \draw[dashed] (0,0) to[out=90,in=0] (-1.2,2) -- (-8,2);
    \draw[fill=gray] (0,0) circle (.25);
    \draw (-2.25,-2) to[out=270,in=270,looseness=1.2] (1.5,0) -- (1.5,2) to[out=90,in=90,looseness=1.2] (-2.25,4);
    \draw (-3.5,-2) to[out=270,in=270,looseness=1.2] (2,-.167) -- (2,2.167) to[out=90,in=90,looseness=1.2] (-3.5,4);
    \draw (-4.75,-2) to[out=270,in=270,looseness=1.2] (2.5,-.333) -- (2.5,2.333) to[out=90,in=90,looseness=1.2] (-4.75,4);
    \draw (-6,-2) to[out=270,in=270,looseness=1.2] (3,-.5) -- (3,2.5) to[out=90,in=90,looseness=1.2] (-6,4);
\end{tikzpicture}
\caption{The diagram obtained by performing the (expanding) shift by the split $(b,a)$ on the closed strand diagram represented in \cref{fig_CSD}.}
\label{fig_CSD_shift_ex}
\end{figure}
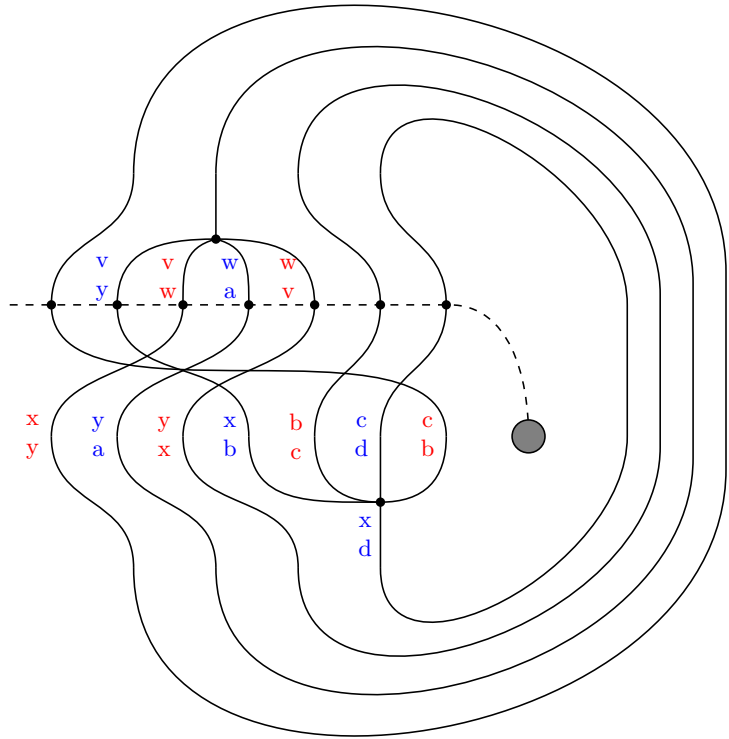

As we will see in \cref{PROP conjugator}, shifts of the base line represent conjugations by split and merge diagrams (which were defined in \cref{SUB groupoid generators}).

\medskip

\phantomsection\label{TXT similarity}
The two types of transformations defined so far are called \textbf{similarities}, and we say that two closed strand diagrams are \textbf{similar} if they only differ by the application of shifts and permutations.
If $\eta$ is a closed diagram, we denote by $\llbracket \eta \rrbracket$ its similarity class.
In practice, two closed diagrams are similar if they are the same up to ``forgetting'' about the base line and the base points.

\begin{remark}\label{RMK cohomology}
    The base line can be thought of as a cocycle of the graph cohomology of the graph obtained from the closed strand diagram by ``forgetting'' about the base line and the base points.
    In this sense, similarities preserve cohomology, and cohomology classes of this graph correspond to similarity classes of closed strand diagrams.
    This is going to be useful for the algorithm in \cref{SUB algorithm}.
\end{remark}

\subsubsection{Reductions of Closed Strand Diagrams}

A \textbf{Type 1 or 2 reduction} of an $R$-branching closed diagram is either of the two moves shown in \cref{fig_SD_reductions}.
These are the same reductions of strand diagrams, but it is important to keep in mind that these cannot ``cross'' the base line.
In practice, shifts (defined at \cpageref{TXT shifts}) are often needed to move the base line and ``unlock'' these types of reductions.

\phantomsection\label{TXT 3 reductions}
A \textbf{loping strand} is a path of strands without splits nor merges.
A \textbf{Type 3 reduction} is a move obtained by considering a sequence of looping strands that are consecutive in their given order at the intersection with the base line and whose labels correspond to the bottom labels of some replacement tree $T$ up to some renaming;
then replacing those looping strands with a single looping strand labeled as the top strand of $T$, up to the same renaming.
Each loop can wind multiple times around the central ``hole'', in which case there are multiple copies of base points that are labeled as the top strand of $T$;
the strands that are not involved in the reduction must intersect the base line on the left or on the right of those that are involved (this can be achieved by performing a permutation).
For example, \cref{fig_3reduction} depicts the general shape of a Type 3 reduction with winding number equal to 2.

Additionally, a Type 3 reduction is only allowed when the labeling symbols being ``deleted'' by the reduction do not appear among labels of strands not involved in the reduction.
The reason for this is that we would otherwise lose data about edge adjacency, since Type 3 reductions encode ``anti-expansions'' of some portion of the base graph and thus delete certain vertices.
Finally, it is important to keep in mind that the looping strands involved in a Type 3 reduction must be in the right order, which is why in practice permutations of the base points (defined at \cpageref{TXT permutations}) are often needed to ``unlock'' this kind of reductions.

\begin{figure}\centering
\begin{tikzpicture}[scale=.65,font=\scriptsize]
    \useasboundingbox (-8.25,-13.5) rectangle (7.25,5.5);
    %
    \draw[dashed] (0,0) -- (-8,0);
    \draw[fill=gray] (0,0) circle (.25);
    \draw (-7,1) -- (-7,0) node[black,circle,fill,inner sep=1.25]{} -- node[Orange,left,align=center]{v\textsubscript{1}\\w\textsubscript{1}} (-7,-1);
    \draw[dotted] (-6,1) -- (-6,-1);
    \draw (-5,1) -- (-5,0) node[black,circle,fill,inner sep=1.25]{} -- node[Plum,left,align=center]{v\textsubscript{n}\\w\textsubscript{n}} (-5,-1);
    \draw (-3.5,1) -- (-3.5,0) node[black,circle,fill,inner sep=1.25]{} -- node[Orange,left,align=center]{v\textsubscript{1}\textsuperscript{$\prime$}\\w\textsubscript{1}\textsuperscript{$\prime$}} (-3.5,-1);
    \draw[dotted] (-2.5,1) -- (-2.5,-1);
    \draw (-1.5,1) -- (-1.5,0) node[black,circle,fill,inner sep=1.25]{} -- node[Plum,left,align=center]{v\textsubscript{n}\textsuperscript{$\prime$}\\w\textsubscript{n}\textsuperscript{$\prime$}} (-1.5,-1);
    \draw (-7,-1) to[out=270,in=270] (7,-1.5);
    \draw[dotted] (-6,-1) to[out=270,in=270] (6,-1.5);
    \draw (-5,-1) to[out=270,in=270] (5,-1.5);
    \draw (-3.5,-1) to[out=270,in=270] (3.5,-1.5);
    \draw[dotted] (-2.5,-1) to[out=270,in=270] (2.5,-1.5);
    \draw (-1.5,-1) to[out=270,in=270] (1.5,-1.5);
    \draw (1.5,-1.5) to[out=90,in=270] (5,1.5);
    \draw[dotted] (2.5,-1.5) to[out=90,in=270] (6,1.5);
    \draw (3.5,-1.5) to[out=90,in=270] (7,1.5);
    \draw (5,-1.5) to[out=90,in=270] (1.5,1.5);
    \draw[dotted] (6,-1.5) to[out=90,in=270] (2.5,1.5);
    \draw (7,-1.5) to[out=90,in=270] (3.5,1.5);
    \draw (7,1.5) to[out=90,in=90] (-7,1);
    \draw[dotted] (6,1.5) to[out=90,in=90] (-6,1);
    \draw (5,1.5) to[out=90,in=90] (-5,1);
    \draw (3.5,1.5) to[out=90,in=90] (-3.5,1);
    \draw[dotted] (2.5,1.5) to[out=90,in=90] (-2.5,1);
    \draw (1.5,1.5) to[out=90,in=90] (-1.5,1);
    \draw[thick,-stealth] (0,-6) -- (0,-7);
    \begin{scope}[yshift=-10.5cm]
    \draw[dashed] (0,0) -- (-4,0);
    \draw[fill=gray] (0,0) circle (.25);
    \draw (-3,0) -- node[Green,left,align=center]{i\\t} (-3,-1) to[out=270,in=270] (3,-1) to[out=90,in=270] (1.5,1) to[out=90,in=90] (-1.5,1) -- (-1.5,0);
    \draw (-1.5,0) node[black,circle,fill,inner sep=1.25]{} -- node[Green,left,align=center]{i\textsuperscript{$\prime$}\\t\textsuperscript{$\prime$}} (-1.5,-1) to[out=270,in=270] (1.5,-1) to[out=90,in=270] (3,1) to[out=90,in=90] (-3,1) -- (-3,0) node[black,circle,fill,inner sep=1.25]{};
    \end{scope}
\end{tikzpicture}
\caption{A schematic depiction of a Type 3 reduction with winding number equal to 2.}
\label{fig_3reduction}
\end{figure}
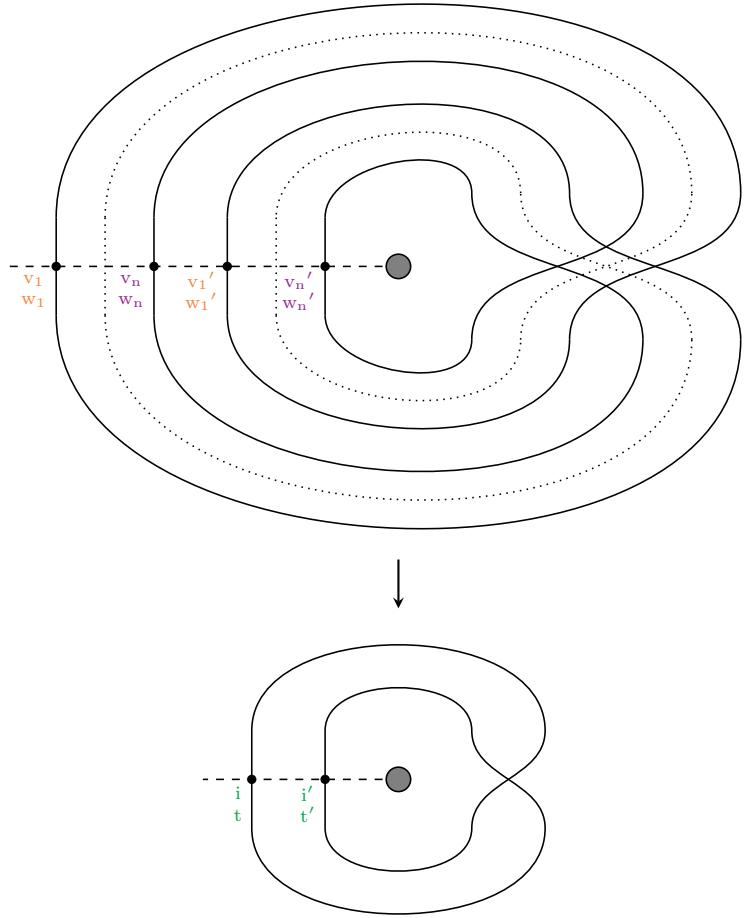

\subsection{Confluent Reduction Systems and Reduced Closed Diagrams}
\label{SUB confluent reductions}

We say that two closed strand diagrams are \textbf{equivalent} if we can obtain one from the other by applying a sequence of reductions.
The similarity class of a closed strand diagram (defined at \cpageref{TXT similarity}) is \textbf{reduced} if no reduction can be performed on any diagram that it contains.
We will just say that a closed strand diagram is reduced if it belongs to a reduced similarity class, as these only differ by moving the base line or the base points.

\phantomsection\label{TXT airplane non confluent}
Observe that, in general, reduced similarity classes of closed strand diagrams may not be unique.
For example, \cref{fig_TA_reduced} depicts two distinct reduced closed diagrams in the Airplane replacement system.
However, under the hypothesis of \textit{reduction-confluence} of the replacement system, we can immediately circumvent this issue, as is shown below.

\begin{figure}\centering
\begin{subfigure}[c]{\textwidth}
\begin{tikzpicture}[font=\small,scale=.85]
    \draw (0,0) circle (1) node[red,xshift=-0.975cm,yshift=.35cm,align=center]{y\\x};
    \draw (0,0) circle (1.5) node[blue,xshift=-1.45cm,yshift=.35cm,align=center]{y\\w};
    \draw (0,0) circle (2) node[red,xshift=-1.9cm,yshift=.35cm,align=center]{x\\y};
    \draw (0,0) circle (2.5) node[blue,xshift=-2.375cm,yshift=.35cm,align=center]{x\\v};
    \draw[dashed] (0,0) -- (-1,0) node[black,circle,fill,inner sep=1.25]{} -- (-1.5,0) node[black,circle,fill,inner sep=1.25]{} -- (-2,0) node[black,circle,fill,inner sep=1.25]{} -- (-2.5,0) node[black,circle,fill,inner sep=1.25]{} -- (-3,0);
    \draw[fill=gray] (0,0) circle (.25);
    \draw[thick,-stealth] (-3.5,0) to[out=180,in=90] node[xshift=-1.5cm,align=center]{Reduce by the 3\\inner strands} (-4.25,-.75);
    \begin{scope}[xshift=-4.5cm,yshift=-3.75cm]
    \draw (0,0) circle (1.25) node[red,xshift=-1.25cm,yshift=.35cm,align=center]{x\\x};
    \draw (0,0) circle (2.25) node[blue,xshift=-2.15cm,yshift=.35cm,align=center]{x\\v};
    \draw[dashed] (0,0) -- (-1.25,0) node[black,circle,fill,inner sep=1.25]{} -- (-2.25,0) node[black,circle,fill,inner sep=1.25]{} -- (-2.9,0);
    \draw[fill=gray] (0,0) circle (.25);
    \end{scope}
    \draw[thick,-stealth] (3.5,0) to[out=0,in=90] node[xshift=1.5cm,align=center]{Reduce by all of\\the 4 strands}  (4.25,-.75);
    \begin{scope}[xshift=4.5cm,yshift=-3.75cm]
    \draw (0,0) circle (1.8) node[blue,xshift=-1.8cm,yshift=.35cm,align=center]{v\\w};
    \draw[dashed] (0,0) -- (-1.8,0) node[black,circle,fill,inner sep=1.25]{} -- (-2.9,0);
    \draw[fill=gray] (0,0) circle (.25);
    \end{scope}
\end{tikzpicture}
\caption{Two reduced closed strand diagrams obtained from the identity of $T_A$.}
\end{subfigure}
\begin{subfigure}[c]{\textwidth}
\begin{tikzpicture}
    \draw[blue,dotted] (0,0) ellipse (2.3cm and .8cm);
    \draw[red,dotted] (.95,0) ellipse (1.8cm and .8cm);
    \draw[->-=.5,blue] (-0.5,0) node[yshift=.3cm,xshift=-.2cm,black]{x} -- (-2,0) node[yshift=.3cm,black]{v} node[black,circle,fill,inner sep=1.25]{}; \draw[->-=.5,blue] (0.5,0) node[yshift=.3cm,xshift=.2cm,black]{y} -- (2,0) node[yshift=.3cm,black]{w} node[black,circle,fill,inner sep=1.25]{};
    \draw[->-=.5,red] (0.5,0) node[circle,fill,inner sep=1.25]{} to[out=90,in=90,looseness=1.7] (-0.5,0);
    \draw[->-=.5,red] (-0.5,0) node[black,circle,fill,inner sep=1.25]{} to[out=270,in=270,looseness=1.7] (0.5,0) node[black,circle,fill,inner sep=1.25]{};
    \draw[-stealth,blue] (-3,0) to[out=180,in=90] (-4,-1);
    \begin{scope}[xshift=-3cm,yshift=-2cm]
    \draw[->-=.5,blue] (-2,0) node[yshift=.3cm,black]{v} node[black,circle,fill,inner sep=1.25]{} -- (0,0) node[yshift=.3cm,xshift=-.2cm,black]{w} node[black,circle,fill,inner sep=1.25]{};
    \end{scope}
    \draw[thick,-stealth,red] (3,0) to[out=0,in=90] (4,-1);
    \begin{scope}[xshift=4.75cm,yshift=-2cm]
    \draw[->-=.5,blue] (-0.5,0) node[yshift=.3cm,xshift=-.2cm,black]{x} -- (-2,0) node[yshift=.3cm,black]{v} node[black,circle,fill,inner sep=1.25]{};
    \draw[->-=0,red] (0.5,0) to[out=90,in=90,looseness=1.7] (-0.5,0);
    \draw[red] (-0.5,0) node[black,circle,fill,inner sep=1.25]{} to[out=270,in=270,looseness=1.7] (0.5,0);
    \end{scope}
    \draw[white] (0,1.25);
    \draw[white] (-6,0);
    \draw[white] (6,0);
\end{tikzpicture}
\caption{The graph reductions corresponding to the strand diagram reductions shown above.}
\end{subfigure}
\caption{The reason why the Airplane replacement system (portrayed in \cref{fig_replacement_A}) is not reduction-confluent.}
\label{fig_TA_reduced}
\end{figure}
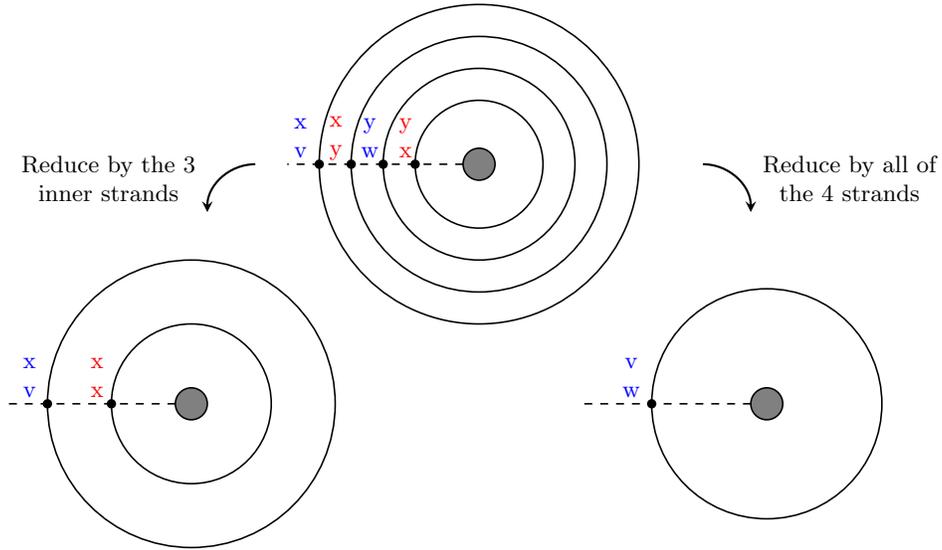
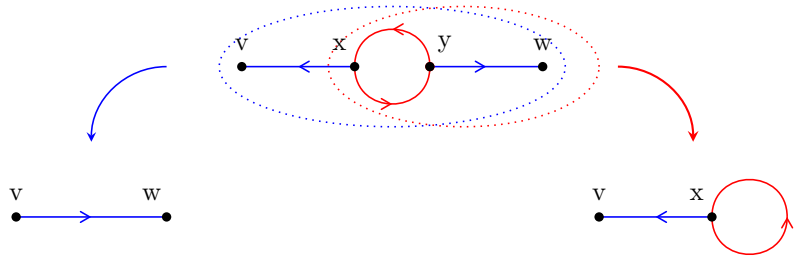

\subsubsection{Reduction-confluent Replacement Rules}
\label{SUB reduction systems}

Given a set of replacement rules $(R, \mathrm{C})$, we define its \textbf{reduction system} as the directed graph whose set of vertices is the set of all graphs with edges colored by $\mathrm{C}$ and whose edges are described by the anti-expansions of $(R, \mathrm{C})$.
More explicitly, given a graph $\Gamma$ and a color $c \in \mathrm{C}$, suppose that $\Gamma$ contains a subgraph $\Delta$ that is isomorphic to $R_c$ with the possible exception of having $\iota$ and $\tau$ glued together, and such that each edge of $\Gamma$ that is adjacent to some vertex of $\Delta$ except for $\iota$ and $\tau$ belongs to $\Delta$.
In this case, there is an edge ($\Gamma$, $\Gamma^*$), where $\Gamma^*$ is the graph obtained from $\Gamma$ by replacing $\Delta$ with an edge starting at $\iota$ and terminating at $\tau$ colored by $c$.

\begin{definition}\label{DEF red-conf}
    We say that a set of replacement rules $(R, \mathrm{C})$ is \textbf{reduction-confluent} if its reduction system is confluent
    (i.e., whenever $a \dashrightarrow b$ and $a \dashrightarrow c$ are two finite sequences of reductions of $a$, there exist a graph expansion $d$ and two finite sequences of reductions $b \dashrightarrow d$ and $c \dashrightarrow d$).
\end{definition}

Now, if the replacement rules are both reduction-confluent and expanding (\cref{DEF expanding}), then it is clear that the reduction system is terminating.
So, using Newman's Diamond Lemma from \cite{NewmanDiamond}, we have that each connected component of the rewriting system has a unique reduced graph.

\phantomsection\label{TXT Basilica Vicsek BubbleBath Houghton QV}
The replacement systems for Thompson groups $F, T$ and $V$ (and the Higman-Thompson groups) are clearly reduction-confluent.
In truth, most of the replacement systems discussed in literature (see \cref{SUB examples}) are reduction-confluent:
the Basilica, the Vicsek and the Bubble Bath reduction rules (\cref{fig replacement B,,fig replacement Vic,,fig replacement BB}, respectively) are confluent, and the same holds for their generalized versions (the rabbit and the generalized Vicsek replacement systems).
The replacement rules for the Houghton groups $H_n$ (from \cite{Houghton1978TheFC}; see \cref{fig Houghton}) and for the groups $QV$, $QT$, $QF$, $\tilde{Q}V$ and $\tilde{Q}T$ (originally from \cite{QV}, also studied in \cite{QV1,QV2}; see \cref{fig QV}) are reduction-confluent too.
All of this can be easily proved using Newman's Diamond Lemma.

For example, consider the Basilica replacement system.
We say that two graph reductions of the same graph are disjoint if they involve subgraphs that share no edges (they can share vertices, provided they are the initial or terminal vertices of the replacement graph).
If two reductions are disjoint, then each can be applied after the other and there is nothing to prove.
Let $A$ and $B$ be two distinct graph reductions of the same graph that are not disjoint.
Then, since each reduction must identify a copy of the Basilica replacement graph (\cref{fig replacement B}), $A$ and $B$ cannot involve the same loop, otherwise they would be the same reduction.
Thus there are two cases, both portrayed in \cref{fig_B_reduction_confluent}.
In each case, both reductions produce the same graph, so we are done.

The Airplane replacement system instead is not reduction-confluent, as the base graph itself can be reduced in three ways that produce two distinct reduced graphs (these correspond to the ones shown in \cref{fig_TA_reduced}).

\begin{remark}
Among the replacement systems discussed here, the only one with more than one color happens to be exactly the one that is not reduction-confluent.
This is just a coincidence, as it is not hard to build replacement systems with one color that are not reduction-confluent and replacement systems that are reduction-confluent despite having more than one color.
\end{remark}

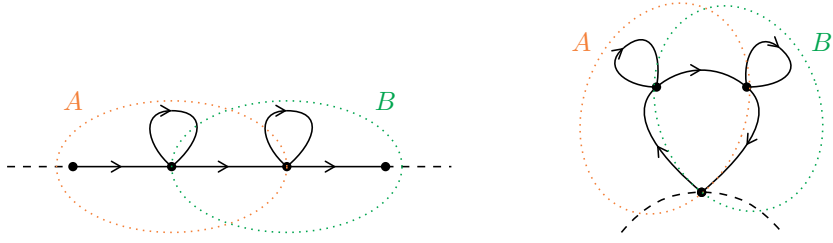
\begin{figure}\centering
\begin{subfigure}[b]{.575\textwidth}\centering
\begin{tikzpicture}[scale=.875]
    \draw[->-=.5] (-.5,0) node[circle,fill,inner sep=1.25]{} -- (1,0) node[circle,fill,inner sep=1.25]{};
    \draw[->-=.5] (1,0) -- (2.75,0) node[circle,fill,inner sep=1.25]{};
    \draw[->-=.5] (2.75,0) -- (4.25,0) node[circle,fill,inner sep=1.25]{};
    \draw[dashed] (-1.5,0) -- (-.5,0);
    \draw[dashed] (4.25,0) -- (5.25,0);
    \draw (1,0) to[out=140,in=180,looseness=1.5] (1,0.85); \draw[->-=0] (1,0.85) to[out=0,in=40,looseness=1.75] (1,0);
    \draw (2.75,0) to[out=140,in=180,looseness=1.5] (2.75,0.85); \draw[->-=0] (2.75,0.85) to[out=0,in=40,looseness=1.75] (2.75,0);
    \draw[dotted, Orange] (1,0) ellipse (1.75cm and 1cm);
    \draw[Orange] (-.5,1) node{$A$};
    \draw[dotted, Green] (2.75,0) ellipse (1.75cm and 1cm);
    \draw[Green] (4.25,1) node{$B$};
\end{tikzpicture}
\end{subfigure}
\begin{subfigure}[b]{.4\textwidth}\centering
\begin{tikzpicture}[scale=.8]
    \draw[->-=.5] (0,0) node[circle,fill,inner sep=1.25]{} to[out=140,in=220] (-.75,1.75) node[circle,fill,inner sep=1.25]{};
    \draw[->-=.5] (-.75,1.75) to[out=40,in=140] (.75,1.75) node[circle,fill,inner sep=1.25]{};
    \draw[->-=.5] (.75,1.75) to[out=-40,in=40] (0,0) node[circle,fill,inner sep=1.25]{};
    \draw[dashed] (-1.333,-.667) to[out=50,in=180] (0,0) to[out=0,in=130] (1.333,-.667);
    \begin{scope}[xshift=-.75cm,yshift=1.75cm,rotate=40]
    \draw (0,0) to[out=140,in=180,looseness=1.5] (0,0.85); \draw[->-=0] (0,0.85) to[out=0,in=40,looseness=1.75] (0,0);
    \end{scope}
    \begin{scope}[xshift=.75cm,yshift=1.75cm,rotate=-40]
    \draw (0,0) to[out=140,in=180,looseness=1.5] (0,0.85); \draw[->-=0] (0,0.85) to[out=0,in=40,looseness=1.75] (0,0);
    \end{scope}
    \draw[dotted, Orange, rotate=-20] (-1.05,1.1) ellipse (1.35cm and 1.8cm);
    \draw[Orange] (-2,2.5) node{$A$};
    \draw[dotted, Green, rotate=20] (1.05,1.1) ellipse (1.35cm and 1.75cm);
    \draw[Green] (2,2.5) node{$B$};
\end{tikzpicture}
\end{subfigure}
\caption{The only two possible distinct non-disjoint graph reductions of the Basilica replacement rules.}
\label{fig_B_reduction_confluent}
\end{figure}

\subsubsection{Uniqueness of Reduced Closed Diagrams}

The following Lemma tells us that, if we are dealing with a replacement system that is based on reduction-confluent replacement rules, then each closed strand diagram is equivalent to a unique reduced closed strand diagram up to similarity.

\begin{lemma}
\label{LEM reduced CSD}
Suppose that the replacement rules $(R, \mathrm{C})$ are reduction-confluent.
Then, for each similarity class $\llbracket \eta \rrbracket$ of $R$-branching closed diagrams, there is a unique reduced similarity class of $R$-branching closed diagrams that is equivalent to $\llbracket \eta \rrbracket$.
\end{lemma}

\begin{proof}
As we did earlier for \cref{LEM reduced SD}, we will use Newman's Diamond Lemma from \cite{NewmanDiamond}.
Consider the directed graph defined as follows: the set of vertices is the set of similarity classes of $R$-branching closed diagrams and we have an edge $\llbracket \eta \rrbracket \longrightarrow \llbracket \zeta \rrbracket$ for each distinct reduction from a closed diagram equivalent to $\eta$ to one equivalent to $\zeta$.
It then suffices to prove that this graph is terminating and locally confluent.

It is clear that similar closed diagrams have the same number of splits and merges.
Since each reduction strictly decreases the number of splits and merges of a closed diagram, the directed graph is terminating, so we only need to prove the local confluence.

Suppose $\eta, \zeta$ and $\kappa$ are closed strand diagrams such that $\llbracket \eta \rrbracket \overset{A}{\longrightarrow} \llbracket \zeta \rrbracket$ and $\llbracket \eta \rrbracket \overset{B}{\longrightarrow} \llbracket \kappa \rrbracket$ are distinct reductions.
Without loss of generality we may assume that the reductions of closed diagrams are $\eta_A \overset{A}{\longrightarrow} \zeta$ and $\eta_B \overset{B}{\longrightarrow} \kappa$, where $\eta_A$ and $\eta_B$ are similar.
This means that the reductions $A$ and $B$ are performed without shifting the base line nor permutating the order of the base points, whereas we can transform $\eta_A$ into $\eta_B$ with a finite sequence of permutations, shifts and inversions.
Note that then $\eta_A$ and $\eta_B$ have the same strands, splits and merges, and the only possible distinctions must involve the position of the base line and the order of the base points.

First, it is easy to see that if none of the reductions $A$ and $B$ are of Type 3 then, with a sole exception, they must be disjoint, and if they are disjoint the same proof of \cref{LEM reduced SD} holds.
The exception is represented in \cref{fig_onion}:
in this case, after a Type 2 reduction on $\llbracket \eta \rrbracket$ we can apply a Type 3 reduction to get the same diagram obtained by performing a Type 1 reduction on $\llbracket \eta \rrbracket$, so local confluence is verified in this case.
Therefore, we can suppose that $B$ is of Type 3.
But then, since Type 3 reductions can only involve strands devoid of splits and merges, it is clear that if $A$ is of Type 1 or 2 then $A$ and $B$ are disjoint (by which we mean that they involve different strands), so we are done in this case.

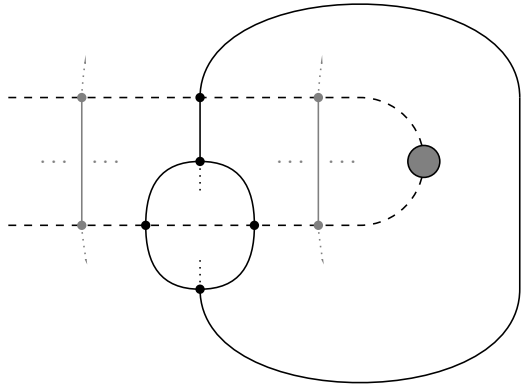
\begin{figure}\centering
\begin{tikzpicture}[font=\small,scale=.85]
    \useasboundingbox (-3,-4.5) rectangle (5,1.5);
    %
    \draw (0,-1) -- (0,0) node[black,circle,fill,inner sep=1.25]{} to[out=90,in=90] (5,0);
    \draw (0,-1) node[black,circle,fill,inner sep=1.25]{} to[out=180,in=90,looseness=1.2] (-.85,-2);
    \draw[dotted] (0,-1) -- (0,-1.5);
    \draw (0,-1) to[out=0,in=90,looseness=1.2] (.85,-2) node[black,circle,fill,inner sep=1.25]{};
    \draw (0,-3) node[black,circle,fill,inner sep=1.25]{} to[out=180,in=270,looseness=1.2] (-.85,-2) node[black,circle,fill,inner sep=1.25]{};
    \draw[dotted] (0,-3) -- (0,-2.5);
    \draw (0,-3) to[out=0,in=270,looseness=1.2] (.85,-2);
    \draw (0,-3) to[out=270,in=270] (5,-3) -- (5,0);
    \begin{scope}[xshift=3.5cm,yshift=-1cm]
    \draw[dashed] (0,0) to[out=90,in=0] (-1,1) -- (-6.5,1);
    \draw[dashed] (0,0) to[out=270,in=0] (-1,-1) -- (-6.5,-1);
    \draw[fill=gray] (0,0) circle (.25);
    \end{scope}
    \draw[gray,dotted] (-1.85,-2) to[out=270,in=285] (-1.8,-2.5);
    \draw[gray] (-1.85,-2) node[circle,fill,inner sep=1.25]{} -- node[left]{$\dots$} node[right]{$\dots$} (-1.85,0) node[circle,fill,inner sep=1.25]{};
    \draw[gray,dotted] (-1.85,0) to[out=90,in=85] (-1.8,.5);
    \draw[gray,dotted] (1.85,-2) to[out=270,in=285] (1.9,-2.5);
    \draw[gray] (1.85,-2) node[circle,fill,inner sep=1.25]{} -- node[left]{$\dots$} node[right]{$\dots$} (1.85,0) node[circle,fill,inner sep=1.25]{};
    \draw[gray,dotted] (1.85,0) to[out=90,in=85] (1.9,.5);
\end{tikzpicture}
\caption{The only possible non-disjoint Type 1 and 2 reductions from the proof of \cref{LEM reduced CSD}. Labels are omitted for the sake of clarity, but they should allow a Type 1 reduction and a Type 2 reduction, and the two base lines represented differ by a shift that ``unlocks'' these reductions.}
\label{fig_onion}
\end{figure}

Finally, suppose that $A$ and $B$ are both of Type 3.
Consider the base graph $\Gamma$ of the diagram (which is the graph represented by the base points, as seen in \cref{SUB closed strand diagrams}).
Observe that Type 3 reductions (defined at \cpageref{TXT 3 reductions}) correspond to reductions of the base graph $\Gamma$ of the closed strand diagram, and it is clear that distinct Type 3 reductions describe distinct reductions of the base graph.
Now, if the reductions $A$ and $B$ are disjoint, by which we mean that they involve different strands of $\llbracket \eta \rrbracket$, then $B$ can be applied to $\zeta$ and $A$ to $\kappa$, so we are done.
If instead they are not disjoint, then consider the subgraph $\Delta$ of $\Gamma$ that is involved in both reductions.
More precisely, $\Delta$ is the union of the two subgraphs isomorphic to replacement graphs which determine the reductions induced by $A$ and $B$ on $\Gamma$.
Denote by $\Delta^A$ and $\Delta^B$ the graphs obtained from $\Delta$ by applying these reductions.
Now, because the replacement rules $(R, \mathrm{C})$ are reduction-confluent, there is a graph $\Delta^*$ that is a common reduction of both $\Delta^A$ and $\Delta^B$.
The graph reductions that one needs to perform to obtain $\Delta^*$ from $\Delta^A$ and $\Delta^B$ can clearly be performed on the copy of $\Delta$ contained as a subgraph in $\Gamma$, and these correspond to Type 3 reductions of $\llbracket \eta \rrbracket$ that make these reductions locally confluent, so we are done.
\end{proof}

It is worth noting that the reduction-confluent hypothesis is only needed to prove confluence of Type 3 reductions.
Type 1 and 2 reductions by themselves are always confluent.

\subsection{Stable and Vanishing Symbols}
\label{SUB stable and vanishing}

In this subsection we investigate what happens when a sequence of shifts moves the base line back to where it started.
This will be important when developing the algorithm for solving the conjugacy problem in \cref{SUB algorithm}

Let $Z$ be an $\mathcal{X}$-strand diagram for some replacement system $\mathcal{X}$ (as defined at \cpageref{TXT X-SDs}).
Denote by $Z^\infty$ the infinite power of $Z$, which is the infinite strand diagram obtained by inductively taking higher and higher powers of $Z$ and keeping track of the lines under which the copies of $Z$ are attached.
More precisely, in order to build $Z^\infty$, start from $Z$ with two ``base lines'', one at the top and one at the bottom;
these will be called \textbf{main base lines}.
Start attaching copies of $Z$ at both sides, drawing new base lines where the elements are glued.
Recall that labels need to be adjusted when gluing together the copies of $Z$, as described in \cref{SUB SDs composition};
keep unchanged the labels of the original copy of $Z$, and adjust the labels above and below the main base lines.
The result of this infinite procedure is $Z^\infty$.
An example is portrayed in \cref{fig_infinite_power}, where the main base lines are depicted in green.

\begin{definition}
    Each symbol that appears in $Z$ is either:
    \begin{itemize}
        \item \textbf{stable} if it appears infinitely many times in $Z^\infty$,
        \item \textbf{vanishing} if it only appears finitely many times in $Z^\infty$.
    \end{itemize}
    In other words, stable symbols appear throughout all of $Z^\infty$, whereas the instances of a vanishing symbol are limited to a finite portion of $Z^\infty$ (i.e., a composition of finitely many copies of $Z$ inside $Z^\infty$).
    Thus, all instances of a vanishing symbol are contained in some minimal portion $Z^I \subset Z^\infty$, where $I$ is an interval of integers, and outside of $Z^I$ such a symbol is replaced by other vanishing symbols.
\end{definition}

In the example portrayed in \cref{fig_infinite_power}, the symbols a, b, c, d, e and f are stable, while the symbols x\textsubscript{$i$} and y\textsubscript{$i$} ($i \in \mathbb{Z}$) are vanishing.

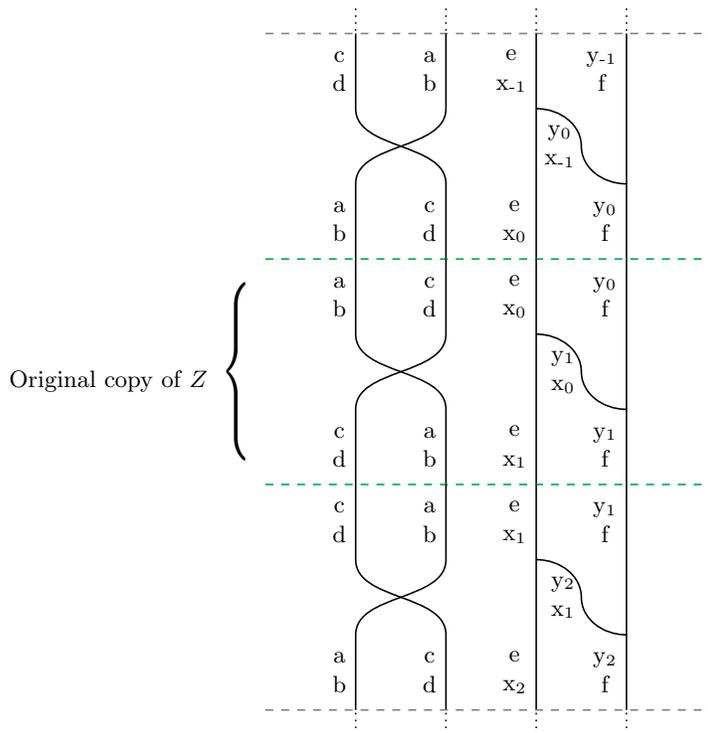
\begin{figure}\centering
\begin{tikzpicture}[font=\small,xscale=1.2]
    \begin{scope}[yshift=3cm]
    \draw[dotted] (0,.333) -- (0,0);
    \draw[dotted] (1,.333) -- (1,0);
    \draw[dotted] (2,.333) -- (2,0);
    \draw[dotted] (3,.333) -- (3,0);
    \draw[dashed,gray] (-1,0) -- (4,0);
    \draw (0,0) -- node[left,align=center]{c\\d} (0,-1) to[out=270,in=90] (1,-2) -- node[left,align=center]{c\\d} (1,-3);
    \draw (1,0) -- node[left,align=center]{a\\b} (1,-1) to[out=270,in=90] (0,-2) -- node[left,align=center]{a\\b} (0,-3);
    \draw (2,0) -- node[left,align=center]{e\\x\textsubscript{-1}} (2,-1) -- (2,-2) -- node[left,align=center]{e\\x\textsubscript{0}} (2,-3);
    \draw (3,0) -- node[left,align=center]{y\textsubscript{-1}\\f} (3,-1) -- (3,-2) -- node[left,align=center]{y\textsubscript{0}\\f} (3,-3);
    \draw (2,-1) to[out=0,in=90] (2.5,-1.5) node[left,align=center,xshift=1.25pt]{y\textsubscript{0}\\x\textsubscript{-1}} to[out=270,in=180] (3,-2);
    \end{scope}
    \draw[dashed,Green] (-1,0) -- (4,0);
    \draw (0,0) -- node[left,align=center]{a\\b} (0,-1) to[out=270,in=90] (1,-2) -- node[left,align=center]{a\\b} (1,-3);
    \draw (1,0) -- node[left,align=center]{c\\d} (1,-1) to[out=270,in=90] (0,-2) -- node[left,align=center]{c\\d} (0,-3);
    \draw (2,0) -- node[left,align=center]{e\\x\textsubscript{0}} (2,-1) -- (2,-2) -- node[left,align=center]{e\\x\textsubscript{1}} (2,-3);
    \draw (3,0) -- node[left,align=center]{y\textsubscript{0}\\f} (3,-1) -- (3,-2) -- node[left,align=center]{y\textsubscript{1}\\f} (3,-3);
    \draw (2,-1) to[out=0,in=90] (2.5,-1.5) node[left,align=center,xshift=1.25pt]{y\textsubscript{1}\\x\textsubscript{0}} to[out=270,in=180] (3,-2);
    \draw[dashed,Green] (-1,-3) -- (4,-3);
    \draw (-2.5,-1.5) node{Original copy of $Z \:$ \scalebox{1.5}[2.5]{$\Biggl\{$}};
    \begin{scope}[yshift=-3cm]
    \draw (0,0) -- node[left,align=center]{c\\d} (0,-1) to[out=270,in=90] (1,-2) -- node[left,align=center]{c\\d} (1,-3);
    \draw (1,0) -- node[left,align=center]{a\\b} (1,-1) to[out=270,in=90] (0,-2) -- node[left,align=center]{a\\b} (0,-3);
    \draw (2,0) -- node[left,align=center]{e\\x\textsubscript{1}} (2,-1) -- (2,-2) -- node[left,align=center]{e\\x\textsubscript{2}} (2,-3);
    \draw (3,0) -- node[left,align=center]{y\textsubscript{1}\\f} (3,-1) -- (3,-2) -- node[left,align=center]{y\textsubscript{2}\\f} (3,-3);
    \draw (2,-1) to[out=0,in=90] (2.5,-1.5) node[left,align=center,xshift=1.25pt]{y\textsubscript{2}\\x\textsubscript{1}} to[out=270,in=180] (3,-2);
    \draw[dashed,gray] (-1,-3) -- (4,-3);
    \draw[dotted] (0,-3.333) -- (0,-3);
    \draw[dotted] (1,-3.333) -- (1,-3);
    \draw[dotted] (2,-3.333) -- (2,-3);
    \draw[dotted] (3,-3.333) -- (3,-3);
    \end{scope}
\end{tikzpicture}
\caption{An example of $Z^\infty$ for an element $Z$ based on the replacement rules for Thompson's group $F$, depicted in \cref{fig_replacement_interval}.}
\label{fig_infinite_power}
\end{figure}

For the curious reader, keeping in mind that each symbol represents a vertex in a graph expansion, one can see that each stable symbol corresponds to a point of the limit space that is periodic under the action of the rearrangement represented by $Z$.

\begin{remark}
\label{RMK vanishing symbols}
    Because of point (3) of \cref{DEF R-branching}, if a symbol is vanishing, then it can only appear in one connected component of the closed diagram $\zeta$ obtained from $Z$.
\end{remark}

\begin{remark}
\label{RMK stable symbols}
    Observe that there is a finite amount of distinguished stable symbols, even if (by definition) there are infinitely many instances in which each stable symbol appears.
    Indeed, each copy of $Z$ contains finitely many symbols, and the stable symbols are the same in every copy of $Z$ featured in $Z^\infty$ (even if they may switch places).
\end{remark}

The next Proposition tells us that moving the base line does not change the labeling, up to a permutation of the stable symbols (for example, in \cref{fig_infinite_power} the pairs of symbols (a,b) and (c,d) may be switched), and the results of this permutation can be listed by an algorithm, as stated in the Corollary below.

\begin{proposition}
    Let $\zeta$ be the closed diagram obtained from $Z$.
    If a sequence of shifts moves the base line of $\zeta$ back to the original position, then the resulting labeling of the diagram remains the same up to renaming symbols.
    This renaming induces a permutation of the stable symbols in each connected component of $\zeta$.
\end{proposition}

\begin{proof}
    Observe that each shift of the base line of $\zeta$ corresponds to applying the same shift to the base lines of $Z^\infty$.
    Thus, moving the base line of $\zeta$ to the original position is equivalent to moving the ``window'' between the two main base lines to some other portion of $Z^\infty$ that is the same up to renaming symbols.
    More precisely, this results in moving every base line of $Z^\infty$ up or down by $n_i$ steps (where $n_i \in \mathbb{Z}$) through each connected component $Z_i$ of $Z^\infty$.
    Now, as noted in \cref{RMK vanishing symbols}, vanishing symbols are distinct in each connected component, so one can always rename them to get them back to the original configuration.
    Instead, by \cref{RMK stable symbols} there is a finite amount of distinguished stable symbols, so they must be permuted independently in each connected component.
\end{proof}

Since there are only finitely many distinguished stable symbols, the permutation associated by cycling once through a connected component of a closed strand diagram has finite order.
Thus, if one cycles through the base line in this way multiple times, at some point one will find the original configuration of labels of the connected component (up to renaming the vanishing symbols).
Going through this process for every connected component yields an algorithm that, by exhaustion, lists every possible such configuration of the stable symbols, so we have the following.

\begin{corollary}
\label{COR labels}
    There is an algorithm that lists the finitely many possible configurations of labels of a closed diagram $\zeta$ (up to renaming) obtained by shifting the base line from a fixed position back to itself.
\end{corollary}

\section{Solving the Conjugacy Problem}
\label{SEC conjugacy problem}

In this last Section we examine conjugacy in the case of reduction-confluent replacement rules (\cref{SUB conf}) and we provide a pseudo-algorithm to solve the conjugacy problem under this assumption (\cref{SUB algorithm}).
Then we discuss the problem without assumption of reduction-confluence, solving the conjugacy problem in the Airplane rearrangement group (\cref{SUB non conf}).
Finally, in \cref{SUB DPO} we link this problematic to the general setting of DPO graph rewriting systems studied in computer science.

\subsection{Reduction-confluent Replacement Rules}
\label{SUB conf}

Let $\mathcal{X} = (X_0, R, \mathrm{C})$ be a fixed a replacement system.

\begin{proposition}
Suppose that $(R, \mathrm{C})$ are reduction-confluent replacement rules.
Then, if two rearrangements $f$ and $g$ of the same replacement system $\mathcal{X}$ are conjugate in $G_\mathcal{X}$, their reduced closed diagrams are similar.
\end{proposition}

\begin{proof}
Suppose that $f = g^h$, for $f$, $g$ and $h$ in $G_\mathcal{X}$.
Then you can represent the strand diagram for $f$ as shown in \cref{fig_conj}.
This is usually not a reduced strand diagram, but we do not need it to be.
If we now close the diagram, we can perform shifts to move the base line counterclockwise until it is right between $h$ and $g$, as represented in green in \cref{fig_conj}.
In this configuration, $h$ and $h^{-1}$ cancel out completely, leaving the same closed strand diagram that one would obtain when closing $g$.
Thus the closures of $f$ and $g$ are equivalent to a common closed strand diagram, and so because of \cref{LEM reduced CSD} they must have the same reduced closed strand diagram, up to similarity.
\end{proof}

\begin{figure}\centering
\begin{tikzpicture}[scale=.4]
    \useasboundingbox (-6,-26) rectangle (12.25,4.5);
    %
    \begin{scope}[draw=Orange]
    \draw (-1.75,-.5) node{\textcolor{Orange}{$\cdots$}};
    \draw (1.75,-.5) node{\textcolor{Orange}{$\cdots$}};
    \draw (-3.5,0) -- (-3.5,-2);
    \draw (-3.5,-5) -- (-3.5,-7);
    \draw[dotted] (-4,-2.5) to[out=90,in=180] (-3.5,-2) to[out=0,in=90] (-3,-2.5);
    \draw[dotted] (-4,-4.5) to[out=270,in=180] (-3.5,-5) to[out=0,in=270] (-3,-4.5);
    \draw (3.5,0) -- (3.5,-2);
    \draw (3.5,-5) -- (3.5,-7);
    \draw[dotted] (4,-2.5) to[out=90,in=0] (3.5,-2) to[out=180,in=90] (3,-2.5);
    \draw[dotted] (4,-4.5) to[out=270,in=0] (3.5,-5) to[out=180,in=270] (3,-4.5);
    \draw (0,0) -- (0,-1);
    \draw[dotted] (0,-1) -- (0,-2);
    \draw (0,-1) to[out=180,in=90] (-1.5,-2);
    \draw (0,-1) to[out=0,in=90] (1.5,-2);
    \draw[dotted] (-2,-2.5) to[out=90,in=180] (-1.5,-2) to[out=0,in=90] (-1,-2.5);
    \draw[dotted] (2,-2.5) to[out=90,in=0] (1.5,-2) to[out=180,in=90] (1,-2.5);
    \draw (-5.3,-3.5) node{\textcolor{Orange}{\Large$h^{-1}$}};
    \draw (0,-3.5) node{\textcolor{Orange}{\LARGE$\cdots$}};
    \draw (0,-7) -- (0,-6);
    \draw[dotted] (0,-6) -- (0,-5);
    \draw (0,-6) to[out=0,in=270] (1.5,-5);
    \draw (0,-6) to[out=180,in=270] (-1.5,-5);
    \draw[dotted] (-2,-4.5) to[out=270,in=180] (-1.5,-5) to[out=0,in=270] (-1,-4.5);
    \draw[dotted] (2,-4.5) to[out=270,in=0] (1.5,-5) to[out=180,in=270] (1,-4.5);
    \end{scope}
    \begin{scope}[draw=Plum,yshift=-7cm]
    \draw (-3.5,0) -- (-3.5,-2);
    \draw (-3.5,-5) -- (-3.5,-7);
    \draw[dotted] (-4,-2.5) to[out=90,in=180] (-3.5,-2) to[out=0,in=90] (-3,-2.5);
    \draw[dotted] (-4,-4.5) to[out=270,in=180] (-3.5,-5) to[out=0,in=270] (-3,-4.5);
    \draw (3.5,0) -- (3.5,-2);
    \draw (3.5,-5) -- (3.5,-7);
    \draw[dotted] (4,-2.5) to[out=90,in=0] (3.5,-2) to[out=180,in=90] (3,-2.5);
    \draw[dotted] (4,-4.5) to[out=270,in=0] (3.5,-5) to[out=180,in=270] (3,-4.5);
    \draw (0,0) -- (0,-1);
    \draw[dotted] (0,-1) -- (0,-2);
    \draw (0,-1) to[out=180,in=90] (-1.5,-2);
    \draw (0,-1) to[out=0,in=90] (1.5,-2);
    \draw[dotted] (-2,-2.5) to[out=90,in=180] (-1.5,-2) to[out=0,in=90] (-1,-2.5);
    \draw[dotted] (2,-2.5) to[out=90,in=0] (1.5,-2) to[out=180,in=90] (1,-2.5);
    \draw (-5.3,-3.5) node{\textcolor{Plum}{\Large$g$}};
    \draw (0,-3.5) node{\textcolor{Plum}{\LARGE$\cdots$}};
    \draw (0,-7) -- (0,-6);
    \draw[dotted] (0,-6) -- (0,-5);
    \draw (0,-6) to[out=0,in=270] (1.5,-5);
    \draw (0,-6) to[out=180,in=270] (-1.5,-5);
    \draw[dotted] (-2,-4.5) to[out=270,in=180] (-1.5,-5) to[out=0,in=270] (-1,-4.5);
    \draw[dotted] (2,-4.5) to[out=270,in=0] (1.5,-5) to[out=180,in=270] (1,-4.5);
    \end{scope}
    \begin{scope}[draw=Orange,yshift=-14cm]
    \draw (-3.5,0) -- (-3.5,-2);
    \draw (-3.5,-5) -- (-3.5,-7);
    \draw[dotted] (-4,-2.5) to[out=90,in=180] (-3.5,-2) to[out=0,in=90] (-3,-2.5);
    \draw[dotted] (-4,-4.5) to[out=270,in=180] (-3.5,-5) to[out=0,in=270] (-3,-4.5);
    \draw (3.5,0) -- (3.5,-2);
    \draw (3.5,-5) -- (3.5,-7);
    \draw[dotted] (4,-2.5) to[out=90,in=0] (3.5,-2) to[out=180,in=90] (3,-2.5);
    \draw[dotted] (4,-4.5) to[out=270,in=0] (3.5,-5) to[out=180,in=270] (3,-4.5);
    \draw (0,0) -- (0,-1);
    \draw[dotted] (0,-1) -- (0,-2);
    \draw (0,-1) to[out=180,in=90] (-1.5,-2);
    \draw (0,-1) to[out=0,in=90] (1.5,-2);
    \draw[dotted] (-2,-2.5) to[out=90,in=180] (-1.5,-2) to[out=0,in=90] (-1,-2.5);
    \draw[dotted] (2,-2.5) to[out=90,in=0] (1.5,-2) to[out=180,in=90] (1,-2.5);
    \draw (-5.3,-3.5) node{\textcolor{Orange}{\Large$h$}};
    \draw (0,-3.5) node{\textcolor{Orange}{\LARGE$\cdots$}};
    \draw (0,-7) -- (0,-6);
    \draw[dotted] (0,-6) -- (0,-5);
    \draw (0,-6) to[out=0,in=270] (1.5,-5);
    \draw (0,-6) to[out=180,in=270] (-1.5,-5);
    \draw[dotted] (-2,-4.5) to[out=270,in=180] (-1.5,-5) to[out=0,in=270] (-1,-4.5);
    \draw[dotted] (2,-4.5) to[out=270,in=0] (1.5,-5) to[out=180,in=270] (1,-4.5);
    \draw (-1.75,-6.5) node{\textcolor{Orange}{$\cdots$}};
    \draw (1.75,-6.5) node{\textcolor{Orange}{$\cdots$}};
    \end{scope}
    \draw (-3.5,-21) to[out=270,in=270,looseness=1] (12.25,-22) -- (12.25,0) to[out=90,in=90,looseness=1] (-3.5,0);
    \draw (0,-21) to[out=270,in=270,looseness=1] (10.5,-21.5) -- (10.5,0) to[out=90,in=90,looseness=1] (0,0);
    \draw (3.5,-21) to[out=270,in=270,looseness=1] (8.75,-21) -- (8.75,0) to[out=90,in=90,looseness=1] (3.5,0);
    \draw[dashed] (7,-10.5) to[out=90,in=0,looseness=.75] (3.5,0) -- (-5,0);
    \draw[dashed,ForestGreen] (7,-10.5) to[out=90,in=0,looseness=1] (3.5,-7) -- (-5,-7);
    \draw[fill=gray] (7,-10.5) circle (.25);
\end{tikzpicture}
\caption{A closed strand diagram obtained from $f = g^h$ and, in \textcolor{ForestGreen}{green}, a different position of the base line.}
\label{fig_conj}
\end{figure}
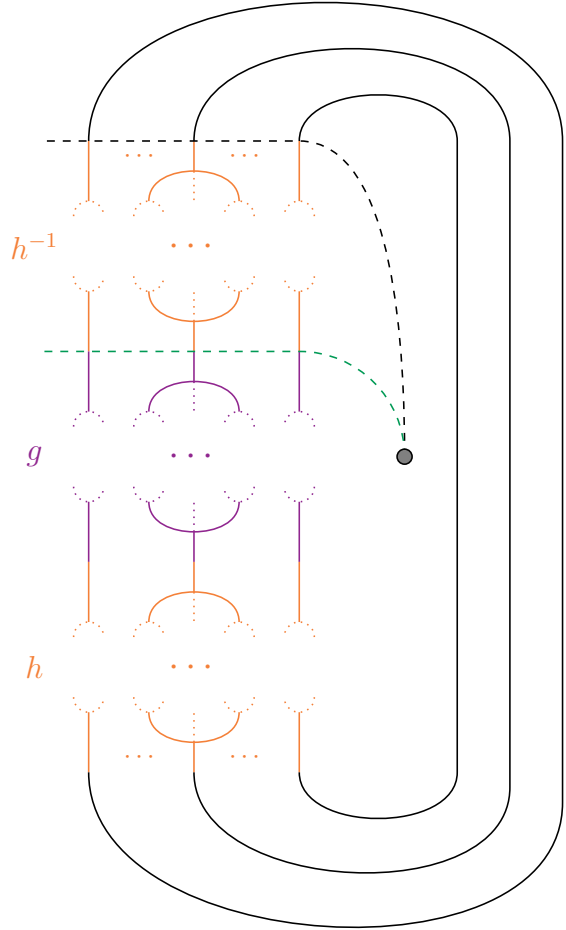

The next result is the bridge that links reductions to conjugacy.
Note that it does not require the hypothesis of reduction-confluence. 

\begin{proposition}
\label{PROP conjugator}
Given $f, g \in G_\mathcal{X}$, if up to similarity the closures of any of their strand diagrams can be rewritten as a common reduced closed diagram using reductions and inverses of reductions, then $f$ and $g$ are conjugate in $G_\mathcal{X}$.
\end{proposition}

\begin{proof}
First, observe that it does not matter which strand diagrams we choose to represent $f$ and $g$.
Indeed, two distinct strand diagrams for the same elements would only differ by a sequence of Type 1 and Type 2 reductions, which can also be performed on the closure of such diagrams.

We now note that, if two closed strand diagrams $\eta$ and $\zeta$ differ by a transformation, then the respective ``open'' strand diagrams $o(\eta)$ and $o(\zeta)$ are conjugate by a specific element of the replacement groupoid which depends on the transformation, as explained in the points below.

\begin{itemize}
\phantomsection\label{TXT transformations to conjugacy}
    \item If the transformation is a shift, then the open strand diagrams $o(\eta)$ and $o(\zeta)$ differ from conjugating by a split or a merge diagram.
    \item If the transformation is a permutation, then the open strand diagrams $o(\eta)$ and $o(\zeta)$ differ from conjugating by a permutation diagram.
    \item If the transformation is a Type 1 or 2 reduction, then the open strand diagrams $o(\eta)$ and $o(\zeta)$ represent the same rearrangement.
    \item If the transformation is a Type 3 reduction, then the open strand diagrams $o(\eta)$ and $o(\zeta)$ differ from conjugating by an amount of split diagrams equal to the winding number of the loops involved in the reduction.
\end{itemize}

Suppose that two closed diagrams $\eta$ and $\zeta$ obtained from $f$ and $g$ have the same reduced closed diagram $\rho$.
Then there exist sequences of transformations from $\eta$ to $\rho$ and from $\zeta$ to $\rho$.
Since each transformation corresponds to a (possibly trivial) conjugation of the open diagrams, it follows that both $o(\eta)^{h_\eta} = o(\rho)$ and $o(\zeta)^{h_\eta} = o(\rho)$, where $h_\eta$  and $h_\zeta$ are compositions of split diagrams, merge diagrams and permutation diagrams, which are elements of the replacement groupoid.
Thus, $o(\eta) = o(\zeta)^{h_\zeta \circ h_\eta^{-1}}$.

Finally, we need to show that $h_\zeta \circ h_\eta^{-1}$ is a strand diagram that represents an element of $G_\mathcal{X}$.
The composition $o(\zeta)^{h_\zeta \circ h_\eta}$ exists and results in $o(\eta)$, so $h_\zeta \circ h_\eta^{-1}$ is an element of the replacement groupoid whose source and sink graphs represent the base graph of $\mathcal{X}$, i.e., $h_\zeta \circ h_\eta^{-1}$ is an $X$-strand diagram (as defined at \cpageref{TXT X-SDs}).
As discussed at \cpageref{TXT X-SDs}, these are exactly the elements of $\mathcal{X}$, so we are done.
\end{proof}

Together these two Propositions yield the following result.

\begin{proposition}
\label{PROP conjugate}
Suppose that $\mathcal{X} = (X_0, R, \mathrm{C})$ is an expanding replacement system whose set of replacement rules $(R, \mathrm{C})$ is reduction-confluent.
Two elements $f, g \in G_\mathcal{X}$ are conjugate in $G_\mathcal{X}$ if and only if their reduced closed diagrams are similar.
\end{proposition}

\subsection{The Algorithm for Reduction-confluent Replacement Rules}
\label{SUB algorithm}

Using the results from the last Subsection, under the assumption of reduction-confluence we can algorithmically check whether two rearrangements $f$ and $g$ are conjugate in the rearrangement group:
it suffices to consider the closed strand diagrams $\eta$ and $\zeta$ for $f$ and $g$, respectively, compute their reduced diagrams $\eta^*$ and $\zeta^*$ and finally check whether they are the same up to similarity.
In more detail, the algorithm is the following.

\begin{enumerate}
    \setcounter{enumi}{-1}
    \item Consider the closed strand diagrams $\eta$ and $\zeta$ obtained from $f$ and $g$.
    \item Compute their reduced closed diagrams $\eta^*$ and $\zeta^*$ in the following way.
    \begin{itemize}
        \item Check whether the diagram contains a Type 1 or 2 reduction up to permutations and shifts (which can be done by ``forgetting'' about the base line and the base points). If it does, perform the reduction, possibly performing a permutation of the base points or moving the base line with a shift if necessary to ``unlock'' the reduction. Continue doing this until no Type 1 and 2 reductions are possible. This procedure must eventually stop, since Type 1 and 2 reductions decrease the amount of strands.
        \item Check whether the diagram contains a Type 3 reduction, up to permutations. If it does, perform the reduction, possibly performing a permutation of the base points if necessary to ``unlock'' the reduction.
    \end{itemize}
    \item In order to check similarity, consider the graphs obtained by ``forgetting'' about the labels, the base line and the base points of the two diagrams and proceed as follows.
    \begin{itemize}
        \item Check whether the graphs are isomorphic (this terminates because the graphs are finite). If they are not isomorphic, then the strand diagrams cannot be similar. Otherwise, execute the next steps multiple times, separately identifying the two graphs under each isomorphism.
        \item Check whether the two positions of the base lines are the same up to similarity.
        This is equivalent to asking whether they describe cohomologous cocycles by \cref{RMK cohomology}, so it can be efficiently computed.
        Explicit examples of these computations for Thompson's group $V$ are contained in the dissertation \cite{Algorithm}, whose content can be adapted to suit the setting of rearrangement groups. If the answer is negative, then the diagrams are not similar.
        \item Once established that the positions of the base lines are the same up to similarity, it only remains to check that labels are the same up to renaming. To do so, find the same position of the base line by computing, step by step and in parallel, all possible similarities of the second diagram $\zeta^*$ until the position of the base line matches the one in $\eta^*$ (we know that this process terminates because in the previous step we have established that the positions of the base lines are the same up to similarity).
        Once the same position of the base lines has been found, check whether labels are the same up to renaming and up to the finitely many permutations of stable symbols (see \cref{COR labels}).
    \end{itemize}
\end{enumerate}

Note that the order in which the reductions are performed in step (1) does not matter, as there is a unique reduced similarity class thanks to \cref{LEM reduced CSD}

Then we have the following:

\begin{theorem}
\label{THM conjugate}
Suppose that $\mathcal{X} = (X_0, R, \mathrm{C})$ is an expanding replacement system whose set of replacement rules $(R, \mathrm{C})$ is reduction-confluent.
Then the conjugacy problem of $G_\mathcal{X}$ is solvable.
\end{theorem}

Using this method, not only we can say whether two given rearrangements are conjugate, but when they are we can also explicitly compute a conjugating element, as seen in the proof of \cref{PROP conjugator}.
Thus, we also have:

\begin{corollary}
Suppose that $\mathcal{X} = (X_0, R, \mathrm{C})$ is an expanding replacement system whose set of replacement rules $(R, \mathrm{C})$ is reduction-confluent.
Then the conjugacy search problem of $G_\mathcal{X}$ is solvable.
\end{corollary}

As a consequence of this Theorem, since the Higman-Thompson groups, the Houghton groups, the groups $QV$, $\tilde{Q}V$, $QT$ and $QF$, the Basilica, the Rabbit, the Vicsek and the Bubble Bath replacement rules are reduction-confluent (as noted at \cpageref{TXT Basilica Vicsek BubbleBath Houghton QV}), the conjugacy problem and the conjugacy search problem in their rearrangement groups can be solved using closed strand diagrams.

\begin{remark}
    This algorithm is arguably not the most efficient that it can be.
    For example, an algorithm that checks whether the two base lines describe cohomologous cycles could also provide an explicit sequence of shifts that bring the first base line to the second, allowing to partially skip the subsequent step of the algorithm.
    There might be more places where the algorithm may be optimized, which might be of interest for possible applications to group-based cryptography (see for example Section 3 of \cite{crypto}), but a full analysis of the efficiency of this method is beyond the scope of this paper.
\end{remark}

\subsection{Non Reduction-confluent Replacement Rules}
\label{SUB non conf}

\cref{THM conjugate} shows that reduction-confluence of the replacement system (\cref{DEF red-conf}) is a sufficient condition for solving the conjugacy problem with closed strand diagrams.
However, we can still adapt this method to solve the conjugacy problem in certain rearrangement groups that do not have this property.

For example, at \cpageref{TXT airplane non confluent} we showed that the Airplane replacement system $\mathcal{A}$ is not reduction-confluent.
However, this issue is only caused by the fact that, whenever we have a subgraph that is isomorphic to the base graph of $\mathcal{A}$ and that features at least one vertex of degree 1, this subgraph can be reduced in two different manners, as shown in \cref{fig_TA_reduced}.
It is intuitive that such a small issue can be overcome, and indeed one can modify the reduction system (defined in \cref{SUB reduction systems}) by adding the reduction described in \cref{fig_Airplane_graph_reduction}.
This additional rule means that a subgraph isomorphic to the one on the left in the figure, if no edge is adjacent to the vertex x, can be replaced with the one on the right, maintaining the adjacencies of the vertex v.
This can be stated more precisely and more generally using DPO graph rewriting systems, as will be discussed in the following \cref{SUB DPO}, but the idea should be clear without having to introduce that machinery.

It is not hard to prove that this newly built reduction system (comprised of the two original reductions alongside the new one) is locally confluent and terminating, thus by the usual argument we find that each connected component has a unique reduced graph.
The reason why this translates to conjugacy is that the newly added reduction of graphs is precisely the same as a red expansion followed by a blue reduction, thus we must allow the new special kind of Type 3 reduction that, from looping strands as in the bottom left of \cref{fig_TA_reduced}, produces a looping strand as in the bottom right of that same figure.
Then one can prove an analogous to \cref{PROP conjugator} by adding the following rule to the list at \cpageref{TXT transformations to conjugacy}:
\begin{itemize}
    \item If the transformation is a Type 3 reduction of the newly added kind, then the open strand diagrams differ from conjugation by the diagram in \cref{fig_Airplane_conj} (a number of these diagrams equal to the winding number of the loops).
\end{itemize}
It is easy to check that the new rewriting system is confluent, by the same reasoning used to prove \cref{LEM reduced CSD} along with the fact that the reduction system comprised of the two original reductions alongside the new one is confluent.

\begin{figure}\centering
\begin{tikzpicture}[scale=1]
    \draw[blue] (0,0) node[black,circle,fill,inner sep=1.25]{} node[black,above]{v} -- (2,0) node[black,anchor=south east]{x};
    \draw[red,->-=.5] (2,0) to[out=270,in=270,looseness=1.7] (3,0) to[out=90,in=90,looseness=1.7] (2,0) node[black,circle,fill,inner sep=1.25]{};
    \draw[thick, -stealth] (3.7,0) -- (4.3,0);
    \draw[blue] (5,0) node[black,circle,fill,inner sep=1.25]{} node[black,above]{v} -- (7,0) node[black,circle,fill,inner sep=1.25]{} node[black,above]{w};
\end{tikzpicture}
\caption{The additional graph reduction that allows to solve the conjugacy problem in the Airplane rearrangement group.}
\label{fig_Airplane_graph_reduction}
\end{figure}
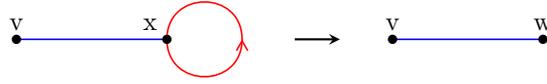

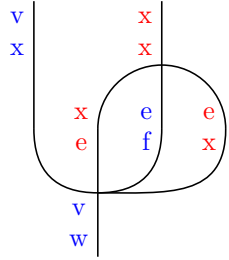
\begin{figure}\centering
\begin{tikzpicture}[scale=.85]
    \draw (0,0) -- node[left,blue,align=center,looseness=1.2]{v\\x} (0,-1) -- (0,-2);
    \draw (2,0) -- node[left,red,align=center,looseness=1.2]{x\\x} (2,-1);
    \draw (2,-1) to[out=180,in=90] (1,-2) node[left,red,align=center]{x\\e};
    \draw (2,-1) to (2,-2) node[left,blue,align=center]{e\\f};
    \draw (2,-1) to[out=0,in=90] (3,-2) node[left,red,align=center]{e\\x};
    \draw (0,-2) to[out=270,in=180,looseness=1.2] (1,-3);
    \draw (1,-2) to (1,-3);
    \draw (2,-2) to[out=270,in=0,looseness=1.2] (1,-3);
    \draw (3,-2) to[out=270,in=0,looseness=1.25] (1,-3);
    \draw (1,-3) -- node[left,blue,align=center]{v\\w} (1,-4);
\end{tikzpicture}
\caption{The conjugator corresponding to a Type 3 reduction associated with the graph reduction shown in \cref{fig_Airplane_graph_reduction}.}
\label{fig_Airplane_conj}
\end{figure}

Thus, collectively we have the following result.

\begin{theorem}\label{THM known groups}
Strand diagrams can be used to solve the conjugacy problem and the conjugacy search problem in the Higman-Thompson groups, the Airplane rearrangement group $T_A$, the Rabbit rearrangement groups (including the Basilica rearrangement group $T_B$) the Vicsek and Bubble Bath rearrangement groups, the Houghton groups $H_n$, the groups $QV$, $\tilde{Q}V$, $QT$, $QF$ and those topological full groups of (one-sided) subshifts of finite type whose replacement system is reduction-confluent.
\end{theorem}

Instead, this does not seem to immediately apply to replacement rules such as the one portrayed in \cref{fig replacement Bad}, as adding the obvious reduction depicted in \cref{fig replacement Bad_conf} creates new non-confluent paths of reductions.
It may still be possible that adding a large (but finite) amount of new reductions might solve this.
It should also be noted that the rearrangement group for this specific replacement system is a diagram group, thus the conjugacy problem is solvable if and only if the semigroup $\langle r,b \mid r = r b b r, b = b r r b \rangle$ ($r$ and $b$ standing for red and blue, respectively) has solvable word problem, by a result in \cite{guba1997diagram}.

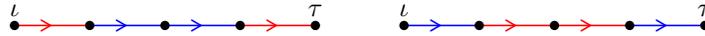
\begin{figure}\centering
\begin{subfigure}{.4\textwidth}\centering
    \begin{tikzpicture}[scale=1]
        \draw[->-=.5,red] (0,0) node[black,circle,fill,inner sep=1.25]{} node[above,black]{$\iota$} -- (1,0);
        \draw[->-=.5,blue] (1,0) node[black,circle,fill,inner sep=1.25]{} -- (2,0);
        \draw[->-=.5,blue] (2,0) node[black,circle,fill,inner sep=1.25]{} -- (3,0);
        \draw[->-=.5,red] (3,0) node[black,circle,fill,inner sep=1.25]{} -- (4,0) node[black,circle,fill,inner sep=1.25]{} node[above,black]{$\tau$};
    \end{tikzpicture}
\end{subfigure}
\begin{subfigure}{.4\textwidth}\centering
    \begin{tikzpicture}[scale=1]
        \draw[->-=.5,blue] (0,0) node[black,circle,fill,inner sep=1.25]{} node[above,black]{$\iota$} -- (1,0);
        \draw[->-=.5,red] (1,0) node[black,circle,fill,inner sep=1.25]{} -- (2,0);
        \draw[->-=.5,red] (2,0) node[black,circle,fill,inner sep=1.25]{} -- (3,0);
        \draw[->-=.5,blue] (3,0) node[black,circle,fill,inner sep=1.25]{} -- (4,0) node[black,circle,fill,inner sep=1.25]{} node[above,black]{$\tau$};
    \end{tikzpicture}
\end{subfigure}
\caption{The \textcolor{red}{red} and \textcolor{blue}{blue} replacement graphs for a non confluent-reduction set of replacement rules where closed strand diagrams do not seem to easily solve the conjugacy problem.}
\label{fig replacement Bad}
\end{figure}

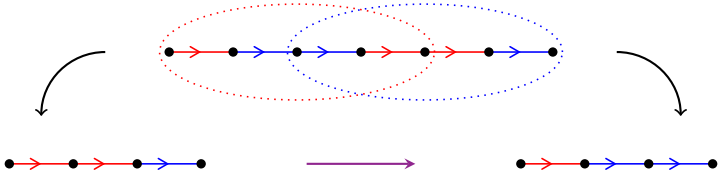
\begin{figure}\centering
\begin{tikzpicture}[scale=.85]
    \begin{scope}[xshift=-3cm]
    \draw[->-=.5,red] (0,0) node[black,circle,fill,inner sep=1.25]{} -- (1,0);
    \draw[->-=.5,blue] (1,0) node[black,circle,fill,inner sep=1.25]{} -- (2,0);
    \draw[->-=.5,blue] (2,0) node[black,circle,fill,inner sep=1.25]{} -- (3,0);
    \draw[->-=.5,red] (3,0) node[black,circle,fill,inner sep=1.25]{} -- (4,0);
    \draw[->-=.5,red] (4,0) node[black,circle,fill,inner sep=1.25]{} -- (5,0);
    \draw[->-=.5,blue] (5,0) node[black,circle,fill,inner sep=1.25]{} -- (6,0) node[black,circle,fill,inner sep=1.25]{};
    \draw[red,dotted] (2,0) ellipse (2.15cm and .75cm);
    \draw[thick,-to] (-1,0) to[out=180,in=90] (-2,-1);
    \draw[blue,dotted] (4,0) ellipse (2.15cm and .75cm);
    \draw[thick,-to] (7,0) to[out=0,in=90] (8,-1);
    \end{scope}
    \begin{scope}[xshift=-5.5cm,yshift=-1.75cm]
    \draw[->-=.5,red] (0,0) node[black,circle,fill,inner sep=1.25]{} -- (1,0);
    \draw[->-=.5,red] (1,0) node[black,circle,fill,inner sep=1.25]{} -- (2,0);
    \draw[->-=.5,blue] (2,0) node[black,circle,fill,inner sep=1.25]{} -- (3,0) node[black,circle,fill,inner sep=1.25]{};
    \end{scope}
    \begin{scope}[xshift=2.5cm,yshift=-1.75cm]
    \draw[->-=.5,red] (0,0) node[black,circle,fill,inner sep=1.25]{} -- (1,0);
    \draw[->-=.5,blue] (1,0) node[black,circle,fill,inner sep=1.25]{} -- (2,0);
    \draw[->-=.5,blue] (2,0) node[black,circle,fill,inner sep=1.25]{} -- (3,0) node[black,circle,fill,inner sep=1.25]{};
    \end{scope}
    \draw[Plum,thick,-stealth] (-.85,-1.75) -- (.85,-1.75);
\end{tikzpicture}
\caption{Adding \textcolor{Plum}{this reduction} to the reduction system associated to the replacement rules of \cref{fig replacement Bad} does not make the reduction system confluent, differently from how the new reduction of \cref{fig_Airplane_graph_reduction} worked for the Airplane reduction system.}
\label{fig replacement Bad_conf}
\end{figure}

\subsection{Graph Rewriting Systems and Confluence}
\label{SUB DPO}

Replacement rules (\cref{DEF replacement}), reduction systems (\cref{SUB reduction systems}) and the modified reduction system discussed for the Airplane from the previous \cref{SUB non conf} are all instances of \textbf{Graph Rewriting Systems}, which is a specific instance of abstract rewriting systems whose objects are graphs.
The modern DPO (double pushout) approach to graph rewriting systems was first introduced in \cite{4569741}, but the precise definition may vary, as explained and investigated in \cite{habel_müller_plump_2001}.

An application of this general theory is the Critical Pair Lemma, which is Lemma 10 from \cite{Plump2005}.
It states that it suffices to check specific pairs of reductions (called \textit{critical pairs}) in order to establish that a graph rewriting system is locally confluent (which implies being confluent if it is terminating).
Our proof that the Basilica replacement system is reduction-confluent (\cpageref{TXT Basilica Vicsek BubbleBath Houghton QV}) can actually be seen as an instance of this Lemma.
However, the Critical Pair Lemma does not provide a necessary condition, and in general the problem of deciding whether a graph rewriting system is confluent is undecidable (Theorem 5 of \cite{Plump2005}).
Nevertheless, since reductions of expanding replacement systems are a very specific and arguably quite simple family of graph rewriting systems, it may very well be that checking for confluence in this context is instead decidable.
Checking whether this is the case or not is an interesting question that we will not tackle here.

Finally, what we did in the previous \cref{SUB non conf} to solve the conjugacy problem for the Airplane replacement system can be called a \textit{confluent completion} of the reduction system: we have added new reductions that make the reduction system confluent, while preserving termination and the equivalence generated by the rewriting.
In the different context of \textit{term rewriting systems}, the Knuth-Bendix completion algorithm (\cite{Knuth1983}) does exactly what we just described.
However, such an algorithm in the larger context of graph rewriting systems has not been developed yet, and is currently of interest to the community of graph rewriting systems (as stated in Section 5 of \cite{Plump2005}).
So it is possible that some future application of the theory of graph rewriting systems will apply to reduction systems of replacement systems, possibly providing a wider sufficient condition for the feasibility of the method presented here to solve the conjugacy problem in rearrangement groups.
As far as the author knows, if a completion algorithm could always be produced for the reduction system of any replacement system, then the conjugacy (search) problem in every rearrangement group would be solved with this approach based on building confluent completions that allow algorithmic comparisons of closed strand diagrams.

\printbibliography[heading=bibintoc]

\end{document}